\documentclass{myclass}

\usepackage{comment}
\usepackage{mathabx,epsfig}
\usetikzlibrary{calc}
\usetikzlibrary{decorations.pathmorphing}
\usetikzlibrary{positioning}
\usetikzlibrary{shapes.multipart}

\bibliography{bibliography}

\newcounter{thmCounter}					
\numberwithin{thmCounter}{section}		

\theoremstyle{plain}
\newtheorem{myprop}[thmCounter]{Proposition}
\newtheorem{mylem}[thmCounter]{Lemma}
\newtheorem{mythm}[thmCounter]{Theorem}
\newtheorem{mycor}[thmCounter]{Corollary}

\theoremstyle{definition}
\newtheorem{mydef}[thmCounter]{Definition}

\theoremstyle{remark}
\newtheorem{myrem}[thmCounter]{Remark}
\newtheorem{myexa}[thmCounter]{Example}

\defcals{HMDXORLGJTQ}
\deffraks{HR}
\defbbs{NRCPZQ}
\defops{dih,div,Fix,PGL,Ker}

\let\hat\relax
\newcommand{\hat}{\widehat}
\let\tilde\relax
\newcommand{\tilde}{\widetilde}

\let\Im\relax
\let\Re\relax
\DeclareMathOperator{\Im}{Im}				
\DeclareMathOperator{\Re}{Re}				
\DeclareMathOperator{\valence}{val}			
\DeclareMathOperator{\Div}{Div}				
\DeclareMathOperator{\Rat}{Rat}				
\DeclareMathOperator{\Res}{Res}				
\DeclareMathOperator{\tropDiv}{Div}			
\DeclareMathOperator{\supp}{supp}			
\DeclareMathOperator{\ord}{ord}				
\DeclareMathOperator{\Spec}{Spec}			
\DeclareMathOperator{\valuation}{val}		
\DeclareMathOperator{\tgt}{tgt}				
\DeclareMathOperator{\src}{src}				

\newcommand{\an}{{\mathrm{an}}}				
\newcommand{\trop}{{\operatorname{trop}}}	
\newcommand{\tropHodge}{\bP \Omega^k M_g^\trop}	
\newcommand{\tropMSD}{\bP \Xi^k M_g^\trop}	
\newcommand{\classicDiv}{\mathcal{D}iv}		
\newcommand{\compact}[1]{\overline{#1}}		
\newcommand{\bs}{\mathbf{s}}                
\newcommand{\hor}{\mathrm{hor}}             
\newcommand{\id}{\mathrm{id}}               
\newcommand{\prim}{\mathrm{prim}}           
\newcommand{\PD}{\mathrm{PD}}				
\newcommand{\dR}{\mathrm{dR}}				
\newcommand{\red}{\mathrm{red}}				
\newcommand\MSD[4][]{\Xi^{#1} \overline \cM_{{#2},{#3}}({#4})}
\newcommand\PMSD[4][]{\bP \Xi^{#1} \overline \cM_{{#2},{#3}}({#4})}

\newcommand{\acts}{\mathrel{\reflectbox{$\righttoleftarrow$}}}

\title{Realizability of tropical pluri-canonical divisors}
\date{}
\subjclass{14H15, 14H10, 32G15, 14T20}
\keywords{Tropical geometry, moduli spaces, Hodge bundle, twisted differentials, canonical divisors, Berkovich spaces}

\author{Felix Röhrle}
\address{Fachbereich Mathematik, Universit\"at T\"ubingen, Auf der Morgenstelle 10, 72076 T\"ubingen, Germany}
\email{\href{mailto:roehrle@math.uni-tuebingen.de}{roehrle@math.uni-tuebingen.de}}

\author{Johannes Schwab}
\address{Institut f\"ur Mathematik, Goethe Universit\"at Frankfurt, 60325 Frankfurt am Main, Germany}
\email{\href{mailto:schwab@math.uni-frankfurt.de}{schwab@math.uni-frankfurt.de}}

\begin{document}


\begin{abstract}
	Consider a pair consisting of an abstract tropical curve and an effective divisor from the linear system associated to $k$ times the canonical divisor for $k \in \bZ_{\geq 1}$. In this article we give a purely combinatorial criterion to determine if such a pair arises as the tropicalization of a pair consisting of a smooth algebraic curve over a non-Archimedean field with algebraically closed residue field of characteristic~0 together with an effective pluri-canonical divisor. To do so, we introduce tropical normalized covers as special instances of cyclic tropical Hurwitz covers and reduce the realizability problem for pluri-canonical divisors to the realizability problem for normalized covers. Our main result generalizes the work of Möller-Ulirsch-Werner on realizability of tropical canonical divisors and incorporates the recent progress on compactifications of strata of $k$-differentials in the work of Bainbridge-Chen-Gendron-Grushevsky-Möller.
\end{abstract}

	\maketitle
	\tableofcontents


\section{Introduction}
\label{sec:introduction}

The close analogy between Riemann surfaces and graphs was first described in \cite{Mik05}.
Since then many definitions in tropical geometry have been modeled with the aim that tropicalization of algebro-geometric objects produces the corresponding tropical objects. The realizability problem then asks whether a given instance of the tropical notion does indeed arise in this way. 
For example consider this question for curves with effective divisors. Given an abstract tropical curve $\Gamma$ (i.e. a vertex weighted metric graph) with an effective divisor $D$ (i.e. a nonnegative linear combination of points on $\Gamma$), the realizability problem asks if there exists a smooth proper algebraic curve $X$ with effective divisor $\tilde D$ of the same degree and rank as $D$ such that the tropicalization of $(X, \tilde D)$ is $(\Gamma, D)$ (see Section \ref{subsec:tropicalization} below for details on the tropicalization of curves with divisor). 
This question is very difficult in general. In fact, \cite{Cartwright} shows that it satisfies a version of Murphy's law that makes a general solution seem unlikely. In this article we restrict our attention to the special case of effective pluri-canonical divisors and give a complete characterization of those tropical objects that are realizable over an algebraically closed base field of characteristic~0.

\subsection{The tropical $k$-Hodge bundle}

Let $g \geq 2$ and $k \geq 1$ be integers. In algebraic geometry the $k$-Hodge bundle $\Omega^k \cM_{g}$ is a moduli space parametrizing pairs $(X, \eta)$ consisting of a smooth curve $X$ of genus $g$ and a $k$-differential $\eta$, i.e.\ a global section of the $k$-th tensor power of the canonical bundle on $X$. 
We start our exposition in Section \ref{sec:tropical} with a review of basic definitions for tropical curves, divisors, and linear equivalence. In Section \ref{subsec:tropical_k_Hodge_bundle} we then construct a tropical counterpart of the projectivized moduli space $\bP \Omega^k \cM_g$. More precisely, we prove:

\begin{mythm} \label{intro:thm:trop_k_Hodge_bundle}
	There exists a generalized cone complex in the sense of \cite[Section~2.6]{ACP} which parametrizes pairs $([\Gamma], D)$ of isomorphism classes of abstract tropical curves $\Gamma$ of genus $g$ and effective divisor $D \in \Div(\Gamma)$ linearly equivalent to $k$ times the canonical divisor $K_\Gamma$. We denote this space by $\tropHodge$ and call it the \emph{tropical $k$-Hodge bundle}. It is not equidimensional. The dimension of a maximal cone is $(3+2k)(g - 1)$.
\end{mythm}

\subsection{Tropicalizing $k$-differentials}
Tropicalization of curves with divisor has been described e.g.\ in \cite[Section~6.3]{BJ16} in the following way. Let $X$ be a smooth curve over a non-Archimedean field and let $D$ be an effective divisor on $X$. Let $\cX$ be the stable model of $X$. Define $\Gamma$ to be the dual graph of the nodal special fiber endowed with edge lengths obtained from the deformation parameters of the nodes. Furthermore, via the specialization map in the sense of Baker \cite[Section~2C]{Bak08} the divisor $D$ gives rise to a divisor on $\Gamma$. 
In \cite{MUW17_published} the authors gave a description of this procedure as a continuous map between moduli spaces. By restricting this general construction to effective pluri-canonical divisors we obtain a continuous tropicalization map $\trop_{\Omega^k} : \bP \Omega^k \cM_{g}^\an \to \tropHodge$ in Section \ref{subsec:tropicalization}. Throughout, $()^\an$ denotes analytification in the sense of \cite{Ber90}. The dimension of $\bP \Omega^k \cM_g$ is $(2+2k)(g-1)-1$ if $k > 1$ and $4g -4$ if $k = 1$ by \cite[Theorem~1.1]{BCGGM19}. Comparing this to the dimension of $\tropHodge$ obtained in Theorem~\ref{intro:thm:trop_k_Hodge_bundle} we see that $\trop_{\Omega^k}$ cannot be surjective. The realizability problem amounts to describing the image of $\trop_{\Omega^k}$, which is called the \emph{realizability locus}, as a subset of $\tropHodge$.

\subsection{Reduction to the case of normalized covers}
Our approach to the realizability problem crucially relies on the techniques developed in \cite{BCGGM19}. 
Recall that the authors of \cite{BCGGM19} canonically associate to any smooth curve with $k$-differential an admissible, normalized, cyclic, potentially ramified and disconnected cover $\pi : \hat X \to X$ with abelian differential $\omega$ on $\hat X$ and a deck transformation $\tau : \hat X \to \hat X$ such that $\omega^k = \pi^\ast \eta$ and $\tau^\ast \omega = \zeta \omega$ for a primitive $k$-th root of unity $\zeta$. 
The $k$-Hodge bundle admits a natural stratification by so-called \emph{types} $\mu = (m_1, \ldots, m_n) \in \bZ^n$ such that the sum of the $m_i$ is $k(2g-2)$.
The moduli space of multi-scale $k$-differentials $\bP \Xi^{k} \overline \cM_{g,n}{(\mu)}$ was introduced in \cite{CMZ19}. It is a compactification of the projectivized strata $\bP \Omega^k \cM_g(\mu)$ of the $k$-Hodge bundle and parametrizes normalized covers in its interior. In order to make use of this theory, we define a \emph{tropical normalized cover} in Definition~\ref{def:tropical_normalized_cover} to be a $k$-cyclic tropical Hurwitz cover $\hat \Gamma \to \Gamma$ such that the legs of $\hat\Gamma$ and $\Gamma$ encode a \mbox{(pluri-)}canonical divisor and additionally we require a deck transformation on $\hat \Gamma$ as well as compatibility conditions mimicking the above. Our notion of $k$-cyclic tropical Hurwitz cover is a special case of an unramified $\bZ/k\bZ$-cover in the sense of \cite{LUZ24} which in turn builds on the notion of tropical Hurwitz covers as defined in \cite{CMR16}.
Again we introduce a tropical moduli space in analogy to the algebro-geometric setting.

\begin{mythm} \label{intro:thm:tropMSD}
	There is a moduli space of tropical normalized covers, denoted $\tropMSD$, together with a natural forgetful map $\tropMSD \to \tropHodge$. It carries the structure of a generalized cone complex. The dimension of a maximal cone is $(3 + 2k)(g-1)$. Furthermore, there is a well-defined, continuous, closed, and proper tropicalization map $\trop_{\Xi^k} : \bP \Omega^{k} \cM_{g}{(1, \ldots, 1)}^\an \to \tropMSD$ such that the following diagram commutes
	\begin{equation} \label{eq:reduction_triangle}
		\begin{tikzcd}[column sep = large]
			\bP \Omega^{k} \cM_{g}{(1, \ldots, 1)}^\an \arrow[r, "\trop_{\Xi^k}"] \arrow[dr, "\trop_{\Omega^k}"'] & \tropMSD \arrow[d] \\
			& \tropHodge \ .
		\end{tikzcd}
	\end{equation} 
\end{mythm}

The forgetful map to $\tropHodge$ is not injective but finite. This means that a tropical normalized cover contains strictly more information than a pair $(\Gamma, D)$ of a tropical curve with effective pluri-canonical divisor, but given $(\Gamma, D)$ there are only finitely many tropical normalized covers above it. Moreover, finding these covers can in principal be automated and carried out on a computer. 
In this sense the following corollary effectively reduces our original realizability problem to the realizability of tropical normalized covers (see Corollary~\ref{cor:reduction_to_covers} for the precise statement).
It follows directly from the commutative triangle in \eqref{eq:reduction_triangle}.
\begin{mycor} \label{intro:cor:reduction}
	A tropical curve $\Gamma$ with effective pluri-canonical divisor $D = kK_\Gamma + (f)$ is realizable if and only if there exists a realizable tropical normalized cover $\pi : \hat{\Gamma} \to \Gamma$ such that the legs of $\Gamma$ encode $D$.
\end{mycor}  

\subsection{Realizability of tropical normalized covers}
In Section \ref{sec:realizability_locus} we solve the realizability problem for tropical normalized covers $\pi : \hat\Gamma \to \Gamma$. Via Corollary~\ref{intro:cor:reduction} this may be viewed as an auxiliary theorem towards a solution of our original realizability problem. However, since the notion of tropical normalized covers provides a better (but still inadequate) approximation of the boundary combinatorics of the multi-scale $k$-differential compactification of the algebraic $k$-Hodge bundle, we believe that the following results are of independent interest as well. 

To solve the realizability problem for tropical normalized covers we use similar ideas as in \cite{MUW17_published}. This means that we proceed in two steps.
\begin{enumerate}
	\item For every vertex $v$ in $\Gamma$ we realize $\pi|_{\pi^{-1}(\{v\})}$ with a normalized cover of smooth curves with meromorphic differentials. 
	\item We glue these parts to obtain a normalized cover of nodal curves which lies in the boundary of $\bP \Xi^{k} \overline \cM_{g,n}{(1, \ldots, 1)}$ and smoothen these curves. 
\end{enumerate}
Observe at this point that for a tropical curve with pluri-canonical divisor $D = kK_\Gamma + (f)$ the zero and pole orders of any of the realizations in step~(1) are already determined. More precisely, the rational function $f$ gives rise to a canonical \emph{enhanced level graph structure} on $\Gamma$ (see Definition \ref{def:enhanced_level_graph} and Lemma \ref{lem:construction_of_k-enhancement} for details). Consequently, we only need to specify ($k$-)residues to proceed. Both steps from above impose restrictions on the possible choices. For step (1) these are given by \cite{GT21}, \cite{GT_quadratic} and \cite{GT_kdiff} and lead us to the notions of \emph{illegal vertex} (such a vertex is never realizable) and \emph{inconvenient vertex} (here special care in choosing residues has to be taken). Step (2) is only feasible if the \emph{global residue condition} (see Definition~\ref{def:GRC}) as well as the above mentioned compatibilities with $\pi$ and $\tau$ are respected. We will ensure this by assigning residues along $\tau$-orbits of simple closed cycles in $\hat \Gamma$. In contrast to the case $k = 1$ that was treated in \cite{MUW17_published} not any cycle is sufficient for this purpose. Rather we have to ask for each inconvenient vertex for a corresponding \emph{admissible cycle} (Definition \ref{def:effective_cycle}) or an \emph{independent pair of cycles} (Definition~\ref{def:independent_pair}). Having introduced the necessary notation we state our main result in Theorem \ref{thm:main_theorem} which roughly says the following.

\begin{mythm} \label{intro:thm:main_theorem}
	Fix an algebraically closed base field of characteristic~0.
	Let $g \geq 2$ and fix an integer $k \geq 1$. Let $\pi : \hat \Gamma \to \Gamma$ be a tropical normalized cover and $D = kK_\Gamma + (f)$ be an effective pluri-canonical divisor on $\Gamma$. The pair is realizable if and only if the following conditions hold.
	\begin{enumerate}[(i)]
		\item There is no illegal vertex in $\pi$.
		\item For every edge $\hat e$ in $\hat \Gamma$ for which $f \circ \pi$ is constant there is an effective cycle in $\hat \Gamma$ through $\hat e$.
		\item For every inconvenient vertex $v$ in $\Gamma$ there is an admissible cycle in $\hat \Gamma$ through one of the preimages $\hat v$ or there is an independent pair of cycles.
	\end{enumerate}
\end{mythm}

In Figure~\ref{fig:strategy} we give an overview over the key objects in this article and their relation to each other. The diagram illustrates how Theorem~\ref{intro:thm:main_theorem} fills in the missing link in lifting a tropical normalized cover along $\trop_{\Xi^k}$ (right hand column of Figure~\ref{fig:strategy}).
Via Corollary~\ref{intro:cor:reduction}, Theorem~\ref{intro:thm:main_theorem} also provides a complete description of the locus of realizable curves in $\tropHodge$. 
Figure~\ref{fig:strategy} may also be understood as a compendium of the notation which we use throughout the article. 

\begin{figure}
	\centering
	\tikzset{
	mynode/.style={rectangle, draw, on grid, font = \small}, 
}
\def\arrowoffset{8}
\newcommand{\myindent}{\phantom{$\bullet$} \quad}
\begin{tikzpicture}[node distance=4.5 and 7.3, every text node part/.style={align=left}]
	\node[mynode](kdiff) {\textbf{$k$-differential} $(X, \eta) \in \bP \Omega^k \cM_g(\mu)$ \\
		consisting of \\
		$\bullet$ \quad smooth curve $X$ of genus $g$ \\
		$\bullet$ \quad section of $k$-th tensor power \\
		\myindent of canonical bundle $\eta$ \\
		\myindent with zero/pole orders $\mu = (m_1, \ldots, m_n)$};
	\node[mynode](mskdiff)[right=of kdiff] {Def. \ref{def:normalized_cover}: \textbf{normalized cover of} \\
		\textbf{$k$-differential} $\big[ \tau \acts \hat X \xrightarrow{\pi} X \big]$ with \\ 
		$\bullet$ \quad smooth curves $\hat X$ and $X$ \\
		$\bullet$ \quad ramified cover $\pi$ of degree $k$ \\
		$\bullet$ \quad deck transformation $\tau$ with $\tau^k = \id$ \\
		$\bullet$ \quad $k$-differential $\eta$ on $X$ \\ 
		$\bullet$ \quad abelian differential $\omega$ on $\hat X$ s.t. \\
		\myindent $\pi^\ast \eta = \omega^{\otimes k}$ and $\tau^\ast \omega = \zeta_k \omega$};
	
	\node[mynode](twistedkdiff)[below=of kdiff] {Def. \ref{def:twisted_k_differential}: \textbf{twisted $k$-differential} \\
		$\big(X, \{\eta_v\}_v \big)\in \bP \Omega^k \overline{\cM}_g(\mu)$ consisting of \\
		$\bullet$ \quad nodal curve $X = \bigcup X_v$ \\
		$\bullet$ \quad $k$-differential for each irreducible \\
		\myindent component $\{\eta_v\}_v$ \\
		$\bullet$ \quad level function $\ell : V \to \bZ_{\leq 0}$};
	\node[mynode](twistedmskdiff)[below=of mskdiff] {Def. \ref{def:multi-scale_k-diff}: \textbf{multi-scale $k$-differential} \\
		$\big[ \tau \acts \hat X \xrightarrow{\pi} X \big] \in \bP \Xi^k \overline{\cM}_g(\mu)$ with \\
		$\bullet$ \quad nodal curves $\hat X$ and $X$ \\
		$\bullet$ \quad twisted $k$-differential $\{\eta_v\}_v$ on $X$ \\
		$\bullet$ \quad twisted differential $\{\omega_v\}_v$ on $\hat X$ \\
		$\bullet$ \quad level function $\ell : V \to \bZ_{\leq 0}$
	};
	
	\node[mynode](enhancedlg)[below = of twistedkdiff] {Def. \ref{def:enhanced_level_graph}: \textbf{$k$-enhanced level graph $G^+$} \\
		consisting of \\
		$\bullet$ \quad underlying graph $G = (V, H, L)$ \\ 
		$\bullet$ \quad level function $\ell : V \to \bZ_{\leq 0}$ \\ 
		$\bullet$ \quad $k$-enhancement $o : H \cup L \to \bZ$};
	
	\node[mynode](covenhancedlg)[below = of twistedmskdiff] {Def. \ref{def:normalized_cover_enhanced_level_graph}: \textbf{normalized cover of enhanced} \\ 
		\textbf{level graphs} $\tau \acts \hat G^+ \xrightarrow{\pi} G^+$ with \\
		$\bullet$ \quad enhanced level graph \\
		\myindent $\hat G^+ = (\hat G, \hat \ell = \ell \circ \pi, \hat o)$\\
		$\bullet$ \quad $k$-enhanced level graph $G^+ = (G, \ell, o)$};
	
	\node[mynode](tropcurve)[below = of enhancedlg] {Def. \ref{def:tropical_k-Hodge_bundle}: \textbf{tropical $k$-canonical} \\
		\textbf{divisor $(\Gamma, D) \in \tropHodge$} consisting of \\
		$\bullet$ \quad Tropical curve $\Gamma = (G, g, l)$ with \\
		\myindent vertex genus function $g : V(G) \to \bN$, \\
		\myindent edge-length function $l : E(G) \to \bR_{>0}$ \\
		$\bullet$ \quad effective $k$-canonical divisor  \\
		\myindent $D = kK_\Gamma + (f)$ with rational function $f$};
	\node[mynode](covtropcurve)[below = of covenhancedlg] {Def. \ref{def:tropical_normalized_cover}: \textbf{normalized cover of tropical} \\
		\textbf{curves }$\big[\tau \acts \hat\Gamma \xrightarrow{\pi} \Gamma \big] \in \tropMSD$ with \\
		$\bullet$ \quad unramified $k$-cyclic tropical Hurwitz \\
		\myindent cover $\pi$ \\
		$\bullet$ \quad deck transformation $\tau$ with $\tau^k = \id$ \\
		$\bullet$ \quad effective $k$-canonical divisor \\
		\myindent $D = kK_\Gamma + (f)$ marked by the legs of $\Gamma$ \\
		$\bullet$ \quad effective canonical divisor \\
		\myindent $F = K_{\hat \Gamma} + \big(\frac{f \circ \pi}{k}\big)$ marked by the legs of $\hat \Gamma$};

	\draw[->] (mskdiff) -- (kdiff) node[midway, above] {$\cong$};
	\draw[->] (twistedmskdiff) -- (twistedkdiff) node[midway, above] {};
	\draw[->] (covenhancedlg) -- (enhancedlg) node[midway, above] {};
	\draw[->] (covtropcurve) -- (tropcurve) node[midway, above] {};
	
	\begin{scope}[transform canvas={xshift=-\arrowoffset}]
		\draw[->] (kdiff) -- (twistedkdiff) node[midway, left] {degenerate};	
		\draw[->] (mskdiff) -- (twistedmskdiff) node[midway, left] {degenerate};
		\draw[->] (twistedmskdiff) -- (covenhancedlg) node[midway, left] {combinatorial data};
	\end{scope}
	
	\draw[->] (twistedkdiff) -- (enhancedlg) node[midway, right] {combinatorial data};		
	\draw[->] (tropcurve) -- (enhancedlg) node[midway, right] {forget edge-lengths \\ and use Lem. \ref{lem:construction_of_k-enhancement}};
	\draw[->] (covtropcurve) -- (covenhancedlg) node[midway, right] {forget edge-lengths};
	
	\begin{scope}[transform canvas={xshift=\arrowoffset}]
		\draw[->] (twistedkdiff) -- (kdiff) node[midway, right] {smoothen};	
		\draw[->] (twistedmskdiff) -- (mskdiff) node[midway, right] {smoothen};	
		\draw[->, dashed] (covenhancedlg) -- (twistedmskdiff) node[midway, right] {Thm. \ref{intro:thm:main_theorem}};
	\end{scope}
	
	\draw [->, rounded corners] (kdiff.west) -- +(-0.3,0) |- node[pos=0.25, right] {$\trop_{\Omega^k}$ \\ Sec. \ref{subsec:tropicalization}} (tropcurve);
	\draw [->, rounded corners] (mskdiff.east) -- +(0.7,0) |- node[pos=0.25, left] {$\trop_{\Xi^k}$\!\!\!} (covtropcurve);
	
\end{tikzpicture}	
	\caption{Overview of the notions and notation used in this article. Horizontal arrows forget the cover. The commutative triangle in \eqref{eq:reduction_triangle} can be seen on the outside of this diagram.}
	\label{fig:strategy}
\end{figure}

We conclude Section~\ref{sec:realizability_locus} with a result in analogy to \cite[Theorem 6.6]{MUW17_published}.

\begin{mythm} \label{thm:dimension_realizability_locus}
	For $k \geq 2$, the realizability locus admits the structure of a generalized cone complex, all of whose maximal cones have dimension $(2+2k)(g-1)-1$.
	The fiber in the realizability locus over a maximal-dimensional cone in $M_g^\trop$ is a generalized cone complex, all whose maximal cones have relative dimension $(2k-1)(g-1)$.
\end{mythm}

In Section~\ref{sec:example} we illustrate the language that we developed throughout Section~\ref{sec:realizability_locus} by applying our theory to give a complete description of the realizability locus over the dumbbell graph for $k = 2$.

\begin{myrem}
	\begin{enumerate}[(i)]
	\item For the general study of the realizability of curves with divisors as in \cite{Cartwright} it is crucial to ask for realizations by divisors \emph{of the same rank}, i.e.~Baker's specialization inequality \cite[Corollary~2.11]{Bak08} should be an equality. Without this condition every effective divisor on a tropical curve of genus $\geq 2$ would be realizable, simply because tropicalization of curves with divisors is surjective onto the tropical moduli space by \cite[Theorem~3.2]{MUW17_published}. For pluri-canonical divisors, this rank condition is always implicitly included, simply because the (tropical) rank of a (tropical) pluri-canonical divisor is always equal to $(2k - 1)(g - 1)$ by the (tropical) Riemann-Roch theorem (see \cite{GK08} for the tropical Riemann-Roch theorem).
	
	\item For $k = 1$ our Theorem~\ref{intro:thm:trop_k_Hodge_bundle} contains \cite[Theorem 4.3 (i) and (ii)]{LU17} as special case. Furthermore, every tropical normalized cover with $k = 1$ is necessarily the identity and the conditions from Theorem \ref{intro:thm:main_theorem} reduce to the conditions of \cite[Theorem~6.3]{MUW17_published} (see Remark \ref{rem:MUW_recovered} for details). Hence we recover the results of \cite{MUW17_published}.
	
	\item Our construction of $\tropHodge$ is a straight-forward generalization of the tropical Hodge bundle introduced in \cite{LU17}. In fact, Theorem \ref{intro:thm:trop_k_Hodge_bundle} could have been proved with the same ideas as in \cite{LU17}.
	
	\item The techniques involved in the proof of Theorem \ref{intro:thm:main_theorem} give a very similar criterion to decide which boundary strata of the moduli space of multi-scale $k$-differentials $\PMSD[k]{g}{n}{\mu}$ are nonempty, see Appendix~\ref{app:generalized_strata}.
	
	\item In Theorem \ref{intro:thm:main_theorem} we are concerned with finding realizations in the principal stratum $\mu = (1, \ldots, 1)$. A slight modification of the ideas from the proof can be used to give a criterion for realizability in any other stratum as well, see Remark~\ref{rem:realization_in_different_strata}. 
	
	\item The reason for reducing the realizability problem to the seemingly more complicated question for normalized covers is subtle. On the classical side, a $k$-differential is called \emph{primitive} if it is not a power of some $k'$ differential with $k' < k$ and $k'$ dividing $k$. This property is entirely invisible on the tropical side, i.e.\ when realizing a tropical curve consisting of a single vertex we may choose to realize it with a primitive or non-primitive differential. This choice has to be fixed in order to proceed and corresponds precisely to choosing a normalized cover.
\end{enumerate}
\end{myrem}

\subsection{Related work and future applications}

Very little is known about the topology of the projectivized strata $\bP \Omega^k \cM_{g} (\mu)$ of the $k$-Hodge bundle. We believe that our criterion will be useful for further research in this direction.

In the recent and much-celebrated work \cite{CGP21} the authors computed the top weight cohomology of $\cM_{g}$ from the reduced rational cohomology of the link around the cone point of $M_{g}^\trop$. The same technique was applied to the top weight cohomology of the moduli space of abelian varieties $\mathcal{A}_g$ in \cite{BBCMMW21} shortly after. In both cases it is vital to identify the tropical moduli space with the boundary complex of the classical moduli space (see e.g.~\cite{ACP} for the case of curves). With our description of the realizability locus in $\tropHodge$ we take the first step towards a similar computation of the top weight cohomology of strata of $k$-differentials. More precisely, we expect there to be a map from the dual boundary complex of the multi-scale differential space to the link around the cone point of our realizability locus. This map cannot be expected to be injective; for this we would have to include more combinatorial data into the definition of tropical normalized covers (e.g. an analogue of the so-called \emph{prong matching} which was introduced in \cite{BCGGM} or the data of a \emph{dilation flow} as discussed in \cite{MHRSY24}). However, we are optimistic that a suitable enrichment of our definition will enable a similar application of the techniques of \cite{CGP21} to the multi-scale differential space.

\medskip

We would like to highlight some work related to this article.
Amini-Baker-Brugall\'{e}-Rabinoff \cite{ABBR_I, ABBR_II} study the realizability problem for finite harmonic morphisms of tropical curves. Without the extra data of a pluri-canonical divisor, the global obstruction to realizability induced by the global $k$-residue condition from \cite{BCGGM19} does not occur. Indeed, \cite[Corollary~1.6]{ABBR_I} shows that the only obstructions occur locally at the vertices. Furthermore, adding the data of an effective divisor to the problem, the condition on the rank of the realization being equal to the rank of the tropical divisor is a non-trivial condition, see \cite[Section~5]{ABBR_II}.
	
By \cite[Theorem~1.1]{CJP15} every effective divisor class on a chain of loops is realizable by an effective divisor of the same rank. This is not a contradiction to our findings in Section~\ref{subsec:example_dumbbell_graph} because in this article we consider the much harder problem of realizability of divisors rather than divisor classes.
	
In some sense, Baker-Nicaise \cite{BN16} use a different framework to discuss tropicalizations of $k$-differentials. More precisely, they associate to any pluri-canonical form on a curve $X$ a so-called \emph{weight function} on the Berkovich analytification $X^\an$. This is related to our divisor-based point of view since the induced divisor of a weight function is again a pluri-canonical divisor by \cite[Corollary~3.2.5]{BN16}.

\subsection*{Acknowledgements.}
We are grateful to our advisors M.~Möller and M.~Ulirsch for their support during the preparation of this article and for their detailed and helpful feedback on an earlier draft.
We thank Q.~Gendron for helpful conversations on his joint work with G.~Tahar and in particular for sharing the content of the back then forthcoming \cite{GT_quadratic}.
We are also grateful to D. Maclagan for a discussion on the topic of this article during the 2021 edition of the conference ``Effective Methods in Algebraic Geometry'' which inspired Example~\ref{subsec:kK_Gamma_realizable}. 
Finally, we thank the anonymous referee for pointing out a flaw in the proof of Proposition~\ref{prop:maximal_cone} in an earlier version of this text.
The authors acknowledge support by the LOEWE-Schwerpunkt ``Uniformisierte Strukturen in Arithmetik und Geometrie'', by the Deutsche Forschungsgemeinschaft (DFG, German Research Foundation) - project number 456557832 and the Collaborative Research Centre TRR 326 \textit{Geometry and Arithmetic of Uniformized Structures}, project number 444845124.

\section{Tropical \texorpdfstring{$k$}{k}-Hodge bundle}
\label{sec:tropical}

Fix integers $g \geq 2$ and $k \geq 1$. In this section we will describe a tropical version $\tropHodge$ of $\bP \Omega^k \cM_g$ together with a tropicalization map
\[ \trop_{\Omega^k} : \bP \Omega^k \cM_g^\an \longrightarrow \tropHodge \ . \]
The underlying set of the tropical $k$-Hodge bundle $\tropHodge$ parametrizes pairs $([\Gamma], D)$ of isomorphism classes of stable tropical curves $\Gamma$ of genus $g$ and effective divisors $D$ linearly equivalent to $kK_\Gamma$. In the special case of $k = 1$ we recover the description of the tropical Hodge bundle from \cite[Definition~4.1]{LU17}. In Section~\ref{subsec:tropical_k_Hodge_bundle} we prove Theorem \ref{intro:thm:trop_k_Hodge_bundle}. To this end we use the moduli space $\tropDiv_{g,d}^\trop$ of tropical curves with effective divisor of fixed degree $d = k(2g-2)$ which was constructed in \cite[Definition 2.1]{MUW17_published} and exhibit $\tropHodge$ as a locus in $\tropDiv_{g,d}^\trop$. The tropicalization map $\trop_{\Omega^k}$ is defined in Section~\ref{subsec:tropicalization} by restricting the more general tropicalization map from \cite[Section 3.1]{MUW17_published}. 

We conclude this section by defining the \emph{realizability locus} as the image of $\trop_{\Omega^k}$ and formally state the realizability problem in Section~\ref{subsec:realizability_problem}.

\subsection{Tropical curves}

\begin{mydef}
	A \emph{graph} is a tuple $G = (V, H, L, \iota, a)$ where
	\begin{enumerate}[(i)]
		\item the finite sets $V$, $H$, and $L$ are the vertices, half-edges and legs of the graph respectively,
		\item the map $\iota : H \to H$ is a fixpoint-free involution on the half-edges $H$ that determines the edges of the graph, and
		\item the map $a : H \cup L \to V$ assigns to every half-edge and leg the incident vertex.
	\end{enumerate}
	For a graph $G$ let $E := \big\{ \{h, h'\} \in H^2 \mid \iota(h) = h' \big\}$ be the set of unoriented edges. In the following we will often denote a graph simply as 3-tuple $(V, E, L)$ of vertices, edges and legs with the rest of the underlying data remaining implicit. If there are no legs, we abbreviate further and simply write $(V, E)$.
	
	The \emph{valence} of a vertex $v \in V$ is defined as $\valence(v) := \vert a^{-1}(v) \vert$. 
	
	A \emph{(vertex) weighted graph} is a graph $G$ together with a map $g : V \to \bN$. The weighted graph is called \emph{stable} if for each vertex $v \in V$ the stability condition
	\[
	2g(v) -2 + \valence(v) > 0
	\]
	holds. The \emph{genus} of a weighted graph is defined to be
	\[ g(G) := b_1(G) + \sum_{v \in V} g(v), \]
	where $b_1(G)$ is the first Betti number of $G$.
\end{mydef}

A \emph{tropical curve} is a connected weighted metric graph $\Gamma$ given by the data of a graph $G$, vertex weights $g : V \to \bN$ and edge lengths $l : E \to \bR_{>0}$. We call $g(v)$ the genus of the vertex $v$. The topological realization of $\Gamma$ is the metric space obtained by gluing real intervals $[0, l(e)]$ for every edge and $[0, \infty)$ for every leg according to adjacency in $G$. Any weighted graph $(G', g')$ giving rise to the same topological realization is referred to as a \emph{model} for $\Gamma$. Note that every stable tropical curve has a unique minimal model in the sense of minimal number of edges and vertices. We will usually not distinguish between topological realization and minimal model.
The \emph{genus} of $\Gamma$ is defined to be the genus of any model (there is no dependence on this choice) and is denoted $g(\Gamma)$.
The tropical curve $\Gamma$ is called \emph{stable} if its minimal model is stable.

\subsection{Moduli of tropical curves}
\label{subsec:moduli_tropical_cuves}

The following description of the moduli space $M_{g,n}^\trop$ of stable tropical curves of genus $g$ with $n$ legs can be found e.g. in \cite[Section 4]{ACP}. Note that in the description of $M_{g,n}^\trop$ the legs are usually assumed to be labeled.
Let $G$ be a weighted graph and let $e$ be an edge in $G$. We denote by $G / \{e\}$ the graph that arises from $G$ by contracting $e$ into a single vertex $v$ of weight 
\begin{equation*}
	g(v) = \begin{cases}
		g(v_1) + g(v_2) & \text{if $e$ was connecting $v_1$ and $v_2$} \\
		g(v_1) + 1 & \text{if $e$ was a self-loop at vertex $v_1$.}
	\end{cases}
\end{equation*}
Define the category $\cG_{g,n}$ with objects being stable weighted graphs of genus $g$ with $n$ legs and morphism are generated by weighted edge contractions $G \to G/\{e\}$ as well as graph automorphisms respecting the labeling of the legs.

Given $G \in \cG_{g,n}$ we associate to it the rational polyhedral cone $\sigma_G := \bR_{\geq 0}^{E(G)}$. In fact, this defines a contravariant functor from $\cG_{g,n}$ to the category of rational polyhedral cones, where edge contractions are taken to isomorphisms onto faces. The moduli space is now defined as
\[ M_{g,n}^\trop := \varinjlim\limits_{\cG_{g,n}} \sigma_G. \]
Note that the points of $M_{g,n}^\trop$ are in one-to-one correspondence with isomorphism classes of tropical curves of genus $g$ with $n$ legs.
A topological space arising as colimit over a finite diagram of rational polyhedral cones where all morphisms are isomorphisms onto faces is called \emph{generalized cone complex} in \cite[Section 2.6]{ACP}.

\subsection{Divisors on tropical curves}

Let $\Gamma$ be a tropical curve without legs. A \emph{divisor} $D$ on $\Gamma$ is an element of the free abelian group generated by the points in the topological realization of $\Gamma$. We denote the abelian group of divisors on $\Gamma$ by $\Div(\Gamma)$. A divisor $D = \sum a_p p$ is called \emph{effective} if $a_p \geq 0$ for every $p$. In this case we write $D \geq 0$. The \emph{degree} of $D$ is defined as $\deg(D) := \sum a_p$. The \emph{support} of $D$ is $\supp(D) := \{ p \in \Gamma \mid a_p \neq 0 \}$. By definition, the support is a finite subset of $\Gamma$. One often imagines an effective divisor $D$ as a pile of $D(p) = a_p$ ``chips'' at every point $p \in \supp(D)$.
The data of an effective divisor $D$ on a tropical curve $\Gamma$ without legs is equivalent to a tropical curve $\tilde \Gamma$ arising from $\Gamma$ by attaching $D(p)$ many unlabeled legs at every $p \in \supp(D)$. From now on all tropical curves are a priori without legs, but given a divisor we will pass to the equivalent representation $\tilde \Gamma$ whenever convenient.

A rational function on $\Gamma$ is a continuous function $f : \Gamma \to \bR$ whose restriction to any edge is piece-wise linear with integer slopes. We denote the set of rational functions on $\Gamma$ by $\Rat(\Gamma)$. Every $f \in \Rat(\Gamma)$ gives rise to an induced divisor
\[ (f) := \sum_{p \in \Gamma}^{} \big( \text{sum of outgoing slopes of $f$ at $p$} \big) \cdot p \in \Div(\Gamma)\ . \]
Note that $(f)$ is indeed a finite sum. Two divisors $D, D' \in \Div(\Gamma)$ are \emph{linearly equivalent} if there exists $f \in \Rat(\Gamma)$ such that $D = D' + (f)$. In this case we write $D \sim D'$. Analogously to \cite[Definition 3.1]{LU17} we define:

\begin{mydef} \label{def:tropical_linear_system}
	Let $D \in \Div(\Gamma)$ be a divisor. We define the \emph{linear system} of $D$ to be
	\[\vert D \vert := \{ D' \in \Div(\Gamma) \mid D' \geq 0 \text{ and } D \sim D' \}. \]
\end{mydef}

The canonical divisor of a tropical curve $\Gamma$ without legs is defined as
\[ K_\Gamma := \sum_{v \in V}^{} \big(2 g(v) - 2 + \valence(v) \big) v.  \]
If $\Gamma$ has legs, then we define $K_\Gamma$ to be the canonical divisor of the tropical curve arising from $\Gamma$ by removing the legs. 
Note that contrary to the classical situation there is a canonical element in the canonical linear system. Furthermore, note that $\deg K_\Gamma = 2g(\Gamma) - 2$. The elements of $\vert k K_\Gamma \vert $ are called \emph{pluri-canonical} divisors -- they are effective divisors by definition.

\subsection{Tropical \texorpdfstring{$k$}{k}-Hodge bundle}
\label{subsec:tropical_k_Hodge_bundle}

We are now ready to define the central object of this section. A priori we define the $k$-Hodge bundle as a set, but the proof of Theorem~\ref{intro:thm:trop_k_Hodge_bundle} shows that it can in fact be considered as a generalized cone complex.

\begin{mydef} \label{def:tropical_k-Hodge_bundle}
	Let integers $g \geq 2$ and $k \geq 1$ be given. Define
	\begin{equation*}
		\tropHodge := \big\{ ([\Gamma], D) \mid [\Gamma] \in M_g^\trop \text{ and } D \in \vert kK_\Gamma \vert \big\}.
	\end{equation*}
	This space is called the \emph{tropical $k$-Hodge bundle}. 
\end{mydef}

Recall \cite[Proposition 2.2]{MUW17_published}, where the authors construct a moduli space $\tropDiv_{g,d}^\trop$ parametrizing pairs $([\Gamma], D)$ of isomorphism classes of stable tropical curves of genus $g$ and effective divisors $D \in \Div(\Gamma)$ of degree $d$. The construction is completely analogous to that of $M_{g,n}^\trop$ given in Section \ref{subsec:moduli_tropical_cuves} above; the only modification is that legs (corresponding to support points of divisors) are now unlabeled. Hence, the colimit involves more automorphisms.

\begin{proof}[Proof of Theorem \ref{intro:thm:trop_k_Hodge_bundle}]
	We identify $\tropHodge$ as a subcomplex of (a subdivision of) $\tropDiv_{g,d}^\trop$ for $d = k(2g-2)$ as follows. Let $([\Gamma], D) \in \tropDiv_{g, d}^\trop$ and let $G$ be the minimal model for $(\Gamma, D)$ such that $D$ is supported on the vertices of $G$. By construction $([\Gamma], D)$ is contained in (a quotient of) the cone $\sigma_G = \bR_{\geq 0}^{E(G)}$. We will now describe finitely many rational polyhedral cones in $\sigma_G$ that contribute to $\tropHodge$.
	
	By Definition \ref{def:tropical_k-Hodge_bundle} the pair $([\Gamma], D)$ is contained in $\tropHodge$ if and only if there exists a rational function $f$ on $\Gamma$ such that 
	\begin{equation} \label{eq:pluricanonical_divisor}
		D = kK_\Gamma + (f). 
	\end{equation}
	Fix an orientation for the edges of $G$. To specify a rational function $f$ (up to a global additive constant in $\bR$) that satisfies \eqref{eq:pluricanonical_divisor} we first need to choose an initial slope $m_e \in \bZ$ at the beginning of every edge $e \in E(G)$ subject to the condition that at every vertex $v \in V(G)$
	\begin{equation} \label{eq:condition_existence_rational_function}
		D(v) = k\big(2g(v) - 2 + \valence(v)\big) + \sum_{\text{outward edges at } v}^{} m_e - \sum_{\text{inward edges of }v}^{} m_e
	\end{equation}
	holds. 
	For each such choice a linear subspace of $\sigma_{G}$ is cut out by the continuity of~$f$ and this determines a cone for each choice of $\{m_e\}_{e \in E(G)}$. By \cite[Lemma 1.8]{GK08}\footnote{Note that \cite[Lemma 1.8]{GK08} starts with a choice of slopes coming from a rational function and therefore obtains only finitely many collections of slopes compatible with Equation~\eqref{eq:condition_existence_rational_function}. If on the other hand we start by choosing slopes, we might make an \enquote{impossible choice}, e.g. assigning a nonzero $m_e$ to a self-loop $e$. In this case the continuity of $f$ forces the length of $e$ to be zero. Thus, the choice of $\{m_e\}_e$ yields a non-trivial cone, which however does not contribute in any non-trivial way to the cone complex structure since it coincides with the cone associated to the edge contraction $G/e$.} only finitely many of these cones have non-empty intersection with $\sigma_G^\circ = \bR_{>0}^{E(G)}$.
	The finitely many cones determined this way when ranging over all $G$ constitute the entire generalized cone complex structure of $\tropHodge$.
	
	For the statement on the dimension recall from \cite[Proposition 2.2]{MUW17_published} that 
	\[ \dim \tropDiv_{g, k(2g-2)}^\trop = 3g - 3 + k(2g-2)   \ .\] 
	This provides an upper bound. This bound is attained by the cone described in Example~\ref{exa:maximal_cone_tropicalHodge}. Moreover, the tropical $k$-Hodge bundle is not equidimensional: see \cite[Figure~1]{LU17} for an example in $k = 1$ and Example~\ref{exa:tropicalHodge_not_purediml} for a generalization to any $k$.
\end{proof}

\begin{myexa} \label{exa:maximal_cone_tropicalHodge}
	Consider the graph $G$ depicted in Figure \ref{fig:maximal_cone_tropicalHodge}. It consists of $g$ vertices each of which has one self-loop as well as an incident separating edge joining it to a central chain of $g-2$ vertices. All vertices have weight 0. This graph is trivalent and hence stable. If $G$ is endowed with edge-lengths, we obtain a tropical curve $\Gamma$. The canonical divisor of $\Gamma$ is the sum over all trivalent vertices, hence $kK_\Gamma$ has $k$ chips on each vertex. All of these chips can be moved independently onto the bridge edges joining the vertices with self-loops to the rest of the graph. Call the resulting divisor $D$ (see Figure \ref{fig:maximal_cone_tropicalHodge} for a picture of $D$ with $k = 3$). The pair $(\Gamma, D)$ has precisely $(3 + 2k)(g-1)$ degrees of freedom: $g$ for the length of the self-loops, $2g-3$ for the lengths of the remaining edges, and $k(2g-2)$ for the positions of the support points of $D$ along the edges they lie on. Hence, the cone of tropical curves with underlying graph $G$ and divisor $D$ is of maximal dimension in $\tropHodge$.
	
	\begin{figure}[htb]
		\centering
		\begin{minipage}{0.9\textwidth}
			\centering
			\begin{tikzpicture}[scale=2]
	
	\def\rad{0.3}													
	\newcommand{\blackDot}[1]{\fill #1 circle [radius = 1pt]}		
	\def\k{3}				
	\def\twok{6}
	\def\twokplusone{7}
	\def\kplusone{4}
	
	\draw (-1-\rad, 0) circle [radius = \rad];
	\path[draw] (-1,0) -- ++ (1,0) 
		node(A)[inner sep = 0]{} 
		-- ++ (0,1);
	\draw (0,1+\rad) circle [radius = \rad];
	\path[draw] (A) -- ++ (1,0) 
		node(B)[inner sep = 0]{} 
		-- ++ (0,1);
	\draw (1,1+\rad) circle [radius = \rad];
	\path[draw] (B) -- ++ (0.5,0);
	
	\draw (1.75, 0) node[anchor = center] {$\cdots$};

	\draw (3.5+\rad, 0) circle [radius = \rad];
	\path[draw] (3.5,0) -- ++ (-1,0) 
		node(C)[inner sep = 0]{} 
		-- ++ (0,1);
	\draw (2.5,1+\rad) circle [radius = \rad];
	\path[draw] (C) -- ++ (-0.5,0);

	\foreach \x in {1,..., \k} {
		\blackDot{(-1 + \x / \kplusone, 0)};
		\blackDot{(2.5 + \x / \kplusone, 0)};
	}
	\foreach \y in {1,...,\twok} { 
		\blackDot{(0, \y / \twokplusone)};
		\blackDot{(1, \y / \twokplusone)};
		\blackDot{(2.5, \y / \twokplusone)};
	}
	
\end{tikzpicture}
			\caption{Graph $G$ with divisor $D$ defining a cone in $\bP \Omega^3 M_g^\trop$ of maximal dimension.}
			\label{fig:maximal_cone_tropicalHodge}
		\end{minipage}
	\end{figure}
\end{myexa}

\begin{myexa} \label{exa:tropicalHodge_not_purediml}
	Consider the same graph $G$ as in Example~\ref{exa:maximal_cone_tropicalHodge} but this time with a pluri-canonical divisor $D'$ which has all but two chips on bridge edges (as before) and the remaining two chips on one of the self-loops (see Figure~\ref{fig:tropicalHodge_not_purediml} for a picture with $k = 3$). In this case, any metrization $(\Gamma, D')$ of this situation has one degree of freedom less than in Example~\ref{exa:maximal_cone_tropicalHodge}: the continuity of the rational function $f$ which defines $D' = kK_\Gamma + (f)$ enforces the distance between each of the two chips on the loop and the trivalent vertex where the loop is rooted to be equal. This condition cannot be overcome without changing the combinatorial type of the divisor, so the cone of $\tropHodge$ associated to $(G, D')$ is still maximal.
	
	\begin{figure}[htb]
		\centering
		\begin{minipage}{0.9\textwidth}
			\centering
			\begin{tikzpicture}[scale=2]
	
	\def\rad{0.3}													
	\newcommand{\blackDot}[1]{\fill #1 circle [radius = 1pt]}		
	\def\k{3}				
	\def\twok{6}
	\def\twokplusone{7}
	\def\kplusone{4}
	
	\draw (-1-\rad, 0) circle [radius = \rad];
	\path[draw] (-1,0) -- ++ (1,0) 
	node(A)[inner sep = 0]{} 
	-- ++ (0,1);
	\draw (0,1+\rad) circle [radius = \rad];
	\path[draw] (A) -- ++ (1,0) 
	node(B)[inner sep = 0]{} 
	-- ++ (0,1);
	\draw (1,1+\rad) circle [radius = \rad];
	\path[draw] (B) -- ++ (0.5,0);
	
	\draw (1.75, 0) node[anchor = center] {$\cdots$};

	\draw (3.5+\rad, 0) circle [radius = \rad];
	\path[draw] (3.5,0) -- ++ (-1,0) 
	node(C)[inner sep = 0]{} 
	-- ++ (0,1);
	\draw (2.5,1+\rad) circle [radius = \rad];
	\path[draw] (C) -- ++ (-0.5,0);

	\blackDot{(-1 + 2 / \kplusone, 0)};
	\blackDot{(-1-\rad, \rad)};
	\blackDot{(-1-\rad, -\rad)};
	\foreach \x in {1,..., \k} {
		\blackDot{(2.5 + \x / \kplusone, 0)};
	}
	\foreach \y in {1,...,\twok} { 
		\blackDot{(0, \y / \twokplusone)};
		\blackDot{(1, \y / \twokplusone)};
		\blackDot{(2.5, \y / \twokplusone)};
	}
	
\end{tikzpicture}
			\caption{Graph $G$ with divisor $D'$ defining a cone in $\bP \Omega^3 M_g^\trop$ which is maximal but not of maximal dimension.}
			\label{fig:tropicalHodge_not_purediml}
		\end{minipage}
	\end{figure}
\end{myexa}

\subsection{Analytification}
\label{subsec:analytification}

The natural domain for the tropicalization maps use in this article is the analytification $\cM^\an$ in the sense of \cite{Ber90}, where $\cM$ may be substituted with any of the relevant moduli spaces such as $\cM_{g,n}$, $\bP \Omega^k \cM_g$, or $\classicDiv_{g, d}$. More precisely, what we will be working with is the underlying topological space of $\cM^\an$ which by abuse of notation we will denote with the same symbol again. If we denote the base field by $K$, then the points of $\cM^\an$ are precisely the $K'$-points of $\cM$ where $K'$ is any valued field extension of $K$, subject to the following equivalence relation: a $K'$-point and a $K''$-point are equivalent if there is a valued field extension $L$ of both $K'$ and $K''$ such that the base change to $L$ agrees (see \cite[Section~3]{Uli_stack_quotient} for details).

\subsection{Tropicalization}
\label{subsec:tropicalization}

Let $g \geq 2$ be an integer and let $X$ be a smooth, proper algebraic curve of genus $g$ over a non-Archimedean field $K$, i.e. a $K$-point of the moduli space of curves. 
Possibly after passing to a finite non-Archimedean field extension $K \subseteq K'$, there is a stable model $\cX$ of $X$ over the valuation ring $R$ of $K'$. The central fiber $\cX_0$ is a nodal curve. Denote the set of irreducible components of $\cX_0$ by $\{C_v\}_{v \in V}$. Let $G$ denote the dual graph of $\cX_0$, i.e. the set of vertices of $G$ is precisely $V$ and for every node in $\cX_0$ there is one edge in $G$. Here, the edge corresponding to a node $q$ joins two distinct vertices $v$ and $w$ if $q \in C_v \cap C_w$ and it is a self-loop at vertex $v$ if $q$ is a node of $C_v$. This graph is vertex weighted by $g(v)$ equal to the genus of the normalization $C_v^\nu$ of $C_v$. We endow the edge $e$ corresponding to some node $q \in \cX_0$ with an edge length in the following way. Write $\cX$ locally around $q$ as $xy = f$ for $f \in R$ and let $\valuation_R$ denote the valuation of $R$. The length of $e$ is defined to be $\valuation_R(f)$. The resulting metric graph $\Gamma$ is the tropicalization of $X$ in the sense of curves. The tropicalization map
\begin{align*}
	\trop : \cM_{g, n}^\an &\longrightarrow M_{g, n}^\trop \\
	X &\longmapsto \Gamma
\end{align*}
is well-defined (see \cite[Lemma-Definition~2.2.7]{Viv13} for independence of the choice of $K'$), continuous, and surjective by \cite[Theorem 1.2.1]{ACP}. 

If $X$ was endowed with a divisor $D$ then we obtain a divisor on $\trop(X)$ by specialization. \cite{MUW17_published} presents this extended construction as a map between moduli spaces again. More precisely, for any degree $d \geq 0$ the authors of \cite{MUW17_published} construct moduli spaces $\classicDiv_{g, d}$ and $\tropDiv_{g, d}^\trop$ of pairs of smooth algebraic (resp.\ stable tropical) curves of genus $g$ together with an effective divisor of degree $d$ and give a tropicalization map $\trop_{g, d} : \classicDiv_{g, d}^\an \to \tropDiv_{g, d}^\trop$ in the following way.
The curve $X$ can be extended to a semi-stable model $\cX$ such that $D$ extends to a divisor $\cD$ on $\cX$ that does not meet any of the nodes of the special fiber. As before, this might require a base change to a non-Archimedean field extension. The specialization of $D$ to $\Gamma$ is defined to be the multidegree of $\cD_0 := \cX_0 \cap \cD$, i.e.
\[ \operatorname{mdeg}(\cD_0) = \sum_{v \in V} \deg(\cD_0 |_{C_v^\nu}) \cdot [v]. \]

For the purposes of this article we simply define 
\[ \trop_{\Omega^k} : \bP \Omega^k \cM_g^\an \longrightarrow \tropHodge \]
to be the restriction of $\trop_{g, k(2g-2)}$.

\begin{mylem} \label{lem:trop_Omega}
	The map $\trop_{\Omega^k}$ is well-defined, continuous, proper, and closed.
\end{mylem}

\begin{proof}
	By \cite[Lemma 4.20]{Bak08}, the specialization of a canonical divisor on a curve $X$ is the canonical divisor on $\trop(X)$.
	Furthermore, the specialization map is linear and linearly equivalent divisors tropicalize to (tropically) linearly equivalent divisors (see e.g. \cite[Theorem 4.2]{BrandtUlirsch}). In particular, $\trop_{\Omega^k}$ is well-defined. 
	The map is also continuous because $\trop_{g,d}$ is continuous by {\cite[Theorem 3.2]{MUW17_published}}.
	Towards properness of $\trop_{\Omega^k}$, recall from \cite[Section~3.3]{MUW17_published} that $\trop_{g, d}$ is proper. 
	We argue that $\bP \Omega^k \cM_g^\an$ is closed in $\classicDiv_{g, d}^\an$, which then implies properness for $\trop_{\Omega^k}$. Indeed, we have that $\bP \Omega^k \cM_g$ is a closed substack of $\classicDiv_{g, d}$ and hence there is an closed immersion of the underlying coarse moduli spaces. 
	The analytification of a closed immersion of schemes is a closed immersion of Berkovich spaces by \cite[Corollary~3.5.2]{Ber90}. Finally, the underlying topological spaces of the analytic stacks $\bP \Omega^k \cM_g^\an$ and $\classicDiv_{g, d}^\an$ are homeomorphic to the topological spaces of their coarse moduli spaces by \cite[Proposition~3.8]{Uli_stack_quotient} and this finishes the argument. 
	Closedness of $\trop_{\Omega^k}$ follows from properness since $\tropHodge$ is a locally compact Hausdorff space.
\end{proof}

Note that $\bP \Omega^k \cM_g$ as well as $\tropHodge$ admit natural forgetful maps to $\cM_g$ and $M_g^\trop$, respectively. These are compatible with tropicalization maps in the following sense.

\begin{myprop}
	The diagram	
	\begin{equation*}
		\begin{tikzcd}
			\bP \Omega^k \cM_g^\an \arrow[rr, "\trop_{\Omega^k}"] \arrow[d] & & \tropHodge \arrow[d] \\
			\cM_g^\an \arrow[rr, "\trop"] & & M_g^\trop
		\end{tikzcd}
	\end{equation*}
	commutes, where the vertical arrows are natural forgetful morphisms.
\end{myprop}

\begin{proof}
	This is essentially a modified version of the first part of \cite[Theorem 1.2.2]{ACP} using unlabeled points. Alternatively, one can see this with the explicit descriptions of the two tropicalization maps that were given above.
\end{proof}

\subsection{The realizability problem}
\label{subsec:realizability_problem}

We conclude this section with the following observation: by Theorem \ref{intro:thm:trop_k_Hodge_bundle}, the dimension (of a maximal cone) of $\tropHodge$ is $(3 + 2k)(g-1)$. On the other hand, $\dim \bP \Omega^k \cM_g \leq (2+2k)(g-1)-1$ by \cite[Theorem 1.1]{BCGGM19}. By the following argument, this implies that $\trop_{\Omega^k}$ cannot be surjective.

First note that $\bP \Omega^k \cM_g$ is a closed substack of $\classicDiv_{g,d}$. The compactification $\classicDiv_{g,d} \subseteq \compact{\classicDiv}_{g,d}$ was identified in \cite[Theorem 1.2]{MUW17_published} as a toroidal embedding of Deligne-Mumford stacks in the sense of \cite[Definition 6.1.1]{ACP}. This means that locally around any geometric point of $\compact{\classicDiv}_{g,d}$ there exists a so-called \emph{small toric chart} $V$ \cite[Definition 6.2.4]{ACP}, i.e.\ a scheme $V$ and an \'etale morphism $V \to \compact{\classicDiv}_{g,d}$ such that the pull-back $V^\circ$ of $\classicDiv_{g,d}$ to $V$ is a toroidal embedding $V^\circ \subseteq V$ in the sense of \cite{KKMS}. In particular, the boundary of $V^\circ \subseteq V$ is without self-intersection. Now consider the pull-back $U$ of $\bP \Omega^k \cM_g$ to $V$. The tropicalization $\trop_V(U)$ in the sense of \cite{Uli17} is then a finite rational polyhedral cone complex of dimension $\leq \dim U \leq \dim \bP \Omega^k \cM_g$ by \cite[Theorem 1.1]{Uli15}. Taking the supremum over all small toric charts $V$ we get
\[ \dim \trop_{\compact{\classicDiv}_{g,d}} (\bP \Omega^k \cM_g) = \sup \dim \trop_V(U) \leq \dim \bP \Omega^k \cM_g. \]
It remains to argue that this coincides with the dimension of $\Im \trop_{\Omega^k}$. To this end note that by \cite[Theorem 1.2]{Uli17} the tropicalization $\trop_{\compact{\classicDiv}_{g, d}}$ coincides with the retraction map $\rho : \compact{\classicDiv}_{g, d}^\an \to \mathfrak{S}(\compact{\classicDiv}_{g, d})$ in the sense of \cite{Thuillier}. However, the same holds for $\trop_{\Omega^k}$ by \cite[Theorem~3.2]{MUW17_published}, so the claim follows. Putting everything together we conclude
$$\dim \trop_{\Omega^k}(\bP \Omega^k \cM_g) \leq (2+2k)(g-1)-1 < (3+2k)(g-1) = \dim \tropHodge  .$$ 
Thus, $\trop_{\Omega^k}$ cannot be surjective. This motivates the following definition.

\begin{mydef} \label{def:realizability_locus}
	The image of $\trop_{\Omega^k}$ is called \emph{realizability locus}. A tropical curve $\Gamma$ with effective pluri-canonical divisor $D \in |kK_\Gamma|$ is called \emph{realizable} if the pair $([\Gamma], D) \in \tropHodge$ is contained in the realizability locus.
\end{mydef}

The realizability problem asks for a criterion to determine if $([\Gamma], D)$ is in the realizability locus. Corollary~\ref{cor:reduction_to_covers} together with Theorem~\ref{thm:main_theorem} provide its answer.

	\section{Moduli space of multi-scale \texorpdfstring{$k$}{k}-differentials}

Fix integers $g \geq 2$ and $k \geq 1$.
Throughout this section we work over the field $\bC$ of complex numbers.
A tuple $\mu = (m_1, \dots, m_n) \in \bZ^n$ such that $\sum m_i = k (2g-2)$ is called a \emph{type}.
The stratum of $k$-differentials of type $\mu$ is the space $\Omega^k \cM_g(\mu)$ parametrizing $k$-differentials where the zero and pole orders are as prescribed by $\mu$. If all entries of $\mu$ are nonnegative, then this is a subspace of the $k$-Hodge bundle $\Omega^k \cM_g$, otherwise it is a subspace of the $k$-Hodge bundle twisted by the negative part of $\mu$.

For an integer $d \mid k$, taking a global $d$-th power of a $k/d$-differential on a curve $X$ gives rise to a $k$-differential on $X$.
We will often be interested in those $k$-differentials that are not global powers of $k/d$-differentials.

\begin{mydef}
	A $k$-differential is called \emph{primitive} if it is not a global $d$-th power of a $k/d$-differential for any $d > 1$.
	We denote the union of connected components of $\Omega^k \cM_g(\mu)$ parametrizing primitive $k$-differentials by $\Omega^k \cM_g(\mu)^{\prim}$.
\end{mydef}

For strata of abelian differentials $\Omega \cM_{g} (\mu)$, the authors of \cite{BCGGM} constructed a closure $\MSD{g}{n}{\mu}$, the \emph{moduli space of multi-scale differentials}.%
\footnote{\label{note:order} In \cite{BCGGM}, the marked points of the stratum are labeled. We consider the marked points to be unlabeled, i.e.\ we consider the quotient of the space in \cite{BCGGM} by $\mathrm{Sym}(\mu) = \{ \phi \in S_n \mid m_{\phi(i)} = m_i \text{ for } i = 1, \ldots, n \}$.}
Recall that $\bC^\times$ acts on $\Omega \cM_{g} (\mu)$ by multiplication on the differential.
This action extends to $\MSD{g}{n}{\mu}$ and the projectivization $\PMSD{g}{n}{\mu}$ with respect to this action is a well-behaved compactification of $\bP \Omega \cM_{g} (\mu)$.
The construction of $\MSD{g}{n}{\mu}$ was generalized to strata of primitive $k$-differentials $\Omega^k \cM_{g} (\mu)^{\prim}$ in \cite{CMZ19}. 

Let $(X, \eta)$ be a primitive $k$-differential in $\Omega^k \cM_{g} (\mu)^{\prim}$.
In Section~\ref{subsec:normalized_covers} we recall that $(X, \eta)$ admits a canonical cover 
\[
\pi : (\hat X, \omega) \longrightarrow (X, \eta),
\]
where $(\hat X, \omega)$ is an abelian differential of some type $\hat \mu$.
The canonical cover is unique up to multiplication of $\omega$ by a $k$-th root of unity.
Hence, after projectivizing, there is a well-defined map $\bP \Omega^k \cM_{g} (\mu)^{\prim} \to \bP \Omega \cM_{\hat g} (\hat \mu)$ for $\hat g$, $\hat \mu$ as described in Section~\ref{subsec:normalized_covers} below.
In other words, we can think of $\bP \Omega^k \cM_g (\mu)^{\prim}$ not only as a space parametrizing primitive $k$-differentials of type $\mu$, but equivalently as a space parametrizing canonical covers.

The boundary of the compactification constructed in \cite{CMZ19} parametrizes so-called \emph{multi-scale $k$-differentials}. These are twisted $k$-differentials with the additional data of an (enhanced) level graph together with some compatibility conditions. Details are outlined in Sections~\ref{subsec:twisted_k-differentials} and \ref{subsec:enhanced_level_graph}.
The compactification has the following properties.

\begin{mythm}[{\cite{CMZ19}}]
	There exists a complex orbifold $\bP \MSD[k]{g}{n}{\mu}$, the \emph{moduli space of multi-scale $k$-differentials}, with the following properties.%
	\footnote{Here again, we consider the quotient of the space in \cite{CMZ19} by $\mathrm{Sym}(\mu)$.}
	\begin{enumerate}[(i)]
		\item The space $\PMSD[k]{g}{n}{\mu}$ is a compactification of $\bP \Omega^k \cM_g(\mu)^{\prim}$.
		\item Via the canonical cover construction, the space $\bP \MSD[k]{g}{n}{\mu}$ is embedded as a suborbifold in the corresponding stratum $\bP \MSD{\hat g}{\hat n}{\hat \mu}$ of abelian multi-scale differentials.
	\end{enumerate}
\end{mythm}

We will introduce some more properties of $\bP \MSD[k]{g}{n}{\mu}$ in Sections~\ref{subsec:empty_strata} and \ref{subsec:image_residue_map}, which will turn out to introduce major difficulties in solving the realizability problem in Section~\ref{sec:realizability_locus}.

\subsection{Twisted $k$-differentials}
\label{subsec:twisted_k-differentials}

The underlying curves of the boundary points of the moduli space of multi-scale $k$-differentials $\PMSD[k]{g}{n}{\mu}$ will be nodal curves. The $k$-differentials will degenerate into so-called twisted $k$-differentials with some additional data and compatibility conditions that we will describe in the following.

For an abelian differential, the residue at a pole is defined as the coefficient in front of $z^{-1}$ in the Laurent expansion around that pole.
To define a useful notion of residues for $k$-differentials, recall from \cite[Proposition~3.1]{BCGGM19} that a $k$-differential $\eta$ of order $m = \ord_0 \eta$ may locally be written as
\begin{equation}
	\label{eq:kdifflocal}
	\begin{dcases*}
		z^m (dz)^k & if $m > -k$ or $k \nmid m$, \\
		\left(\frac{r}{z}\right)^k (dz)^k & if $m = -k$, \\
		\left(z^{m/k} + \frac{t}{z}\right)^k (dz)^k & if $m < -k$ and $k \mid m$
	\end{dcases*}
\end{equation}
for some $r \in \bC^\times$ and $t \in \bC$.

\begin{mydef}
        \label{def:kRes}
	For a $k$-differential $\eta$ written as in \eqref{eq:kdifflocal}, the \emph{$k$-residue} of $\eta$ is defined as
	\begin{equation}
		\label{eq:kRes}
		\Res^k_0 \eta :=
		\begin{dcases*}
			0 & if $m > -k$ or $k \nmid m$, \\
			r^k & if $m = -k$, \\
			t^k & if $m < -k$ and $k \mid m$
		\end{dcases*}
	\end{equation}
	for $r$ and $t$ as above.
\end{mydef}

\begin{mydef} \label{def:twisted_k_differential}
	Let $X$ be a nodal curve and let $\mu = (m_1, \dots, m_n)$ be a type.
	A \emph{twisted $k$-differential of type $\mu$} on a stable $n$-pointed curve $(X, \bs)$ is a collection of (possibly meromorphic) $k$-differentials $\eta = \{\eta_v\}_v$ on the irreducible components $X_v$ of $X$ such that no $\eta_v$ is identically zero with the following properties.
	\begin{enumerate}[(i)]
		\item \textbf{(Vanishing as prescribed)} Each $k$-differential $\eta_v$ is holomorphic and nonzero outside the nodes and marked points of $X_v$.
			Moreover, if a marked point $s_i$ lies on $X_v$, then $\ord_{s_i} \eta_v = m_i$.
		\item \textbf{(Matching orders)} For any node of $X$ that identifies $q_1 \in X_{v_1}$ and $q_2 \in X_{v_2}$, the vanishing orders satisfy $\ord_{q_1} \eta_{v_1} + \ord_{q_2} \eta_{v_2} = -2k$.
		\item \textbf{(Matching $k$-residues condition, MRC)} If at a node of $X$ that identifies $q_1 \in X_{v_1}$ with $q_2 \in X_{v_2}$ the condition $\ord_{q_1} \eta_{v_1} = \ord_{q_2} \eta_{v_2} = -k$ holds, then $\Res^k_{q_1} \eta_{v_1} = (-1)^k \Res^k_{q_2} \eta_{v_2}$.
	\end{enumerate}
\end{mydef}

\subsection{Normalized covers}
\label{subsec:normalized_covers}

Consider first a smooth curve with $k$-differential $(X, \eta)$. It admits a canonical cover $\pi : (\hat X, \omega) \to (X, \eta)$ that is unique up to multiplying $\omega$ with a $k$-th root of unity.
In particular, $(\hat X, \omega)$ is unique up to scaling.
The cover $\hat X$ is connected if and only if $\eta$ is a primitive $k$-differential.
If $\eta$ is a $d$-th power of a primitive $k/d$-differential, then $\hat X$ has $d$ isomorphic connected components.

For a twisted $k$-differential a similar cover can be constructed, but no longer uniquely.
Assume that the twisted $k$-differential $\eta$ is of type $\mu = (m_1, \dots, m_n)$, let $\hat m_i := (k+m_i)/\gcd(k,m_i)-1$ and let
\[
        \hat \mu := (\underbrace{\hat m_1, \dots, \hat m_1}_{\gcd(k, m_1)},
                \underbrace{\hat m_2, \dots, \hat m_2}_{\gcd(k, m_2)},
                \dots,
                \underbrace{\hat m_n, \dots, \hat m_n}_{\gcd(k, m_n)}).
\]
Moreover, let $\hat n := |\hat \mu|$ and $\hat g := \frac{1}{2} \sum_{\hat m_i \in \hat \mu} \hat m_i + 1$.

\begin{mythm}[\cite{BCGGM19}]
	For a pointed nodal curve $(X, \bs)$ with a twisted $k$-differential $\eta$ of type $\mu$, there exists a pointed nodal curve $(\hat X, \hat \bs)$ with a twisted abelian differential $\omega$ of type $\hat \mu$ such that
	\begin{enumerate}[(i)]
		\item $\pi : \hat X \to X$ is a cyclic cover of degree $k$ with deck transformation $\tau$,
		\item $\pi^\ast \eta = \omega^k$,
		\item $\tau^\ast \omega = \zeta \omega$ for a primitive $k$-th root of unity $\zeta$,
		\item marked points are mapped to marked points, i.e.\ $\pi(\hat \bs) = \bs$,
		\item $\pi$ is unramified outside of the nodes and marked points of $\hat X$,
		\item every node or marked point $q \in X_v$ has precisely $\gcd(k, \ord_q \eta_v)$ preimages.
	\end{enumerate}
\end{mythm}

\begin{mydef} \label{def:normalized_cover}
	We refer to a tuple $(\pi : \hat X \to X, \bs, \omega)$ as above as a \emph{normalized cover} of $(X, \bs, \eta)$.
	A normalized cover is called \emph{primitive} if $\hat X$ is connected.
\end{mydef}

\begin{myrem}
	\begin{enumerate}
		\item Condition (vi) in the above theorem is well-defined at nodes because the twisted $k$-differential $\eta$ is subject to the matching orders condition of Definition~\ref{def:twisted_k_differential}.
		\item If $\eta_v$ is a $d$-th power of a primitive $k/d$-differential, then the irreducible component $X_v$ has precisely $d$ isomorphic preimages.
		\item In general, the normalized cover is not unique: While the fibers $\pi |_{\pi^{-1}(v)} : \coprod_{\hat v} X_{\hat v} \to X_v$ are uniquely determined by the $k$-differential $\eta_v$, there may be a choice how to glue the different fibers along the nodes.
	\end{enumerate}
\end{myrem}

To determine the relation between the residues of the cover and the $k$-residues of the base curve, let us again consider a normalized cover of smooth curves with differentials $\pi : (\hat X, \omega) \to (X, \eta)$.
Let us fix a pole $p \in X$ of $\eta$ and let $q \in \pi^{-1}(p)$ be some preimage.
If $\pi$ is ramified at $p$, we claim that both the $k$-residue $\Res^k_p(\eta)$ and the residue $\Res_q(\omega)$ vanish.
For the $k$-residue, this is immediate by Definition~\ref{def:kRes}, and for the residue this is a consequence of the compatibility with the $\tau$-action as follows.
If $\pi$ is ramified at $q$, then there is an integer $1 < d \mid k$ such that $q$ is fixed by $\tau^d, \tau^{2d}, \dots, \tau^k$.
Because of $\tau^* \omega = \zeta \omega$, this implies that 
\[ \frac{k}{d} \cdot \Res_q(\omega) = \sum_{i=1}^{k/d} \zeta^{id} \Res_q(\eta) = 0, \]
as the $\frac{k}{d}$-th roots of unity sum to zero.
If the cover is unramified at $p$, then the $k$-residues of the twisted $k$-differential and the residues of the normalized cover are related as follows.

\begin{mylem} \label{lem:induced_k_residues}
	If $\pi$ is unramified at $p$, then
	\[
		\Res^k_p(\eta) = \big(\Res_q(\omega) \big)^k.
	\]
\end{mylem}
\begin{proof}
	Let $m := \ord_p \eta$ and note that $k$ divides $m$.
	Recall from~\eqref{eq:kdifflocal} that the $k$-form $\eta$ may locally be written as
	\[
		\begin{dcases*}
			\left(\frac{r}{z}\right)^k (dz)^k & if $m = -k$, \\
			\left(z^{m/k} + \frac{t}{z}\right)^k (dz)^k & if $m < -k$.
		\end{dcases*}
	\]
	In these cases, the $k$-residues are by definition $r^k$ and $t^k$, respectively.
	As $\pi$ is locally given by $\pi : z \mapsto z$, we get
	\[
                \omega^k = \pi^*\eta =
		\begin{dcases*}
			\left(\frac{r}{z} dz\right)^k & if $m = -k$ \\
			\left(\left(z^{m/k} + \frac{t}{z}\right) dz\right)^k & if $m < -k$.
		\end{dcases*}
	\]
	Thus the residues of $\omega$ at $q$ are $r$ and $t$, respectively.
\end{proof}

\begin{myrem}
	In general, the $k$-residue does not coincide with the coefficient in front of $z^{-k}$ in the Laurent expansion around the given pole.
	Moreover, there is nothing similar to the residue theorem for $k$-residues with $k > 1$.
\end{myrem}

\subsection{Enhanced level graphs}
\label{subsec:enhanced_level_graph}

The boundary points of the moduli space of multi-scale differentials $\bP \MSD[k]{g}{n}{\mu}$ are normalized covers subject to some conditions on the underlying dual graph of the stable curves.
We will recall here the necessary terminology to give the characterization of the boundary points in Section~\ref{subsec:multi-scale_k-differentials} below.

Let $G$ be a stable graph.
A \emph{full order} on $G$ is an order $\succcurlyeq$ on the vertices $V$ of $G$ that is reflexive, transitive, and such that for any $v_1, v_2 \in V$ at least one of the statements $v_1 \succcurlyeq v_2$ or $v_2 \succcurlyeq v_1$ holds.
If $v_1 \succcurlyeq v_2$ and $v_2 \succcurlyeq v_1$, we write $v_2 \asymp v_1$.
We call a function $\ell : V \to \bZ_{\leq 0}$ such that $\ell^{-1}(0) \neq \emptyset$ a \emph{level function}.
Note that a level function induces a full order on $G$ by setting $v \succcurlyeq w$ whenever $\ell(v) \geq \ell(w)$.
A \emph{level graph} $(G, \ell)$ is a graph $G$ together with a choice of a level function $\ell$.
When the level function is clear from context, we abuse notation and denote the level graph $(G, \ell)$ by $G$ as well.

For a given level $L$ we call the subgraph of $G$ that consists of all vertices $v$ with $\ell(v) > L$ along with the edges between them the \emph{graph above level $L$} of $G$, and denote it by $G_{>L}$.
We similarly define the graph $G_{\geq L}$ \emph{above or at level $L$}, and the graph $G_{=L}$ \emph{at level $L$}.
An edge $e \in E$ is called \emph{horizontal} if it connects two vertices of the same level, and it is called \emph{vertical} otherwise.
Given a vertical edge $e$, we denote by $v^+(e)$ and $v^-(e)$ the vertex that is its endpoint of higher and lower level, respectively.

Let $\pi : \hat G \to G$ be a morphism of graphs. By this we mean that $\pi$ maps vertices to vertices, edges to edges, and legs to legs while respecting edge--vertex and leg--vertex incidences. Assume further that $\pi$ is surjective on vertices and let $\succcurlyeq_G$ denote a full order on $G$. We get an induced full order on $\hat G$ by setting $v_1 \succcurlyeq_{\hat G} v_2$ if and only if $\pi(v_1) \succcurlyeq_G \pi(v_2)$. If $\succcurlyeq_G$ was induced by a level function $\ell$, then $\succcurlyeq_{\hat G}$ is induced by the \emph{lifted level function} $\hat \ell := \ell \circ \pi$.

In the following, given a twisted $k$-differential $(X, \bs, \eta)$ and a level graph $(G, \ell)$, we will always assume that $G$ is the dual graph of $X$.
We denote by $X_{>L}$ (resp.\ $X_{\geq L}$ resp.\ $X_{=L}$) the subcurve whose dual graph is $G_{>L}$ (resp.\ $G_{\geq L}$ resp.\ $G_{=L}$).

\begin{mydef} \label{def:k_cyclic_cover_of_graphs}
	Let $\pi : \hat G \to G$ be a morphism of graphs.  It is called \emph{cover of graphs} if $\pi$ is surjective on vertices, edges, and legs. Furthermore, it is called \emph{$k$-cyclic cover of graphs} if there is the additional data of an automorphism $\tau$ of $\hat G$ such that $\tau^k = \id$ and $\pi$ is the quotient map $\hat G \to \hat G / \tau$.
\end{mydef}

\begin{myrem}
	We would like to stress that morphisms (and covers) of graphs do not contract edges. Also note that in a $k$-cyclic cover of graphs the order of $\tau$ may in fact be $k' < k$ with $k'$ dividing $k$.
	We think of a $k$-cyclic cover of graphs as the dual graphs of a $k$-cyclic cover of curves. Hence the name $k$-cyclic.
\end{myrem}

The following definition is taken from \cite{BCGGM19}.

\begin{mydef} \label{def:GRC}
	Let $\pi : (\hat G, \hat \ell) \to (G, \ell)$ be a $k$-cyclic cover of level graphs.
	We say that a normalized cover of twisted $k$-differential $(\pi : \hat X \to X, \bs, \omega)$ is \emph{compatible} with $\pi$ if it satisfies the following two conditions.
	\begin{enumerate}[(i)]
		\setcounter{enumi}{3}
		\item \textbf{(Partial order)} If a node of $\hat X$ identifies $q_1 \in\hat X_{v_1}$ with $q_2 \in \hat X_{v_2}$, then $v_1 \succcurlyeq v_2$ if and only if $\ord_{q_1} \omega_{v_1} \geq -1$. In particular, $v_1 \asymp v_2$ if and only if $\ord_{q_1} \omega_{v_1} = -1$.
		\item \textbf{(Global residue condition, GRC)} For every level $L$ and every connected component $\hat Y$ of $\hat X_{>L}$ that does not contain a marked point with a prescribed pole the following condition holds:
			Let $q_1, \dots, q_b$ denote the set of all nodes where $\hat Y$ intersects $\hat X_{=L}$. Then
			\[
				\sum_{j=1}^b \Res_{q_j^-} \omega_{v^-(q_j)} = 0.
			\]
			where $q_j^- \in \hat X_{=L}$ is the point on the irreducible component corresponding to $v^-(q_j) \in \hat G_{=L}$ that is part of the node $q_j$.
	\end{enumerate}
\end{mydef}

Note that condition (iv) is equivalent to the analogous condition on the induced twisted $k$-differential $\eta$ on $X$: If a node of $X$ identifies $q_1 \in X_{v_1}$ with $q_2 \in X_{v_2}$, then $v_1 \succcurlyeq v_2$ if and only if $\ord_{q_1} \eta_{v_1} \geq -k$, and $v_1 \asymp v_2$ if and only if $\ord_{q_1} \eta_{v_1} = -k$. \medskip

Though not strictly necessary at the moment, it will be more convenient later on to consider enhanced level graphs instead of level graphs. Enhanced level graphs additionally carry the data of an integer valued function $o$ which should be thought of as an order at every node and marked point.

\begin{mydef} \label{def:enhanced_level_graph}
	Let $k \in \bN_{\geq 1}$.
	A \emph{$k$-enhanced level graph} $G^+ = (V,\; H,\; L,\; \iota,\; a,\; \ell,\; o)$ is a tuple where $(V,\; H,\; L,\; \iota,\; a)$ is a stable graph, the map $\ell : V \to \bZ_{\leq 0}$ is the \emph{level function} and the so-called \emph{enhancement} $o : H \cup L \to \bZ$ such that the following hold.
	\begin{enumerate}[(i)]
		\item The genus is well-defined, i.e.\ for all $v \in V$ there is a non-negative integer $g(v)$ such that
			\[
				k \big( 2g(v) - 2 \big) = \sum_{h \in a^{-1}(v)} o(h).
			\]
		We call $\mu(v) := (o(h))_{h \in a^{-1}(v)}$ the \emph{type of $v$}.
		\item The orders at edges match, i.e. for all $h \in H$ we have $o(h) + o(\iota(h)) = -2k$. \label{enum:ordersmatchedges}
		\item The orders at the half-edges are compatible with the level function, that is: for all $h \in H$ we have $o(h) \geq o(\iota(h))$ if and only if $\ell \big(a(h) \big) \geq \ell \big(a(\iota(h)) \big)$. \label{enum:ordersmatchlevels}
	\end{enumerate}
	Note that \eqref{enum:ordersmatchedges} and \eqref{enum:ordersmatchlevels} imply that the levels at both ends of an edge are equal if and only if the orders at both ends are $-k$.
	We call such an edge \emph{horizontal}.
	Any other edge is called \emph{vertical}.
\end{mydef}

\begin{mydef} \label{def:normalized_cover_enhanced_level_graph}
	Let $G^+ = (V,\; H,\; L,\; \iota,\; a,\; \ell,\; o)$ be a $k$-enhanced level graph.	A \emph{normalized cover} of $G^+$ is a triple $(\widehat G^+, \pi, \tau)$, where
	\begin{enumerate}[(i)]
		\item $\widehat G^+ = (\hat V,\; \hat H,\; \hat L,\; \hat \iota,\; \hat a,\; \hat \ell,\; \hat o)$ is an $1$-enhanced level graph,
		\item $\pi : \hat G^+ \to G^+$ is a cover of graphs such that
		\begin{enumerate}
			\item $\pi$ preserves the levels, i.e.\ $\hat \ell = \ell \circ \pi$,
			\item the order at the preimages is the expected one, i.e.\ for all half-edges and legs $h \in H \cup L$ and all $\hat h \in \pi^{-1}(h)$ it is
			\[
			\hat o(\hat h) + 1 = \frac{o(h) + k}{\gcd \big(o(h), k \big)}
			\]
			\item the number of preimages is the expected one, i.e.\ for all half-edges and legs $h \in H \cup L$ we have
			\[
			\big| \pi^{-1}(h) \big| = \gcd \big(o(h), k \big),
			\]
		\end{enumerate}
		\item $\tau : \hat G \to \hat G$ is a graph automorphism that exhibits $\pi$ as a $k$-cyclic cover of graphs.
	\end{enumerate}
\end{mydef}

Note that the genus of each vertex $\hat v \in \hat V$ is an integer by definition of an $1$-enhanced level graph.

\begin{mydef} \label{def:GRC2}
	Let $\pi : \hat G^+ \to G^+$ be a $k$-cyclic cover of enhanced level graphs.
	We say that a normalized cover of a twisted $k$-differential $(\pi : \hat X \to X, \bs, \omega)$ is \emph{compatible} with $\pi$ if it is compatible with the underlying cover of level graphs $\pi : (\hat G, \hat \ell) \to (G, \ell)$ and the orders of the differentials coincide with the enhancements.
\end{mydef}

\subsection{Multi-scale \texorpdfstring{$k$}{k}-differentials and the characterization of limit points}
\label{subsec:multi-scale_k-differentials}

The points in the boundary of $\bP \MSD[k]{g}{n}{\mu}$ may be described as suitable equivalence classes of the following. 

\begin{mydef} \label{def:multi-scale_k-diff}
	A \emph{multi-scale $k$-differential of type $\mu$} on a pointed stable curve $(X, \bs)$ consists of the following data.
	\begin{enumerate}[(i)]
		\item A primitive normalized cover of a twisted $k$-differential $(\pi : \hat X \to X, \bs, \omega)$ of type $\mu$.
		\item A compatible $k$-cyclic cover of enhanced level graphs $\pi : \hat G^+ \to G^+$.
		\item A prong-matching for each node of $X$ joining components on non-equal levels.
	\end{enumerate}
\end{mydef}

A prong-matching roughly represents a choice of gluing the differentials at the nodes of the curve.
While it is needed to get a well-behaved compactification, it will be of no importance to us and we will suppress it in the following. Similarly, the definition of the equivalence relation of multi-scale $k$-differentials is rather intricate and not relevant for our treatment.

\begin{mythm}[{\cite{CMZ19}}]
	The points in the moduli space of multi-scale $k$-differentials $\bP \MSD[k]{g}{n}{\mu}$ are precisely the $\bC^\times$-equivalence classes of multi-scale $k$-differentials $(\pi : \hat X \to X, \bs, \omega, \pi : \hat G^+ \to G^+)$ of type $\mu$.
\end{mythm}

Note that the tuple $(\pi : \hat X \to X, \bs, \omega, \pi : \hat G^+ \to G^+)$ is equivalent to the tuple $(\tau \curvearrowright \hat X, \hat \bs, \omega, \tau \curvearrowright G^+)$, where $\hat \bs$ is the lift of $\bs$ to $\hat X$.
We give another version of the same theorem that highlights the possible scaling parameters of one-parameter families approaching the boundary.

Suppose that $S$ is the spectrum of a discrete valuation ring $R$ with residue field $\bC$, whose maximal ideal is generated by $t$.
Let $\hat \cX/S$ be a family of semi-stable curves with smooth generic fiber $\hat X$ and special fiber $\hat X_0$ and such that there is an automorphism $\boldsymbol \tau$ of degree $k$ on the family $\hat \cX/S$.
Let $\boldsymbol{\omega}$ be a section of the $\zeta_k$-eigenspace (with respect to $\boldsymbol \tau$) of $\omega_{\cX/S}$ of type $\hat \mu = (\hat m_1, \dots, \hat m_{\hat n})$ whose divisor is given by the sections $\hat \bs = (\hat s_1, \dots, \hat s_{\hat n})$ with multiplicity $\hat m_i$.
If moreover $(\hat \cX/S, \hat \bs)$ is stable, then the tuple $(\hat \cX/S, \boldsymbol \tau, \hat \bs, \boldsymbol \omega)$ is called a \emph{pointed family of stable $k$-differentials}.
(Note that $(X := \hat X / \tau, \, \eta := (\boldsymbol \omega|_{\hat X})^{\otimes k} / \tau)$ is in fact a $k$-differential, where $\boldsymbol \omega|_{\hat X}$ is the restriction of $\boldsymbol \omega$ to the generic fiber $\hat X$.)
We define the \emph{scaling factor $\hat \ell(\hat v)$} of a vertex $\hat v$ of the dual graph $\hat G$ of $\hat X_0$ as the non-positive integer such that the restriction of the meromorphic differential $t^{-\hat \ell(\hat v)} \cdot \boldsymbol \omega$ to the component $\hat X_{0,\hat v}$ of the special fiber corresponding to $\hat v$ is a well-defined and generically nonzero differential $\omega_{\hat v}$ on $\hat X_{0,\hat v}$.
The $\omega_{\hat v}$ are called the \emph{scaling limits} of $\boldsymbol \omega$.

\begin{mythm}[\cite{CMZ19}] \label{thm:BCGGM_main_theorem}
	If $(\hat \cX/S, \boldsymbol \tau, \hat \bs, \boldsymbol \omega)$ is as above, then the function $\hat \ell$ defines a full order on the dual graph $\hat G$ of the special fiber $\hat X_0$ and the collection of the $\omega_{\hat v}$ is a normalized cover of a twisted $k$-differential of type $\hat \mu$ compatible with the level function $\hat \ell$.

	Conversely, suppose that $\hat X_0$ is a stable $\hat n$-pointed curve with dual graph $\hat G$ and a degree~$k$ automorphism~$\tau$.
	Moreover, suppose that $\omega = \{\omega_{\hat v}\}_{\hat v \in \hat V}$ is a normalized cover of a twisted $k$-differential of type $\mu$ in the $\zeta_k$-eigenspace of $\tau$ and compatible with a full order on $\hat G$.
	Then for every level function $\hat \ell : \hat G \to \bZ$ defining the full order on $\hat G$ and for every $\tau$-invariant assignment of integers $n_{\hat e}$ to horizontal edges there is a stable family $\hat \cX/S$ over $S = \Spec\bC[[t]]$ with smooth generic fiber and special fiber $\hat X_0$ that satisfies the following properties.
	\begin{enumerate}[(i)]
		\item The action of $\tau$ extends to a degree $k$ automorphism on $\hat \cX /S$.
		\item There exists a global section $\boldsymbol \omega$ of the relative dualizing sheaf $\omega_{\cX/S}$ whose horizontal divisor $\div_{\hor}(\boldsymbol \omega) = \sum_{i=1}^{\hat n} \hat m_i \Sigma_i$ is of type $\hat \mu$ and whose scaling limits are the collection $\{\omega_{\hat v}\}_{\hat v \in \hat V}$.
			Moreover, the restriction of $\boldsymbol \omega$ to each fiber is contained in the $\zeta_k$-eigenspace of $\tau$.
		\item The intersections $\Sigma_i \cap \hat X_0 = \{\hat s_i\}$ are smooth points of the special fiber and $\omega$ has a zero or pole of order $\hat m_i$ in $\hat s_i$.
		\item There exists a positive integer $N$ such that near every node $\hat e$ a local equation for $\cX$ in terms of the uniformizer $t$ is given by
			\[
				xy = \begin{dcases*}
					t^{Nn_{\hat e}} & if $\hat e$ is a horizontal edge, \\
					t^{N \big(\hat \ell(v^+(\hat e))-\hat \ell(v^-(\hat e)) \big)} & if $\hat e$ is a vertical edge.
				\end{dcases*}
			\]
	\end{enumerate}
\end{mythm}
\begin{proof}
	The first statement is proved in \cite{CMZ19}. Note that the arguments given there hold over any discrete valuation ring with residue field $\bC$.

	For the second statement we recall from the proof of \cite[Theorem~5.2]{MUW17_published} that there are no constraints for the plumbing fixtures used on horizontal nodes, whereas for a vertical node corresponding to an edge $\hat e$ the level function $\hat \ell_0$ on the cover used for plumbing has to satisfy the condition
	\begin{equation}
		\label{eq:divcover}
		(\ord_{v^+(\hat e)} \hat \omega + 1) \ \big\vert \  \big(\hat \ell_0 (v^+(\hat e)) - \hat \ell_0 (v^-(\hat e))\big).
	\end{equation}
	Multiplying the prescribed function $\hat \ell$ by a sufficiently divisible $N$, the resulting level function $\hat \ell_0 = N \cdot \hat \ell$ satisfies the divisibility property.
\end{proof}

\subsection{Empty primitive strata}
\label{subsec:empty_strata}

The primitive strata of $k$-differentials $\Omega^k \cM_{g}(\mu)^{\prim}$ are empty for some types $\mu$.
To keep the notation concise, we will denote a type $\mu = (m_1, \dots, m_n)$ where multiple $m_i$ agree with exponential notation, e.g.\ we will denote the type $(0, \dots, 0)$ by $(0^n)$.

\begin{mythm}[\cite{GT21}, \cite{GT_quadratic}, \cite{GT_kdiff}] \label{thm:empty_primitive_strata}
	The primitive stratum $\Omega^k \cM_{g}(\mu)^{\prim}$ (and hence the stratum of multi-scale $k$-differentials $\PMSD[k]{g}{n}{\mu}$) is empty if and only if we are in one of the following case.
	\begin{enumerate}[(i)]
		\item $k = 1$ and $\mu = (-1, m_2, \dots, m_n)$ with $m_2, \dots, m_n \geq 0$.
		\item $g = 0$ and $\mu = (m_1, \dots, m_n)$ with $\gcd(m_1, \dots, m_n, k) \neq 1$.
		\item $g = 1$ and $\mu = (0^{n-2},-1,1)$.
		\item $g=1$, $k \geq 2$ and $\mu = (0^n)$.
		\item $g = 2$, $k = 2$ and $\mu = (0^{n-1},4)$ or $\mu = (0^{n-2},1,3)$.
	\end{enumerate}
\end{mythm}

\begin{myrem}
	A stratum $\Omega^k \cM_g(\mu)$ may be nonempty even if its primitive part $\Omega^k \cM_{g}(\mu)^{\prim}$ is empty: For each $d \mid k$, the stratum $\Omega^k \cM_g(\mu)$ may have nonempty connected components that parametrize $d$-th powers of primitive $k/d$-differentials.
\end{myrem}

\subsection{The image of the residue map}
\label{subsec:image_residue_map}

In Section~\ref{sec:realizability_locus}, when we prove Theorem~\ref{intro:thm:main_theorem}, we will first translate a given tropical normalized cover into a normalized cover of enhanced level graphs $\pi : \hat G^+ \to G^+$ and then try to construct a normalized cover of a twisted $k$-differential $(X, \eta)$ that is compatible with $\pi$.
If this is possible, then the irreducible component $(X_v, \eta_v)$ is a (possibly meromorphic) $k$-differential for each $v \in V(G^+)$, whose type $\mu(v)$ is prescribed by the enhancements of $G^+$.
To construct the twisted $k$-differential, we will need to fix the $k$-residues at the poles of $\eta_v$.
The question which $k$-residues are valid choices was answered by Gendron-Tahar.
We will summarize their results in this section.

\begin{mydef} \label{def:reduced_type}
	For a type $\mu = (m_1, \dots, m_n)$, we define the \emph{reduced type} $\mu_\red$ as the sub-tuple of $\mu$ consisting of all nonzero entries.
	Following~\cite{GT_kdiff}, we denote this tuple by
	\[
	\mu_\red = (a_1, \dots, a_t; -b_1, \dots, -b_p; -c_1, \dots, -c_r; -k^s)
	\]
	where the $a_i > -k$, the $b_i \in k \bN_{>1}$ are the poles where the order is greater than $k$ and divisible by $k$, and the $c_i \in \bN_{> k} \setminus (k \bN)$ are the poles with order not divisible by $k$.
	As above, the power $-k^s$ indicates that there are $s$ poles of order $k$.
\end{mydef}

Recall that the $k$-residues at the poles with orders $-c_i$ are zero, while the $k$-residues at the poles with orders $-k$ cannot be zero.
For $k\geq 2$ we let
\[
	\Res^k_g(\mu_\red) : \Omega^k \cM_g(\mu_\red)^{\prim} \longrightarrow \bC^p \times (\bC^\times)^s
\]
denote the residue map.
For $k=1$, the possible residues are restricted by the residue theorem and hence in this case we define the reisude map as
\[
	\Res^1_g(\mu_\red) : \Omega \cM_g(\mu_\red) \longrightarrow
	\left\{(r_1,\dots,r_{p+s})\in \bC^p \times (\bC^\times)^s \; \middle| \; \sum_{i=1}^{p+s}r_i = 0\right\}.
\]
For almost all reduced types $\mu_\red$, the residue map is surjective.
In the rest of this section, we will discuss all the cases where the residue map is not surjective.
The following proposition lists all those cases where exactly the origin is missing in the image of the residue map.

\begin{myprop}[\cite{GT21}, \cite{GT_quadratic}, \cite{GT_kdiff}]
	\label{prop:residue_map_origin_missing}
	In the following cases, precisely the origin is missing from the image of the residue map.
	\begin{enumerate}[(i)]
                \item If $k = 1$, $g=0$, $s=0$ and there exists an index $i$ such that the inequality
                        \[
                                    a_i > \sum_{j=1}^p b_j - (p+1)
                        \]
                        holds.
                        (Note that $k=1$ implies $r=0$.)                    
		\item If $k=2$, $g=1$ and $\mu_\red = (4p;(-4^p))$ or $\mu_\red = (2p-1,2p+1;(-4^p))$ for $p \in \bN^\times$.
		\item If $k \geq 2$, $g=0$, $\mu_\red = (a_1, \dots, a_t; -b_1, \dots, -b_p; -c_1)$ and there is at most one $a_i$ not divisible by $k$ and $\sum_{k \, \mid \, a_i} a_i < kp$.
		\item If $k \geq 2$, $g = 0$, $\mu_\red = (a_1, \dots, a_t; -b_1, \dots, -b_p)$, $p \neq 0$ and \emph{none} of the following holds:
			\begin{enumerate}
				\item $p = 1$ and $t \geq 3$,
				\item $p \geq 2$, $t \geq 3$ and there exist at least three $a_i$ not divisible by $k$,
				\item $p \geq 2$, $t \geq 3$ and there exist precisely two $a_i$ not divisible by $k$ and $\sum_{k \, \mid \, a_i} a_i \geq kp$,
                                \item $k=2$ and $\mu_\red = ((2p+b-5)^2;-b,-b-2,(-4^{p-2}))$ or $\mu_\red = (2p+b-7, 2p+b-5;(-b^2),(-4^{p-2}))$ for $p \geq 2$ and even $b \geq 4$.
			\end{enumerate}
	\end{enumerate}
\end{myprop}

In the cases of the following two propositions not only the origin, but a finite number of $\bC$-lines is missing from the image of the residue map.

\begin{myprop}[\cite{GT_quadratic}, \cite{GT_kdiff}]
	\label{prop:residue_map_line_missing}
	For the reduced types $\mu_\red$ in Figure~\ref{fig:residue_map_line_missing}, precisely the $\bC$-lines spanned by the vectors $w_i$ are missing from the image of the residue map, i.e.
	\[
		\Im \big(\Res^k_g(\mu_\red) \big) = 
		\big( \bC^p \times (\bC^\times)^s \big) \setminus \bigcup_i \langle w_i \rangle_{\bC}.
	\]
	(Note that if there are multiple poles with the same order in $\mu_\red$, then the order of the entries of the vectors $w_i$ may not be uniquely determined. In those cases all possible permutations need to be taken into account.)
\end{myprop}

\begin{figure}[ht]
	\centering
	\begin{minipage}{0.9\textwidth}
		\centering
		\renewcommand{\arraystretch}{1.2}
		\begin{tabular}{r r l | l}
			&& $\mu_\red$ & $w_i$ \\\hline
			\multirow{10}{*}{\rotatebox[origin=c]{90}{$k=2$}} & \multirow{2}{*}{\rotatebox[origin=c]{90}{$g=1$}} & $(2s;(-2^s))$ for $s \in 2\bN^\times$ & $(1,\dots,1)$ \\\cline{3-4}
			&& $(s-1,s+1;(-2^s))$ for $s \in 2\bN^\times$ & $(1,\dots,1)$ \\\cline{2-4}
			& \multirow{8}{*}{\rotatebox[origin=c]{90}{$g=0$}} & $(s-1,s+1;-4;(-2^s))$ for $s \in 2\bN^\times$ & $(0;1,\dots,1)$  \\\cline{3-4}
			&& $(2p-1,2p+1;(-4^p);(-2^2))$ for $p \geq 0$ & $(0, \dots, 0;1,1)$  \\\cline{3-4}
			&& $(s^2;-4;(-2^s))$ for $s \in (\bN^\times \setminus 2\bN)$ & $(1;1,\dots,1)$  \\\cline{3-4}
			&& $((2p-1)^2;(-4^p);-2)$ for $p \geq 1$ & $(1, 0, \dots, 0;1)$  \\\cline{3-4}
			&& \begin{tabular}{@{}l@{}}$((2p+b-5)^2;-b,-b-2,(-4^{p-2}))$ \\ \hspace{1cm} for $p \geq 2$ and even $b \geq 4$\end{tabular} & $(1,1,0, \dots ,0)$  \\\cline{3-4}
			&& \begin{tabular}{@{}l@{}}$(2p+b-7, 2p+b-5;(-b^2),(-4^{p-2}))$ \\ \hspace{1cm} for $p \geq 2$ and even $b \geq 4$\end{tabular} & $(1,1,0, \dots ,0)$  \\\hline
			\multirow{6}{*}{\rotatebox[origin=c]{90}{$k=3$}} & \multirow{6}{*}{\rotatebox[origin=c]{90}{$g=0$}} & $(-1,4;(-3^3))$ & $(1^3)$ \\\cline{3-4}
			&& $(1,2;(-3^3))$ & $(1^3)$ \\\cline{3-4}
			&& $(2,4;(-3^4))$ & $(1^2,-1^2)$ \\\cline{3-4}
			&& $(2,7;(-3^5))$ & $(1^4,-1)$ \\\cline{3-4}
			&& $(2,10;(-3^6))$ & $(1^6)$ \\\cline{3-4}
			&& $(5,7;(-3^6))$ & $(1^6)$ \\\hline
			\multirow{4}{*}{\rotatebox[origin=c]{90}{$k=4$}} & \multirow{4}{*}{\rotatebox[origin=c]{90}{$g=0$}} & $(-1,5;(-4^3))$ & $(1^2,-4)$ \\\cline{3-4}
			&& $(3,5;(-4^4))$ & $(1^4)$ \\\cline{3-4}
			&& $(-1,9;(-4^4))$ & $(1^4)$ \\\cline{3-4}
			&& $(3,13;(-4^6))$ & $(1^6)$ \\\hline
			\multirow{2}{*}{\rotatebox[origin=c]{90}{$k=6$}} & \multirow{2}{*}{\rotatebox[origin=c]{90}{$g=0$}} & $(-1,7;(-6^3))$ & $(1^3)$ \\\cline{3-4}
			&& $(-1,13;(-6^4))$ & $(1^4)$ \\\hline
			\rotatebox[origin=c]{90}{$k \geq 2$} & \rotatebox[origin=c]{90}{$g=0$} & $(-1,1;(-k^2))$ & $(1,(-1)^k)$
		\end{tabular}
		\caption{Reduced types $\mu_\red$ and generators $w_i$ of the $\bC$-lines missing in the image of the residue map.}
		\label{fig:residue_map_line_missing}
	\end{minipage}
\end{figure}

\begin{myprop}[\cite{GT21}]
	\label{prop:residue_map_abelian_coprime}
        For $k=1$, $g=0$ and $\mu_\red = (a_1, \dots, a_t; (-1^s))$ with $s \geq 2$, precisely those $\bC$-lines $\big\langle (x_1, \dots, x_{s_1}, -y_1, \dots, -y_{s_2}) \big\rangle_{\bC}$ are missing from the image of the residue map for which the $x_i,y_j \in \bN$ are pairwise relatively prime and
        \[
                \sum_{i=1}^{s_1} x_i = \sum_{j=1}^{s_2} y_j \leq \max(a_1, \dots, a_t).
        \]
\end{myprop}

Finally, there are some cases where a finite number of at most 2-dimensional subvarieties is missing from the image of the residue map.
For $k=2$, following \cite[D\'efinition~1.8]{GT_quadratic} we call three numbers $R_1, R_2, R_3 \in \bC^\times$ \emph{triangular}, if there exist square roots $r_1, r_2, r_3$ of $R_1, R_2, R_3$ such that $r_1 + r_2 + r_3 = 0$.

\begin{myprop}[\cite{GT_quadratic}]
	\label{prop:residue_map_triangular}
	For $k=2$, $g=0$ and $\mu_\red = (a_1, \dots, a_t; (-2^s))$, precisely the following $\bC$-lines are not in the image of the residue map.
	\begin{enumerate}[(i)]
		\item \label{missing_planes}
		For $\mu_\red = (2s'-1,2s'+1;(-2^{2s'+2}))$ with $s' \in \bN$ the lines spanned by $(1, \dots, 1, R, R)$ for $R \in \bC^\times$ are missing.
		\item 
		\label{missing_cones}
		For $\mu_\red = ((2s'-1)^2; (-2^{2s'+1}))$ with $s' \in \bN^\times$ the lines spanned by $(R_1, R_2, R_3, \dots, R_3)$ for triangular $R_i \in \bC^\times$ are missing.
		\item If precisely two $a_i$ are odd (say $a_1$ and $a_2$), the lines spanned by $(r_1^2, \dots, r_s^2)$ for relatively prime $r_i \in \bN$ and such that for $S := \sum_i r_i$ either
			\begin{enumerate}
				\item $S$ is odd and $S < \max(a_1, a_2) + 2$ or
				\item $S$ is even and $S < a_1 + a_2 + 4$
			\end{enumerate}
			are missing.
	\end{enumerate}
\end{myprop}

\begin{mythm}[\cite{GT21}, \cite{GT_quadratic}, \cite{GT_kdiff}]
	For $k \geq 1$, the residue map is surjective in all cases not covered by Propositions \ref{prop:residue_map_origin_missing}, \ref{prop:residue_map_line_missing}, \ref{prop:residue_map_abelian_coprime} and \ref{prop:residue_map_triangular}.
\end{mythm}

\section{Reduction to realizability of normalized covers}
\label{sec:reduction}

The goal of this section is to reduce the realizability problem for curves with pluri-canonical divisor to the realizability problem for normalized covers, i.e. we want to formally state and prove Corollary~\ref{intro:cor:reduction}. To do so, we define a notion of tropical normalized cover in Definition~\ref{def:tropical_normalized_cover} which should be thought of as a tropical version of Definition \ref{def:normalized_cover}. Essentially, we require a tropical normalized cover to be a tropical unramified $\bZ/k\bZ$-cover $\hat\Gamma \to \Gamma$ in the sense of \cite{LUZ24} with the additional property that the underlying cover of graphs admits the structure of a normalized cover of enhanced level graphs in the sense of Definition \ref{def:normalized_cover_enhanced_level_graph}. Of course, the enhancement should be compatible with the divisor marked by the legs of $\Gamma$. More precisely,  Lemma~\ref{lem:construction_of_k-enhancement} endows a tropical curve with effective pluri-canonical divisor with the structure of a $k$-enhanced level graph. In the definition of a tropical normalized cover we will then require the structure of normalized cover of enhanced level graphs to coincide on $\Gamma$ with the output of Lemma~\ref{lem:construction_of_k-enhancement}. 

Once the notion of tropical normalized cover is introduced, we can construct the moduli space of tropical normalized covers $\tropMSD$ and the tropicalization map
\[ \trop_{\Xi^k} : \bP \Omega^k \cM_{g, n}^\an \longrightarrow \tropMSD \]
which is really a special case of the map defined in \cite[Section~3.3]{LUZ24}.
Finally, we prove Theorem~\ref{intro:thm:tropMSD} and perform the reduction step.

\subsection{From rational functions to $k$-enhanced level graphs}

So far we have considered tropical curves $\Gamma$ together with an effective pluri-canonical divisor $D \in |kK_\Gamma|$ on one hand and on the other hand $k$-enhanced level graphs $G^+$.
These notions are related by

\begin{mylem} \label{lem:construction_of_k-enhancement}
	Let $\Gamma$ be a tropical curve and $D = k K_\Gamma + (f) \in |k K_\Gamma|$ an effective pluri-canonical divisor. Let $G$ be the minimal graph model of $\Gamma$ where $D$ is represented with legs. We can associate a natural $k$-enhanced level graph structure $G^+ = G^+(f)$ on $G$.
\end{mylem}

\begin{proof}
	Let $G$ be the minimal model of $\Gamma$ including $D(p)$ legs for every point $p \in \supp D$. Note that $G$ is always stable, even when we are subdividing an edge of the minimal model of $\Gamma$ at an interior point: the fact that we are introducing legs to represent support points of $D$ ensures that the new vertex will be at least trivalent. To define the enhanced level graph structure, we first endow $G$ with the total order given by $f$, i.e.\ for vertices $v$ and $w$ we set $v \preccurlyeq w$ if and only if $f(v) \leq f(w)$. 	
	Next we define a $k$-enhancement $o : H \cup L \to \bZ$ compatible with the total order by the following rule
	\begin{equation} \label{eq:def_enhancement}
		o(h) := \begin{cases}
			1 & \text{if } h \text{ is a leg} \\
			-s(h) - k & \text{if } h \text{ is part of an edge},
		\end{cases}
	\end{equation}
	where $s(h)$ is the outgoing slope of $f$ on the half-edge $h$. It is easy to check Definition~\ref{def:enhanced_level_graph}.
\end{proof}

\begin{myrem} \label{rem:realization_in_different_strata}
	Our choice to represent the effective divisor $D$ by attaching $D(p)$ legs at each point $p \in \Gamma$ and endowing these with $o$-value 1 amounts to seeking realizations in the principal stratum $\bP \Omega^k \cM_g(1, \ldots, 1)$ in the end. Of course one could also ask for realizability in other strata -- the definition in \eqref{eq:def_enhancement} would then have to be adapted.
\end{myrem}

\subsection{Tropical $k$-cyclic Hurwitz covers}
\label{sec:tropical_Hurwitz_covers}

A tropical Hurwitz cover is a harmonic morphism of tropical curves satisfying the local Riemann-Hurwitz conditions. The moduli space of such maps of fixed degree $d$ and ramification profile $\xi$ was introduced in \cite{CMR16}. For our purposes we need the slightly modified notion of tropical $k$-cyclic Hurwitz covers. See \cite{LUZ24} and the references therein for an overview of previous appearances of such objects in the literature. 

\begin{mydef}
	Let $\Gamma'$ and $\Gamma$ be tropical curves. A morphism of metric graphs $\varphi : \Gamma' \to \Gamma$ is called \emph{morphism of tropical curves} if it maps edges of $\Gamma'$ linearly to edges of $\Gamma$ such that the ratio of edge lengths $d_{e'} := \frac{l(\varphi(e'))}{l(e')}$ is an integer for every edge $e'$ of $\Gamma'$. In this case the numbers $d_{e'}$ are called \emph{expansion factors}.
\end{mydef}

\begin{mydef}
	Let $\varphi : \Gamma' \to \Gamma$ be a morphism of tropical curves, $p' \in \Gamma'$ and $p = \varphi(p')$. Then $\varphi$ is called \emph{harmonic at $p'$} if for every tangent direction $\epsilon \in T_p(\Gamma)$ to $p$ in $\Gamma$ the value of the \emph{local degree}
	\[ d_{p'} := \sum_{\substack{\epsilon' \in T_{p'}(\Gamma') \\ \epsilon' \mapsto \epsilon}} d_{\epsilon'} \]
	does not depend on $\epsilon$. Here the sum is running over all tangent directions to $p'$ that map to $\epsilon$. A morphism is \emph{harmonic} if it is surjective and harmonic at every $p' \in \Gamma'$. In this case the number $d = \sum_{p' \in f^{-1}(p)} d_{p'}$ is independent of $p$ and is called \emph{degree of $\varphi$}.
\end{mydef}

\begin{mydef}
	A harmonic morphism $\varphi : \Gamma' \to \Gamma$ is called \emph{tropical Hurwitz cover} if for every $p' \in \Gamma'$ the \emph{local Riemann-Hurwitz condition} holds, i.e.
	\[ 2 - 2g(p') = d_{p'}\big(2 - 2g(\varphi(p'))\big) - \sum_{\substack{h' \text{ half-edge} \\ \text{incident to } p'}} (d_{h'} - 1) \ . \]
\end{mydef}

We remark that tropical Hurwitz covers are not covers in the sense of topology, i.e.\ they are not local isomorphisms in general.

\begin{mydef}
	Let $\pi : \Gamma' \to \Gamma$ be a tropical Hurwitz cover and let $k \geq 1$ be an integer. We call an automorphism of metric graphs $\tau : \Gamma' \to \Gamma'$ a \emph{(tropical) deck transformation} if it is an isometry and $\pi$ is $\tau$-invariant. The data of $\pi$ together with $\tau$ is called \emph{tropical $k$-cyclic Hurwitz cover} if 
	\begin{enumerate}[(i)]
		\item $\pi$ is of degree $k$,
		\item the morphism of graphs underlying $\pi$ is a $k$-cyclic cover of graphs in the sense of Definition~\ref{def:k_cyclic_cover_of_graphs} (with deck transformation given by the morphism of graphs underlying $\tau$), and
		\item $\Gamma = \Gamma' / \tau$.
	\end{enumerate}
\end{mydef}

\begin{myrem}
	In a tropical $k$-cyclic Hurwitz cover the deck transformation $\tau$ satisfies necessarily $\tau^k = \id$. This however does not mean that $\tau$ is of degree $k$. Rather we only have that the degree of $\tau$ divides $k$. Put differently: we merely store the data of a generator of $\bZ/k\bZ$ acting (not necessarily freely) on $\hat \Gamma$.
\end{myrem}

We will now construct the moduli space of tropical $k$-cyclic Hurwitz covers in analogy to the moduli space of tropical Hurwitz covers (see \cite[Section 3.2]{CMR16}). This construction has been alluded to in \cite[Section~3.3]{LUZ24}. It is also related the cone over the so-called symmetric $\Delta$-complexes which have been studied in \cite{BCK23} or \cite{MHRSY24}. The construction follows the same pattern as the construction of $M_{g,n}^\trop$ in Section~\ref{subsec:moduli_tropical_cuves}. 

Fix a degree $k$, genera $g'$ and $g$, and a tuple of ramification profiles $\xi = (\xi_1, \ldots, \xi_n)$, i.e.\ each $\xi_i$ is a partition of $d$. 
Consider a tropical $k$-cyclic Hurwitz cover $(\pi : \Gamma' \to \Gamma, \tau)$ with the specified parameters. In particular, $\Gamma$ has to have precisely $n$ legs $l_1, \ldots, l_n$ such that the leg $l_i$ has $|\xi_i|$ preimages with expansion factors given by the entries of $\xi_i$. 
Then the \emph{combinatorial type} of $\pi$ is the underlying $k$-cover of weighted graphs $(\pi : G' \to G, \tau)$ together with the data of all the expansion factors. Here we denote by $G'$ and $G$ the minimal graph models for $\Gamma'$ and $\Gamma$ respectively. We describe a category $\cJ_{g' \to g, k}(\xi)$ as follows. Objects are combinatorial types. Morphism are commutative diagrams of the form
\begin{equation*}
	\begin{tikzcd}
		G'_1 \arrow[rr, "f'"] \arrow[dd, "\pi_1"] \arrow[rd, "\tau_1"] && G'_2 \arrow[dd, near end, "\pi_2"'] \arrow[rd, "\tau_2"] & \\
		& G'_1 \arrow[dl, "\pi_1"] \arrow[rr, crossing over, near start, "f'"] && G'_2 \arrow[dl, "\pi_2"] \\
		G_1 \arrow[rr, "f"] && G_2 &
	\end{tikzcd}
\end{equation*}
where the maps $f'$ and $f$ are either graph automorphism respecting expansion factors or $f$ is an edge contraction. Now associate to each combinatorial type $p$ the rational polyhedral cone $\sigma_p := \bR_{\geq 0}^{E(G)}$. This cone parametrizes the set of tropical $k$-cyclic Hurwitz covers with underlying combinatorial type $p$ (note that edge lengths on $G$ determine edge lengths on $G'$). Finally the moduli space is defined as
\[ H_{g' \to g,k}^\trop(\xi) := \varinjlim\limits_{\cJ_{g' \to g, k}(\xi)} \sigma_p.\] 

In \cite[Definition~25]{CMR16} the authors define a tropicalization from the analytification of Hurwitz space to the tropical Hurwitz space. A modified version of this construction in the presence of a group action was introduced in \cite[Section~3.3]{LUZ24}. Given a Hurwitz cover $X' \to X$ defined over a non-Archimedean field, its tropicalization is a tropical Hurwitz cover $\Gamma' \to \Gamma$ where $\Gamma'$ is the tropicalization of $X'$ in the sense of curves and $\Gamma$ the tropicalization of $X$. More precisely, the Hurwitz space $\compact{\mathcal{H}}_{g' \to g, k}(\xi)$  comes with two natural forgetful maps,
\begin{align*}
\tgt :& \compact{\mathcal{H}}_{g' \to g, k}(\xi) \longrightarrow \compact{\cM}_{g,n}, &\quad (X' \to X) &\longmapsto X \\
\src :& \compact{\mathcal{H}}_{g' \to g, k}(\xi) \longrightarrow \compact{\cM}_{g',n'}, &\quad (X' \to X) &\longmapsto X'
\end{align*}
called \emph{target} and \emph{source map} respectively. We abuse notation and denote the corresponding maps on $H_{g'\to g, k}^\trop(\xi)$ with the same symbols. By \cite[Theorem 4]{CMR16} tropicalization (in the sense of \cite[Definition 25]{CMR16}) and source (resp. target) map commute. 
We now define the tropicalization map for the moduli space of $k$-cyclic Hurwitz covers 
\[ \trop_H : \mathcal{H}_{g' \to g, k}(\xi)^\an \longrightarrow H^\trop_{g' \to g, k}(\xi) \]
simply by tropicalizing any $X' \to X$ and adding the induced tropical deck transformation $\tau$ to the data. With this definition, compatibility with source and target map remains true (compare \cite[Theorem~3.5]{LUZ24}).

\begin{myprop} \label{prop:tropicalization_Hurwitz_commutes_source_target}
	Set $n' := \sum |\xi_i|$. Then the following diagram commutes.
	\begin{equation*}
		\begin{tikzcd}
			\mathcal{H}_{g' \to g, k}(\xi)^\an \arrow[dd, "\src^\an"] \arrow[rd, "\trop_H"] \arrow[rr, "\tgt^\an" ]&& \cM^\an_{g, n} \arrow[d, "\trop"] \\
			& H^\trop_{g' \to g, k}(\xi) \arrow[d, "\src"] \arrow[r, "\tgt"] & M_{g, n}^\trop \\
			\mathcal{M}_{g',n'}^\an \arrow[r, "\trop"] & M_{g',n'}^\trop \ . & 
		\end{tikzcd}
	\end{equation*}
\end{myprop}

\begin{proof}
	By definition $\trop_H$ commutes with forgetting the deck transformation. The claim then follows from \cite[Theorem 4]{CMR16}.
\end{proof}

\subsection{Tropical normalized covers}

From now on fix the tuple of ramification profiles to be $\xi = ((k), \ldots, (k))$ with $n = k(2g-2)$ entries.

\begin{mydef} \label{def:tropical_normalized_cover}
	Let $(\pi : \hat\Gamma \to \Gamma, \tau)$ be a tropical $k$-cyclic Hurwitz cover with ramification profiles $\xi = ((k), \ldots, (k))$. Assume that the effective divisor $D \in \Div(\Gamma)$ given by the legs of $\Gamma$ is pluri-canonical, i.e.\ $D = kK_\Gamma + (f)$. 
	Let $\Gamma^+$ denote the $k$-enhanced level graph structure on $\Gamma$ induced by Lemma \ref{lem:construction_of_k-enhancement}. 
	We say that $\pi$ is a \emph{tropical normalized cover} if $\hat \Gamma$ admits a structure $\hat \Gamma^+$ of an 1-enhanced level graph such that $\hat \Gamma^+ \to \Gamma^+$ is a normalized cover of enhanced level graphs in the sense of Definition \ref{def:normalized_cover_enhanced_level_graph}. 
	
	Define $\tropMSD$ as the locus of tropical normalized covers in $H_{g' \to g, k}^\trop((k), \ldots, (k)) / S_n$ with $n = k(2g-2)$.
\end{mydef} 

Our definition of tropical normalized cover is a tropical analog of Definition \ref{def:normalized_cover} in the sense that the legs of $\hat \Gamma$ encode a canonical divisor, see Lemma \ref{lem:def_trop_normalized_cover_is_redundand}.
We emphasize that any stronger statement about the underlying combinatorics of tropical normalized covers capturing anything about the structure of the dual boundary complex of $\bP\MSD{g}{n}{\mu}$ are most likely false and not needed for our purposes. See also the discussion on future work in the introduction. 

For the next lemma note that all legs of $\hat \Gamma$ necessarily have to carry $o$-value $k$ by Definition~\ref{def:normalized_cover_enhanced_level_graph}. Hence, the structure $\hat \Gamma^+$ is not exactly the one constructed in Lemma \ref{lem:construction_of_k-enhancement} but rather ``$k$ times'' the output of Lemma \ref{lem:construction_of_k-enhancement}.

\begin{mylem} \label{lem:def_trop_normalized_cover_is_redundand}
	Let $\pi : \hat \Gamma \to \Gamma$ be a tropical normalized cover with divisor $D = kK_\Gamma + (f)$ on the base. 
	\begin{enumerate}
		\item Then $(f \circ \pi)/k$ is a rational function on $\hat \Gamma$ (i.e. it has integer slopes) and the legs of $\hat \Gamma$ (neglecting the dilation factor $k$) mark the canonical divisor
		\[F = K_{\hat \Gamma} + \left(\frac{f \circ \pi}{k} \right). \]
		
		\item The 1-enhancement of any half-edge which is part of an edge in $\hat \Gamma^+$ coincides with the one constructed in Lemma \ref{lem:construction_of_k-enhancement} based on $F$. 
	\end{enumerate}
	In particular, the 1-enhanced level graph structure on $\hat \Gamma$ is uniquely determined and thus we may speak of the normalized cover of enhanced level graphs associated to $\pi$.
\end{mylem}

\begin{proof}
	Let $p \in \hat \Gamma$. 
	In the following computation we use the notation $s(h)$ to denote the outgoing slope of $f$ on a half-edge $h$ of $\Gamma$ and $\hat s(\hat h)$ for the outgoing slope of $\frac{f \circ \pi}{k}$ on the half-edge $\hat h$ of $\hat \Gamma$. 
	At this point we do not yet claim that $\hat s (\hat h)$ is an integer -- in fact, this will follow a posteriori from the proof of part two. 
	
	We now show that the number of legs at $p$ is given by $F(p)$. 
	Recall that the $k$-enhancement of $\Gamma$ carries $o$-value 1 on every leg. Thus, by the definition of normalized cover of enhanced level graphs, any leg will have a single preimage in $\hat \Gamma$. This leads us to compute
	\begin{align*}
		\# \{\text{legs at } p \} &= \# \{\text{legs at } \pi(p) \} \\
		&= D(\pi(p)) \\
		&= k \Big(2g(\pi(p)) - 2 + \valence(\pi(p)) \Big) + \sum_{h \ni \pi(p)} s(h) \ .
	\end{align*}	
	On the other side, we have by definition $F(p) = 2g(p) - 2 + \valence(p) + \sum_{\hat h} \hat s(\hat h)$. Since the tropical normalized cover satisfies the local Riemann-Hurwitz condition, we denote the local degree of $\pi$ at $p$ by $k'$ and obtain:
	\begin{equation} \label{eq:value_of_canonical_divisor}
		\begin{aligned}
			F(p) &= k' \Big(2g(\pi(p)) - 2 \Big) + \sum_{\hat h \ni p}(d_{\hat h} - 1) + \valence(p) + \sum_{\hat h \ni p} \hat s(\hat h) \\
			&= k' \Big(2g(\pi(p)) - 2 \Big) + \sum_{h \ni \pi(p)} \underbrace{\Bigg(\sum_{\substack{\hat h \in \pi^ {-1} (h) \\ \hat h \ni p}} d_{\hat h} \Bigg)}_{{= k'}} + \sum_ {\hat h \ni p} \hat s(\hat h) \\
			&= k' \Big(2g(\pi(p)) - 2 + \valence(\pi(p)) \Big) + \sum_{\hat h \ni p} \hat s(\hat h) \ .
		\end{aligned}	
	\end{equation}		
	The last step of Equation \eqref{eq:value_of_canonical_divisor} used that $\pi$ is harmonic. If $\hat h$ maps to $h$ with dilation factor $d_{\hat h}$ then 
	\begin{equation} \label{eq:connection_of_slopes}
		\hat s(\hat h) = \frac{s(h)d_{\hat h}}{k} \ .
	\end{equation} 
	Hence, using the harmonic property once more, we see
	\begin{equation*}
		\sum_{\hat h \ni p} \hat s (\hat h) = \sum_{h \ni \pi(p)} \frac{s(h)}{k} \underbrace{\Bigg(\sum_{\substack{\hat h \in \pi^ {-1} (h) \\ \hat h \ni p}} d_{\hat h} \Bigg)}_{{= k'}} = \frac{k'}{k} \sum_{h \ni \pi(p)} s(h) \ .
	\end{equation*}
	Plugging this into Equation \eqref{eq:value_of_canonical_divisor} shows that $F(p)$ is equal to $k'/k$ times the number of legs at $p$. 
	If there are no legs at $p$, then we are done. Else if there is a leg at $p$ then the local degree of $p$ is automatically $k' = k$ and again we are done. 
	This proves the first part of the claim except for the integrality of $\hat s(\hat h)$.
	
	For the second part we need to verify that the 1-enhancement $\hat o$ on $\hat \Gamma$ which was induced by $o$ does satisfy $\hat o (\hat h) + 1 = -\hat s (\hat h)$ for every internal half-edge $\hat h$ of $\hat \Gamma$. By definition of normalized cover of enhanced level graphs we have
	\begin{equation*}
		\hat o (\hat h) +1 
		= \frac{o(h) + k}{\gcd \big(o(h), k\big )} 
		= - \frac{s(h)}{| \pi^{-1}(h) |} \ .
	\end{equation*}
	On the other hand we need to determine $d_{\hat h}$. 
	Once again we use that $\pi$ is harmonic of degree $k$, i.e. 
	\[ \sum_{\hat h \in \pi^{-1}(h)} d_{\hat h} = k \]
	while at the same time $\tau$ is an isometry acting transitively on $\pi^{-1} (h)$. Consequently,  $d_{\hat h} = k / |\pi^{-1}(h)|$.
	Combining this result with Equation \eqref{eq:connection_of_slopes} we obtain
	\begin{equation*}
		-\hat s (\hat h) = - \frac{s(h) d_{\hat h}}{k} = - \frac{s(h)}{\vert \pi^{-1}(h) \vert} \ .
	\end{equation*}
	This completes the proof of the second part. It also shows that $\hat s(\hat h)$ is an integer for every internal half-edge of $\hat \Gamma$, simply because $\hat s (\hat h) = -\hat o (\hat h) - 1$ and the right hand side of that expression is an integer by definition of a normalized cover of enhanced level graphs. Hence every part of the claim is proven.
\end{proof}

Note that for $\xi = ((k), \ldots, (k))$ the tropicalization $\trop_H$ from Section \ref{sec:tropical_Hurwitz_covers} is $S_n$-equivariant. Recall that the canonical cover construction provides a map $\bP \Omega^k\cM_{g, n}(1, \ldots, 1) \to H_{\hat g \to g, k}(\xi) / S_n$. We define tropicalization of normalized covers by composition
\[ \trop_{\Xi^k} : \bP \Omega^k \cM_{g, n}(1, \ldots, 1)^\an \longrightarrow H_{\hat g \to g, k}(\xi)^\an / S_n \xrightarrow{\trop_H / S_n} \tropMSD \]
where the composition lands in $\tropMSD$ by definition of that locus.

\begin{mylem} \label{lem:trop_N}
	The map $\trop_{\Xi^k}$ is well-defined, continuous, proper, and closed. Furthermore, the following diagram commutes. 
	\begin{equation} \label{eq:trop_Xi_commutes_source_target}
		\begin{tikzcd}
			\bP\Omega^k\cM_{g,n}{(1, \ldots, 1)}^\an \arrow[dd, "\src"] \arrow[rd, "\trop_{\Xi^k}"] \arrow[rrd, bend left = 15, "\trop_{\Omega^k}" ]&& \\
			& \tropMSD \arrow[d, "\src"] \arrow[r, "\tgt"] & \tropHodge \\
			\bP\Omega\cM_{\hat g}(1, \ldots, 1)^\an \arrow[r, "\trop_\Omega"] & \bP \Omega M_{\hat g}^\trop \ . & 
		\end{tikzcd}
	\end{equation}
\end{mylem}

\begin{proof}
	To show that $\trop_{\Xi^k}$ is well-defined, let $(\hat X \to X, \bs, \omega)$ be a normalized cover of smooth curves defined over a non-Archimedean field. Its tropicalization is in particular a tropical $k$-cyclic Hurwitz cover $\hat \Gamma \to \Gamma$. By Proposition \ref{prop:tropicalization_Hurwitz_commutes_source_target} we know that $\hat \Gamma = \trop_{\Omega}(X', \omega)$ and $\Gamma = \trop_{\Omega^k}(X, \eta)$, where $\eta$ is the $k$-differential on $X$. By well-definedness of $\trop_{\Omega}$ and $\trop_{\Omega^k}$ the legs on $\hat \Gamma$ and $\Gamma$ do indeed represent canonical and pluri-canonical divisors $F$ and $D$ respectively. 
	Finally we need to check that $\hat \Gamma \to \Gamma$ can be endowed with the structure of a normalized cover of enhanced level graphs $\hat \Gamma^+ \to \Gamma^+$ such that $\Gamma^+$ is induced by $D$ via Lemma \ref{lem:construction_of_k-enhancement}. 
	First of all, there is indeed such a structure $\hat G^+ \to G^+$ on the underlying $k$-cover of graphs $\hat G \to G$ simply because the graphs are dual to the special fiber of the degeneration of $\hat X \to X$.  We check that the enhancements induced by the divisors are the same as $\hat G^+$ and $G^+$. On the cover this was part of the argument of \cite[Theorem~1]{MUW17_published} and hence the claim. 
	
	Finally we check that $\trop_{\Xi^k}$ is continuous, proper, and closed. 
	The proof of \cite[Theorem~1]{CMR16} shows that the retraction map from the analytification of Hurwitz space to its Berkovich skeleton can be extended to a compactification of both of these spaces. Hence the retraction map is continuous, proper, and closed. 
	By \cite[Theorem~1]{CMR16} tropicalization of Hurwitz covers behaves on each cone of the tropical moduli space like the retraction map. Thus the properties still hold for tropicalization of Hurwitz covers. 
	It is now easy to check that they remain to hold after taking the $S_n$-quotient, i.e. for $\trop_H/S_n$ and continue to hold after restricting to $\bP \Omega^k \cM_{g,n}(1, \ldots, 1)^\an$ by the same arguments as in Lemma~\ref{lem:trop_Omega}. This completes the proof.
\end{proof}

\begin{proof}[Proof of Theorem \ref{intro:thm:tropMSD}]
	The claim on $\trop_{\Xi^k}$ was proved in Lemma \ref{lem:trop_N}.
	We describe the generalized cone complex structure on $\tropMSD$. Consider a combinatorial type of a tropical $k$-cyclic Hurwitz cover, i.e.\ a normalized cover of enhanced level graphs $p : \hat G^+ \to G^+$ with dilation factors on every edge of $\hat G$ and with a deck transformation $\tau : \hat G \to \hat G$. In the proof of Theorem \ref{intro:thm:trop_k_Hodge_bundle} we showed that the range of possible choices for edge lengths on $G$ such that it becomes a tropical curve with pluri-canonical divisor is a finite union of rational polyhedral cones. Now note that each such choice determines a unique tropical Hurwitz cover by \cite[Lemma~17]{CMR16}. All of these choices give indeed tropical normalized covers, because this property does only depend on the combinatorial type. Hence we obtain a stratification of $\tropMSD$ in rational polyhedral cones.
	
	For the statement about the dimension consider again Example~\ref{exa:maximal_cone_tropicalHodge}. The graph depicted in Figure~\ref{fig:maximal_cone_tropicalHodge} can be endowed with a tropical normalized cover by covering every vertex with a single preimage and every edge with as many preimages are necessary to satisfy the definition of a normalized cover of enhanced level graphs. This describes a cone in $\tropMSD$ that maps under $\tgt$ isomorphically onto the cone from Example~\ref{exa:maximal_cone_tropicalHodge}.
\end{proof}

\begin{mydef}
	The image of $\trop_{\Xi^k}$ in $\tropMSD$ is called \emph{locus of realizable covers}.
\end{mydef}

Note that the locus of realizable covers is of positive codimension in $\tropMSD$ by the exact same argument as for the realizability locus in $\tropHodge$. Hence, realizability of normalized covers is a nontrivial problem as well. The following is a more precise version of Corollary \ref{intro:cor:reduction} and reduces our original realizability problem to the one for covers.

\begin{mycor} \label{cor:reduction_to_covers}
	Let $([\Gamma], D)$ be a pair consisting of an isomorphism class of a tropical curve $\Gamma$ with an effective pluri-canonical divisor $D = kK_\Gamma + (f)$. The pair is realizable if any only if there exists a realizable tropical normalized cover $\hat{\Gamma} \to \Gamma$ with $\tgt([\hat{\Gamma} \to \Gamma]) = ([\Gamma], D)$.
\end{mycor}

\begin{proof}
	This is simply the triangle in Diagram \eqref{eq:trop_Xi_commutes_source_target} commuting.
\end{proof}

\begin{myrem}
	Note that contrary to the situation of twisted differentials, a tropical curve with pluri-canonical divisor does not admit a unique normalized cover. Indeed, when asking for realizability of a tropical normalized cover where a vertex $v$ is covered with $d_v$ preimages we are asking for a realization by a twisted differential where $\eta_v$ is precisely a $d_v$-th power of a primitive $k/d_v$-differential.
\end{myrem}

\section{The realizability locus}
\label{sec:realizability_locus}

We now turn to the remaining problem of realizability of tropical normalized covers. Let $\pi : \hat \Gamma^+ \to \Gamma^+$ be a tropical normalized cover with associated enhancements (see Lemma~\ref{lem:def_trop_normalized_cover_is_redundand}). This data contains already most of the discrete information of a boundary point of $\PMSD[k]{g}{n}{1, \ldots, 1}$, however it does not contain continuous information such as the \mbox{($k$-)residues} of the twisted differential. Realizing the tropical datum amounts to choosing a ``valid'' combination of ($k$-)residues. Obstructions arise by the $k$-residue map being non-surjective for some types (this leads to the notion of inconvenient vertex, see Definition \ref{def:inconvenient_vertex}) and by global compatibility conditions (we tackle this by assigning residues along certain cycles in $\hat \Gamma^+$, see Definitions~\ref{def:effective_cycle} and \ref{def:independent_pair}). Once notation is established, we state and prove our main theorem (Theorem \ref{thm:main_theorem}).

\subsection{Special vertices}

As we have seen in Section~\ref{subsec:image_residue_map}, the $k$-residue map 
\[ \Res_g^k(\mu_\red) : \Omega^k \cM_g(\mu_\red)^\prim \longrightarrow \bC^p \times (\bC^\times)^s \] 
is not surjective in general.
Vertices with a reduced type for which the residue map is not surjective are a major obstruction to realizability.
In the abelian case treated in \cite{MUW17_published}, those vertices were called inconvenient.
We will extend the notion of inconvenience to $k$-differentials and additionally introduce a short list of illegal vertices -- a concept which did not appear in the abelian case.

Let us fix some $k \geq 1$.
We will formulate the definitions in the language of normalized covers of enhanced level graphs $\pi : \hat G^+ \to G^+$.
Let $v \in V(G^+)$ be a vertex, let $d_v := |\pi^{-1}(v)|$ be the number of preimages and let $k_v := k / d_v$.
Recall the notation introduced in Definition~\ref{def:reduced_type} where we denoted the reduced type of $v$ by
\[
	\mu_\red(v) = (a_1, \dots, a_t; -b_1, \dots, -b_p; -c_1, \dots, -c_r; -k^s).
\]
We want to realize $\pi^{-1}(v) \to v$ as a normalized cover. In particular $v$ has to be realized as a $d_v$-th power of a primitive $k_v$-differential of type
\[
	\mu_\red'(v) := \frac{1}{d_v} \cdot \mu_\red(v) = (a_1', \dots, a_t'; -b_1', \dots, -b_p'; -c_1', \dots, -c_r'; -k_v^s).
\]

For $k_v=1$ we have $r=0$, as in this case the $c_i$ would be divisible by $d_v = k$, contradicting the definition of the $c_i$.
Following \cite{MUW17_published} we call a vertex $v$ inconvenient if the $k_v$-residue map $\Res_{g(v)}^{k_v}(\mu_\red') : \bP\Omega^{k_v}\cM_{g(v)}(\mu_\red')^\prim \to \bC^p \times (\bC^\times)^s$ is not surjective. More precisely:

\begin{mydef} \label{def:inconvenient_vertex}
	A vertex $v$ is called \emph{inconvenient of type I} if $\mu_\red'$ is one of the types in Proposition~\ref{prop:residue_map_origin_missing} with $k_v$ substituted for $k$. It is called \emph{inconvenient of type II} if $\mu_\red'$ is one of the types in Propositions~\ref{prop:residue_map_line_missing}, \ref{prop:residue_map_abelian_coprime} or \ref{prop:residue_map_triangular} again with $k_v$ instead of $k$. Summarizing, we call $v$ \emph{inconvenient} if it is inconvenient of type I or II.
\end{mydef}

Type I inconvenience means that only the origin is missing from the image of the $k_v$-residue map, whereas type II means that a finite number of lines or at most $2$-dimensional subvarieties is missing. 

Recall from Theorem~\ref{thm:empty_primitive_strata} that for some strata the primitive part $\Omega^k \cM_g(\mu)^\prim$ is empty.
Consequently, vertices that ask to be realized by an element of such an empty primitive part of a stratum are not realizable at all.

\begin{mydef} \label{def:illegal_and_splitting_vertices}
	The vertex $v$ is called \emph{illegal} in the following cases.
	\begin{enumerate}[(i)]
		\item If $\mu_\red' = (-1,1)$.
		\item If $k_v = 1$ and $\mu_\red' = (-1,a_1,\dots,a_t)$ with $a_i \geq 1$.
		\item If $k_v = 2$ and $\mu_\red' = (1,3)$.
		\item If $g(v) = 0$ and $\gcd(\mu_\red', k_v) \neq 1$.
		\item If $k_v \geq 2$ and $\mu_\red' = \emptyset$.
		\item If $k_v = 2$ and $\mu_\red' = (4)$.
	\end{enumerate}
\end{mydef}

\begin{myrem}
	Being illegal is not an intrinsic property of a vertex. Rather it depends on the context of the given normalized cover. For example a vertex of type $\mu_\red = \emptyset$ is not illegal if it is covered by precisely $k$ preimages, i.e. it may be realizable as a $k$-th power of an abelian differential.
\end{myrem}

\subsection{Special cycles}

When assigning ($k$-)residues to our given enhanced combinatorial data $\pi : \hat G^+ \to G^+$ we have to ensure some global compatibility conditions. These conditions are compatibility with the deck transformation $\tau$ as well as the matching residue condition, global residue conditions, and the residue theorem on $\hat G^+$. Compatibility with $\tau$ means in particular, that the choice of residues on $\hat G^+$ determines $k$-residues on $G^+$. Hence we will focus on $\hat G^+$. 

Let $\gamma$ be a simple oriented cycle in $\hat G^+$ and let $L_\gamma$ denote the lowest level $\gamma$ passes through.
We want to use such a cycle to modify the residues of $\hat \Gamma$ similar to the course of action in the proof of \cite[Theorem~6.3]{MUW17_published}.
There the authors chose a complex number $r \in \bC^\times$ and added to each half-edge $h$ of $\hat G^+$ on level $L_\gamma$ the value $r$ (resp.\ $-r$) to the residue at the half-edge $h$ if $\gamma$ leaves (resp.\ enters) the vertex incident to $h$ along this half-edge.
This operation maintains the residue theorem at each vertex, the matching residue condition, and the global residue condition, i.e. if each of those conditions was true before modifying the residues of $\hat G^+$, the conditions still hold for the modified residues.

Recall that compatibility with the deck transformation $\tau : \hat G^+ \to \hat G^+$ means that the assigned residue at a half-edge $\tau(h)$ has to be $\zeta$ times the residue at $h$. 
After adding the residues along the cycle $\gamma$ as described above, this is no longer the case in general.
We address this problem by not only considering the cycle $\gamma$, but the entire $\tau$-orbit of $\gamma$ consisting of
\[
	\gamma_i := \tau_*^i(\gamma) \qquad \text{for } i = 0, \dots, k-1.
\]
We provide each of the cycles $\gamma_i$ with the induced orientation and add $\pm \zeta^i r$ to the residues as described above. We will refer to this operation as \emph{assigning the residue $r$ along the orbit of $\gamma$}. The total change to the residues under this operation can be easily expressed with the following shorthand notation.

\begin{mydef} \label{def:vector_of_cycle}
	Let $\gamma$ be a simple closed cycle in $\hat G^+$ with fixed orientation. Given a vertex $\hat v$ in $\hat G^+$, let $H'(\hat v)$ denote the set of half-edges incident to $\hat v$ with $o$-value $\leq -1$. Then we define a vector $R_\gamma(\hat v) = (R_\gamma(\hat v)_h)_{h \in H'} \in \bC^{|H'|}$ as follows. Set it to be 0 if $\hat v$ does not lie on the lowest level that $\gamma$ passes though and otherwise 
	\begin{equation} \label{eq:definition_R_vector}
            R_\gamma(\hat v)_h := \sum_{i = 0}^{k-1} \epsilon_i \zeta^i, \qquad
            \text{where }
            \epsilon_i := \begin{cases}
                    1 & \text{if $\tau_\ast^i \gamma$ enters $\hat v$ through $h$} \\
                    -1 & \text{if $\tau_\ast^i \gamma$ leaves $\hat v$ through $h$} \\
                    0 & \text{if $\tau_\ast^i \gamma$ does not pass though $h$}. 			
            \end{cases}
	\end{equation}
	
	Now let $v$ be the image of $\hat v$ under $\pi : \hat G^+ \to G^+$ and denote $\mu_\red$ the reduced type of $v$ as in Definition \ref{def:reduced_type}. We define a vector $R^k_\gamma(v) \in \bC^{p+s}$ to be 0 if $R_\gamma(\hat v)$ is 0 and otherwise $R^k_\gamma(v)_h := (R_\gamma(\hat v)_{\hat h})^k$ for $\hat h \in \pi^{-1}(h)$ arbitrary.
\end{mydef}

When assigning the residue $r$ along the orbit of $\gamma$ the residues at a vertex $\hat v$ change precisely by adding $r$ times $R_\gamma(\hat v)$ to the vector of residues. The residues obtained in this way obviously still preserve the residue theorem, the matching residue condition, and the global residue condition. Additionally, we have maintained compatibility with the deck-transformation. This means that the residues we assigned on $\hat G^+$ induce well-defined $k$-residues on $G^+$. By Lemma~\ref{lem:induced_k_residues} the change to the $k$-residues at a vertex $v$ of $G$ is precisely $r^k$ times $R_\gamma^k(v)$. Observe that a given simple closed cycle $\gamma$ can only be used to change the $k$-residues at a vertex $v$ by a $\bC$-multiple of $R^k_\gamma(v)$. This means that some cycles will be more useful for our purpose than others. We illustrate two notable phenomena in Examples \ref{exa:non-effective_cycle} and \ref{ex:pair_of_cycles}.

\begin{myrem}
	The choice of an orientation in Definition \ref{def:vector_of_cycle} merely fixes the sign of $R_\gamma(\hat v)$. For the rest of this article only the $\bC$-span of $R_\gamma^k(v)$ will be of relevance (see Definitions~\ref{def:effective_cycle} and \ref{def:independent_pair} below), hence this choice never really matters. 
\end{myrem}

\begin{myexa} \label{exa:non-effective_cycle}
	\begin{figure}[htb]
		\centering
		\begin{minipage}{0.9\textwidth}
			\centering
			\begin{tikzpicture}[scale=1.5, decoration={markings, mark=at position .5 with {\arrow{>}}}]
	\fill (0, 0) circle [radius = 2pt];
	\fill (2, 0) circle [radius = 2pt];
	\draw[postaction={decorate}] (0, 0) to [bend left=100,min distance=30] (2, 0);
	\draw[postaction={decorate}] (2, 0) to [bend right=40] (0, 0);
	\draw[postaction={decorate}] (0, 0) to [bend right=40] (2, 0);
	\draw[postaction={decorate}] (2, 0) to [bend left=100,min distance=30] (0, 0);
	\draw[<->, dotted, shorten <= 6pt, shorten >= 6pt] (0, 0) to node[above]{$\tau$} (2, 0);
	\draw[<->, dotted] (1, 0.45) to node[right]{$\tau$} (1, 0.7);
	\draw[<->, dotted] (1, -0.45) to node[right]{$\tau$} (1, -0.7);

	\path[draw] (0,0) --++ (-0.25,0);
	\draw (-0.4, 0) node {2};
	\path[draw] (2,0) --++ (0.25,0);
	\draw (2.4, 0) node {2};

	\draw (-.5, .5) node {$\hat G^+$};

	\fill (1, -2) circle [radius = 2pt];
	\draw (1.5, -2) circle [radius = .5];
	\draw (0.5, -2) circle [radius = .5];

	\path[draw] (1,-2) --++ (0,-0.35);
	\draw (1, -2.5) node {4};

	\draw (-.5, -1.5) node {$G^+$};
\end{tikzpicture}
		\end{minipage}
		\caption{A cover of graphs with the action of the deck-transformation $\tau$.}
		\label{fig:noneffective_cycle}
	\end{figure}
	
	Consider the cover of enhanced level graphs $\hat G^+ \to G^+$ for $k=2$ depicted in Figure~\ref{fig:noneffective_cycle}.
	Let $\gamma_1$ be the simple cycle in $\hat G^+$ that uses the two topmost edges, considered with the orientation indicated in the picture.
	When we compute $R_{\gamma_1}(\hat v)$ for either of the vertices in $\hat G^+$ we see that the vector is zero.  
	To see this, consider the top left half-edge.
	Along $\gamma_1$, we add $1$.
	The oriented cycle $\tau_* \gamma_1$ agrees with $\gamma_1$, and thus we add at the same half-edge the value $\zeta = -1$ when distributing the residues for $\tau_* \gamma_1$.
	In total we have added $1-1=0$ to the residue.
	This means that $\gamma_1$ cannot be used to assign any nonzero residues at all.

	Let $\gamma_2$ be the simple cycle in $\hat G^+$ that uses the two outermost edges, again considered with the orientation indicated in the picture.
	Computing $R_{\gamma_2}(\hat v)$ for either of the vertices in $\hat G^+$ we see that the vector now is $(-1,-1,1,1)$.
	But this vector is not in the image of the $1$-residue map for a vertex of type $(-1^4, 2)$ by Proposition~\ref{prop:residue_map_abelian_coprime}.
	These two observations motivate the following definition.
\end{myexa}	

\begin{mydef} \label{def:effective_cycle}	
	Let $\gamma$ be a simple closed cycle in $\hat G^+$ and let $v$ be a vertex in $G^+$. Choose and fix an orientation of $\gamma$. We say that $\gamma$ is \emph{effective} for a half-edge $h$ incident to $v$ if $R^k_\gamma(v)_h$ is nonzero. 
	If $v$ is inconvenient of type I, we say that $\gamma$ is \emph{admissible for $v$} if $\gamma$ is effective for at least one vertical half-edge $h$ incident to $v$. 
	If $v$ is inconvenient of type II, we say that $\gamma$ is \emph{admissible for $v$} if $R^{k_v}_\gamma(v)$ lies in the image of $\Res_g^{k_v}(\mu'_\red(v))$.
\end{mydef}

\begin{myexa} \label{ex:pair_of_cycles}
	\begin{figure}[htb]
		\centering
		\begin{minipage}{0.9\textwidth}
			\centering
			\begin{tikzpicture}[scale=1.5, decoration={markings, mark=at position .5 with {\arrow{>}}}]
	\fill 
	(0, 0) circle (2pt);
	\draw
	(-0.5, 1) node[draw, thick, shape = circle, inner sep = 1pt, anchor = south] {1}
	(-1.5, 1) node[draw, thick, shape = circle, inner sep = 1pt, anchor = south] {1}
	(0.5, 1) node[draw, thick, shape = circle, inner sep = 1pt, anchor = south] {1}
	(1.5, 1) node[draw, thick, shape = circle, inner sep = 1pt, anchor = south] {1};
	
	\path[draw] (0,0) to[bend left = 12] (-1.5, 1) to[bend left = 12] (0,0)
	to[bend left = 20] (-.5, 1) to[bend left = 20] (0,0)
	to[bend left = 20] (.5, 1) to[bend left = 20] (0,0)
	to[bend left = 12] (1.5, 1) to[bend left = 12] (0,0);
	
	\path[draw] (0,0) --++ (-45:0.3);
	\draw (0.3, -.3) node {8};
	\path[draw] (0,0) --++ (-135:0.3);
	\draw (-0.3, -.3) node {6};
	
	\draw (-2.5, .5) node {$\hat G^+$};

	\draw (-1, -1) node[draw, thick, shape = circle, inner sep = 1pt, anchor = south east] {1};
	\draw (1, -1) node[draw, thick, shape = circle, inner sep = 1pt, anchor = south west] {1};
	
	\path[draw] (0, -2) to[bend left = 30] (-1, -1) to[bend left = 30] (0, -2) to[bend left = 30] (1, -1) to[bend left = 30] (0, -2);
	
	\fill (0, -2) circle (2pt);
	
	\path[draw] (0,-2) --++ (-45:0.3);
	\draw (0.3, -2.3) node {7};
	\path[draw] (0,-2) --++ (-135:0.3);
	\draw (-0.3, -2.3) node {5};
	
	\draw (-2.5, -1.5) node {$G^+$};
	
\end{tikzpicture}
		\end{minipage}
		\caption{An inconvenient vertex which cannot be redeemed with a single cycle.}
		\label{fig:independent_pair}
	\end{figure}
	
	Contrary to the abelian case in \cite{MUW17_published} there are situations, where using only one cycle will not be sufficient to achieve valid residues. To see this, consider the graph in Figure \ref{fig:independent_pair}. The vertex on lowest level is inconvenient of type II (see Figure~\ref{fig:residue_map_line_missing}, last of the $k = 2$ cases with $p = b = 4$). More specifically, the image of the residue map is missing any tuples where two entries agree while the other two are zero. The depicted cover $\hat G^+$ is the only valid choice -- covering a genus 1 vertex of type $(0,0)$ with only one preimage would be illegal. Notice that each of the four simple cycles in $\hat G^+$ is effective but induces a tuple of $k$-residues on the base which is not contained in the image of the residue map. Hence none of the available cycles is sufficient to redeem the inconvenient vertex, however using two cycles will work. A converse to this phenomenon is illustrated in Example~\ref{exa:contracting_level}: there we have an inconvenient vertex which can only be redeemed with an admissible cycle but no pair of cycles.
\end{myexa}

\begin{mydef} \label{def:independent_pair}
	Let $v$ be a vertex in $\hat G^+$ and let $\gamma$ and $\gamma'$ be oriented cycles which are effective for $v$. We call $(\gamma, \gamma')$ an \emph{independent pair for $v$} if the induced vectors $R^{k_v}_\gamma(v)$ and $R^{k_v}_{\gamma'}(v)$ are not contained in the same linear subspace of the complement of $\Im \Res_g^{k_v}(\mu'_\red(v))$.
\end{mydef}

\begin{myrem}
	The upshot behind Definition \ref{def:independent_pair} is that linear combinations of the vectors $R^{k_v}_\gamma(v)$ and $R^{k_v}_{\gamma'}(v)$ will generically lie in the image of the $k_v$-residue map. This is trivially true for inconvenient vertices of type I as the cycles of an independent pair are necessarily admissible, i.e. effective. In fact, the notion of independent pair is not interesting for type I vertices. For inconvenient vertices of type II this is easily seen to be true for all cases where the complement of $\Res_g^{k_v}(\mu'_\red(v))$ consists of a finite union of lines or planes, i.e. for all cases except those considered in Proposition \ref{prop:residue_map_triangular} Part (\ref{missing_cones}). This is the only case where a finite union of 2 dimensional cones -- none of which contains a 2 dimensional plane -- is missing. Here a pair of cycles may still be independent even if both $R^{k_v}$-vectors lie in the same irreducible component of the complement of $\Im \Res_g^{k_v}(\mu'_\red(v))$.
\end{myrem}

\subsection{Realizability of covers}

\begin{mythm} \label{thm:main_theorem}
	Fix an algebraically closed base field of characteristic~0.
	Let $g \geq 2$ and fix an integer $k \geq 1$. Let $\pi : \hat \Gamma^+ \to \Gamma^+$ be a tropical normalized cover with enhancements associated by Lemma \ref{lem:def_trop_normalized_cover_is_redundand}. Denote the effective pluri-canonical divisor marked by the legs of $\Gamma$ by $D = kK_\Gamma + (f)$. Then $(\pi : \hat \Gamma \to \Gamma, D)$ is realizable if and only if the following conditions hold.
	\begin{enumerate}[(i)]
		\item There is no illegal vertex in $\pi$.
		\item For every horizontal edge $\hat e$ in $\hat \Gamma^+$ there is an effective cycle in $\hat \Gamma^+$ through $\hat e$.
		\item For every inconvenient vertex $v$ in $\Gamma^+$ there is an admissible cycle in $\hat \Gamma^+$ through one of the preimages $\hat v$ or there is an independent pair of cycles.
	\end{enumerate}
\end{mythm}

\begin{myrem} \label{rem:MUW_recovered}
	Let us explain how to recover \cite[Theorem~6.3]{MUW17_published} from Theorem~\ref{thm:main_theorem} for $k = 1$. Recall that \cite[Theorem~6.3]{MUW17_published} states the following: the pair $(\Gamma, D)$ for $D$ as above is realizable if and only if
    \begin{itemize}
        \item[(i')] For every inconvenient vertex (in the sense of \cite[Definition~6.2]{MUW17_published}) $v$ in $\Gamma$ there is a simple cycle in $\Gamma$ through $v$ that does not pass through any node on a level below $v$.
        \item[(ii')] For every horizontal edge $e$ in $\Gamma$ there is a simple cycle passing through $e$ which does not pass through any node on a level below $e$.
    \end{itemize}
	To see that these conditions are equivalent to ours, note that for $k=1$ the identity on $\Gamma$ is the only tropical normalized cover.
	Now assume $(\id : \Gamma \to \Gamma, D)$ satisfies the conditions (i), (ii), and (iii) of Theorem~\ref{thm:main_theorem}. Every inconvenient vertex in the sense of \cite{MUW17_published} is inconvenient in our sense as well. Furthermore, every effective or admissible cycle does not pass through any lower level. Hence $(\Gamma, D)$ satisfies (i') and (ii') as well. 
	
	Conversely, suppose (i') and (ii') hold. First note that the only type of illegal vertex for $k=1$ is $(-1,1)$ and such a vertex does not admit a simple cycle ``at or above level'' through the incident horizontal edge. Thus (ii') ensures that there are no illegal vertices. 
	The next observation is that for $k = 1$ a cycle $\gamma$ is effective for every half-edge at lowest level that $\gamma$ passes though. 
	In particular, (ii) holds.
	Furthermore, there are only two kinds of inconvenient vertices: the ones in Proposition~\ref{prop:residue_map_origin_missing}~(i) and the ones in Proposition~\ref{prop:residue_map_abelian_coprime}. 
	The former is inconvenient in the sense of \cite{MUW17_published} as well. Hence, (i') provides the necessary effective cycles.
	The other kind of inconvenient vertex is not an issue for $k=1$: all simple cycles use precisely two half-edges incident to each vertex they pass through, and thus the residues at $\geq 3$ horizontal half-edges may always be chosen sufficiently generic. In other words, the cycles provided by (ii') contain an independent pair.
\end{myrem}

We split the proof of Theorem \ref{thm:main_theorem} in three parts. First we prove that the conditions in the theorem are sufficient (resp.\ necessary) for realizability over the base field $\bC$ endowed with the trivial valuation. Afterwards we generalize the result to arbitrary algebraically closed base fields of characteristic~0.

\begin{mylem} \label{lem:rational_edge_lengths}
	Let $\cT \subseteq \tropMSD$ be the locus defined by the conditions of Theorem~\ref{thm:main_theorem}.
	Then $\cT$ is closed in $\tropMSD$. Moreover, let $\cQ \subseteq \tropMSD$ be the locus of tropical normalized covers $\hat \Gamma \to \Gamma$ such that all edge-lengths of $\Gamma$ are rational. Evidently, $\cQ$ is dense $\tropMSD$. But also $\cT \cap \cQ$ is dense in $\cT$.
\end{mylem}

\begin{proof}
	The first part of the claim is quite clear and we focus on showing that $\cT \cap \cQ$ is dense in $\cT$. So let $\hat \Gamma \to \Gamma$ be a tropical normalized cover which satisfies all conditions of Theorem~\ref{thm:main_theorem}. Our goal is to write $\Gamma$ as a limit of tropical curves $\Gamma_j$ with rational edge-lengths such that the combinatorial type for each $j$ is the same as that of $\hat\Gamma \to \Gamma$. In particular this means that all special cycles $\gamma$ appearing in the conditions of Theorem~\ref{thm:main_theorem} are still cycles of the same kind for each $\Gamma_j$ since the conditions on being effective, admissible, or part of an independent pair of cycles do in fact not depend on the edge-lengths. 
	Denote the edge-length function of $\Gamma_j$ by $l_j$ and the slope of $f$ on a given (oriented) edge $e$ by $s(e)$.
	The only condition that needs to be ensured in this approximation is 
	\[ \sum_{e \in \gamma} l_j(e) s(e) = 0 \]
	for all $j$. 
	
	Let $Q$ be the $\bQ$-vector space spanned by the edge lengths $l(e) \in \bR$ of $\Gamma$. This is a finite dimensional vector space (because $\Gamma$ has only finitely many edges). Choose a basis $B = \{b^{(i)}\}_i$ of $Q$ with $b^{(i)} \in \bR$ and express the edge-lengths of $\Gamma$ as $l(e) = \sum_{i} l(e)_i \, b^{(i)}$. Now approximate each basis element by rational numbers $b^{(i)} = \lim b^{(i)}_j$ and define $l_j(e) = \sum_i l(e)_i \, b^{(i)}_j$. Then 
	\begin{equation} \label{eq:approximate_cycles}
		\sum_{e \in \gamma} l_j(e) s(e) = \sum_{e \in \gamma} \sum_i l(e)_i \, b^{(i)}_j s(e) = \sum_i \Big( \sum_{e \in \gamma} l(e)_i \, s(e) \Big) b^{(i)}_j \ .
	\end{equation}
	But now notice that $\gamma$ is cycle in $\Gamma$, i.e.
	\[ \sum_{e \in \gamma} l(e) s(e) = \sum_{i} \Big( \sum_{e \in \gamma} l(e)_i \, s(e) \Big) b^{(i)} = 0 \ . \]
	Since the $b^{(i)}$ are a vector space basis it follows that $\sum l(e)_i \, s(e) = 0$ for all $i$. Using this in Equation~\eqref{eq:approximate_cycles} completes the proof.
\end{proof}

\begin{myprop}
	The conditions in Theorem~\ref{thm:main_theorem} are sufficient for realizability over the base field $\bC$.
\end{myprop}

\begin{proof}
	In the following we will show that a tropical normalized cover $\pi : \hat\Gamma \to \Gamma$ with integer edge lengths on $\Gamma$ which satisfies the conditions of Theorem~\ref{thm:main_theorem} is realizable. This is indeed sufficient to prove the proposition: by Lemma~\ref{lem:rational_edge_lengths} the closure of the locus of tropical normalized covers with rational edge-lengths meeting the conditions of Theorem~\ref{thm:main_theorem} is dense in the locus of all covers meeting the conditions of the theorem. Moreover, by Lemma~\ref{lem:trop_N} $\trop_{\Xi^k}$ is closed, thus the realizability locus in $\tropMSD$ is closed. Combined, these arguments give the reduction step to covers with rational edge lengths. 
	But then again if $\pi$ is realizable then so is the cover obtained by rescaling the edge lengths with a global constant, i.e. we may even restrict to integer edge lengths.
	
	Now assume that the conditions of Theorem~\ref{thm:main_theorem} hold for $\pi$ and $\Gamma$ has integer edge lengths. Ultimately we want to realize $\pi$ by a normalized cover $\hat X \to X$ of smooth curves over a non-Archimedean field with residue field $\bC$, such that the $k$-differential on the base is of type $(1, \ldots, 1)$. To do so, we want to choose for every half-edge in $\hat \Gamma^+$ with $o$-value $\leq -1$ and every half-edge in $\Gamma^+$ with $o$-value $\leq -k$ a ($k$-)residue (i.e.\ a complex number) such that all of the following hold.
	
	For every vertex $v \in \Gamma^+$ there exists a smooth curve $X_v$ with meromorphic $k$-differential $\eta_v$ realizing $v$. 
	More precisely, $X_v$ is supposed to be of genus $g(v)$ with distinguished points $z_h \in X_v$ for every half edge $h$ incident to $v$ such that $\ord_{z_h} \eta_v = o(h)$ and $\Res^k_{z_h}\eta_v$ is the value chosen in the beginning and $\eta_v$ is holomorphic and nonzero outside of $\{z_h\}_h$. Furthermore, $\eta_v$ is supposed to be the $d_v$-th power of a primitive $k_v$-differential of type $\mu_\red'(v)$. 
	Yet again further, we require each of the connected components $\hat X_{\hat v}$ of the (uniquely determined) normalized cover of $X_v$ to realize one of the vertices $\hat v \in \pi^{-1}(v)$, again such that orders of the meromorphic abelian differentials $\omega_{\hat v}$ match the $o$-values on $\hat \Gamma^+$ and the residues coincide with the chosen values from the beginning. 
	Finally, we need to do all of this such that the normalized covers $\coprod \hat X_{\hat v} \to X_v$ glue into a normalized cover of nodal curves $\hat X \to X$ with dual graphs given precisely by $\hat \Gamma$ and $\Gamma$.
	Once this is achieved, we obtain the desired normalized cover of smooth curves with deformation parameters corresponding to the edge-lengths of $\hat \Gamma$ and $\Gamma$ by means of Theorem \ref{thm:BCGGM_main_theorem}. 
	
	We note some dependencies among these requirements. Specifying residues on $\hat \Gamma^+$ that satisfy the condition imposed by the residue theorem and that are compatible with the deck transformation $\tau$ already determines the $k$-residues on $\Gamma^+$. This ensures the realizability of $v$ in the above sense if the induced $k_v$-residues are contained in the image of the $k_v$-residue map $\Res^{k_v}_{g(v)}(\mu_\red'(v))$. In this case realizability of all of the $\hat v \in \pi^{-1}(v)$ is immediate. When it comes to global compatibility note first that compatibility with the level structure is already built into the definition of enhanced level graphs. Beyond that, we only need to ensure MRC and GRC for the cover $\pi : \hat \Gamma^+ \to \Gamma^+$. To summarize, our goal is to choose residues on $\hat \Gamma^+$ such that:
	\begin{itemize}
		\item For each $\hat v \in \hat \Gamma^+$ the condition imposed by the residue theorem is satisfied. Moreover, MRC and GRC are satisfied.
		\item Residues on $\hat \Gamma^+$ are compatible with $\tau$, i.e.\ the residue of $\tau(h)$ is precisely $\zeta$ times the residue at $h$.
		\item The $k$-residues which are given on $\pi(h)$ as the $k$-th power of the residue at $h$ make every vertex of $\Gamma$ realizable. 
	\end{itemize}
	
	Let us now argue that a suitable choice of such residues exists. We start by initializing all residues with 0. Let $\gamma_1, \dots, \gamma_\lambda$ be all the simple cycles in $\hat \Gamma^+$ that exist by assumption, i.e.\ the cycles containing the horizontal edges and all kinds of inconvenient vertices. For each $\gamma_i$ we choose and fix an orientation. Note that at this point the first two items of our list of requirements are already satisfied. By construction of the process of ``assigning a residue $r_i$ along the $\tau$-orbit of $\gamma_i$'', these conditions continue to hold after doing so.
	
	Let us now pick numbers $r_1, \dots, r_\lambda \in \bC$ to be assigned along the orbits of the $\gamma_i$ such that the third and final condition is met. This amounts to choosing the $r_i$ sufficiently generic such that no undesirable cancellation happens. More precisely, after all the residues have been assigned, the resulting $k_v$-residues at a vertex $v \in \Gamma^+$ are
	\[ \sum_{i = 1}^{\lambda} r_i^{k_v} R^{k_v}_{\gamma_i}(v) \]
	and this has to lie in the image of the $k_v$-residue map. At every vertex this amounts to avoiding a locus of positive codimension in $\bC^{p+s}$. At the same time, the values that can be achieved using the given $\gamma_i$ form a vector space $V_v$. By assumption we have for every deficit in surjectivity of the residue map an admissible cycle or an independent pair of cycles, i.e.\ $V_v$ is not fully contained in the complement of the image of the residue map. Hence a suitable choice for each $r_i$ is possible and we are done.
\end{proof}

Let $\hat X$ be a smooth complex curve, and denote by $\PD : H^1_{\dR}(\hat X; \bR) \to H_1(\hat X; \bR)$ the map given by Poincar\'e duality.
By abuse of notation, we denote the induced map
\begin{align*}
	\PD : \Omega^1(\hat X) &\longrightarrow H_1(\hat X; \bC) \\
	\omega &\longmapsto \PD \big(\Re(\omega) \big) \oplus i \cdot \PD \big(\Im(\omega) \big)
\end{align*}
by the same symbol.
By naturality of $\PD$ there is a commutative diagram
\begin{equation} \label{eq:diagram_Poincare_duality}
	\begin{tikzcd}
		\Omega^1(\hat X) \ar[d, "\tau^*"] \ar[r, "\PD"] & H_1(\hat X; \bC) \\
		\Omega^1(\hat X) \ar[r, "\PD"] & H_1(\hat X; \bC) \ar[u, "\tau_*"'].
	\end{tikzcd}
\end{equation}

\begin{myprop}
	The conditions in Theorem~\ref{thm:main_theorem} are necessary for realizability over the base field $\bC$.
\end{myprop}

\begin{proof}
	Let $[\hat \Gamma \to \Gamma] = \trop_{\Xi^k}( \pi : \hat X \to X, \bs, \omega )$ be given. We want to show that the tropical normalized cover satisfies the conditions of Theorem \ref{thm:main_theorem}. Since $\trop_{\Xi^k}$ is continuous (Lemma \ref{lem:trop_N}) and the locus in $\tropMSD$ which is cut out by the conditions of Theorem~\ref{thm:main_theorem} is closed, it suffices to show this for any $\trop_{\Xi^k}(\pi : \hat X \to X)$ for $\pi$ taken from a dense subset of $\bP \Omega^k {\cM}_{g, n}(1, \ldots, 1)^\an$.
	Recall that the field of Puisseux series $\bC \{\!\{t\}\!\}$ is algebraically closed and has a non-trivial non-Archimedean valuation. By \cite[\nopp 2.6]{Gubler_guide} the set of $\bC \{\!\{t\}\!\}$-valued points of $\bP \Omega^k {\cM}_{g, n}(1, \ldots, 1)$ is dense in $\bP \Omega^k {\cM}_{g, n}(1, \ldots, 1)^\an$ (use \cite[Proposition~3.8]{Uli_stack_quotient} to pass from stacks to spaces). 
	But now any such point $(\pi : \hat X \to X, \bs, \omega)$ may be interpreted as the smooth generic fiber of a family of normalized covers of twisted $k$-differentials defined over a discrete valuation ring of some finite extension of $\bC(t)$. 
	Hence, restricting to this situation, we may now take the equivalent $\bC$-analytic point of view and consider this data as a family of normalized covers $(\pi_t : \hat X_t \to X_t)_t$ over the punctured unit disc $\Delta^\ast$. In particular, each $\hat X_t$ and $X_t$ is a smooth curve over $\bC$. 
	
	Let $\hat X_0 \to X_0$ denote the admissible cover obtained as the limit of the family within $\bP \Xi^k \compact{\cM}_{g,n}(1, \ldots, 1)$ for $t \to 0$. By assumption, the dual graphs of $\hat X_0$ and $X_0$ are the underlying unmetrized graphs of $\hat \Gamma$ and $\Gamma$ respectively. Furthermore, the enhanced level graph structures induced by $\hat X_0$ and $X_0$ are precisely $\hat \Gamma^+$ and $\Gamma^+$ respectively (see the argument in the proof of Lemma \ref{lem:trop_N}).
	We check that the conditions of Theorem \ref{thm:main_theorem} hold for these.

	There cannot be any illegal vertex $v$ in $\Gamma$. Otherwise the restriction $\coprod_{\hat v \in \pi^{-1} (v) } \hat X_{0,\hat v} \to X_{0,v}$ of the central fiber would provide an element in a stratum that is empty by Theorem~\ref{thm:empty_primitive_strata}, a clear contradiction.

	Next, we show the existence of effective cycles for horizontal edges and admissible cycles or independent pairs of cycles for all inconvenient vertices.
	Fix such an edge or vertex and let $L$ denote its level in $\Gamma^+$.
	Recall that we are only interested in cycles at or above level $L$.
	To ensure that any cycle we find during this proof satisfies this condition, we use the following trick.
	Take the truncated cover $\hat X_{0,\geq L} \to X_{0,\geq L}$ at or above level $L$. After restricting to the connected component that contains the edge or vertex under consideration, we obtain a twisted differential from some holomorphic stratum $\PMSD[k]{g'}{n'}{\mu'}$.
	It can be written as limit of a family $(\pi'_t : \hat X'_t \to X'_t)_t$ of smooth normalized covers.
	Moreover, any cycle in the dual graph of $\hat X_{0,\geq L}$ is also a cycle in the dual graph of $\hat X_0$ at or above level $L$ that inherits the property of being effective (resp.\ admissible).
	Thus it suffices to find suitable cycles in $\hat X_{0,\geq L}$.
	Hence we will implicitly work with the family $\pi'_t$ and assume all our cycles to be at or above level $L$.

	Let $\omega_t$ be the abelian differential on $\hat X_t$ and let $\gamma_t := \PD(\omega_t)$. By the commutativity of Diagram \eqref{eq:diagram_Poincare_duality} we have
	\[ \gamma_t = \PD(\omega_t) = \tau_\ast \PD(\tau^\ast \omega_t) = \tau_\ast \PD(\zeta \omega_t) = \zeta \tau_\ast \gamma_t. \]
	Repeating the argument with $\tau$ being replaced by a power $\tau^i$ for $i = 0, \ldots, k-1$ and summing the resulting equations we obtain
	\begin{equation} \label{eq:gamma_is_effective}
            k\gamma_t = \sum_{i = 0}^{k-1} \zeta^i \tau^i_\ast \gamma_t.
	\end{equation}
	
	Now consider a vertex $v \in \Gamma^+$ and let $h_1, \dots, h_{p+s}$ be the half-edges incident to $v$ with $o$-value $\leq -k$ and divisible by $k$. Let $\hat v$ be a preimage of $v$ and let $\hat h_1, \dots, \hat h_{p+s}$ be half-edges incident to $\hat v$ such that $\hat h_\iota$ is a preimage of $h_\iota$. Let $\alpha_t^{(1)}, \ldots, \alpha_t^{(p + s)}$ be families of simple closed cycles in $\hat X_t$ which get pinched into the corresponding nodes $\hat q_1, \dots, \hat q_{p+s}$.
	Observe that $\int_{\alpha_t^{(\iota)}} \omega_t$ converges for $t \to 0$ to the residue $r_{\hat q_\iota}$ of the limiting twisted differential at $\hat q_\iota$. By Poincar\'e duality and equation~\eqref{eq:gamma_is_effective} this implies
	\begin{equation} \label{eq:gamma_cap_alpha_not_0}
            \lim_{t\to 0} \sum_{i = 0}^{k-1} \zeta^i \tau^i_\ast \gamma_t \cap \alpha_t^{(\iota)} = \lim_{t \to 0} k\gamma_t \cap \alpha_t^{(\iota)} = k r_{\hat q_\iota} \qquad \text{for } \iota = 1, \dots, p+s \ .
	\end{equation}

        Next we want to make a consistent choice of a basis of $H_1(\hat X_t; \bC)$. In general, such a choice is not possible across the entire family. Hence we restrict our family to a real ray in $\Delta^\ast$, i.e.\ from now on only consider $t \in (0,1)$.
        Recall that there is a surjective map
        \[
                \Phi : H_1(\hat X_0; \bZ) \longrightarrow H_1(\hat G; \bZ)
        \]
        into the homology of the dual graph $\hat G$ of the special fiber $\hat X_0$.
        Let $(\beta^\trop_j)_{j \in \{1, \dots, b_1(\hat G)\}}$ be a basis of $H_1(\hat G; \bZ)$ consisting of simple cycles and let $\beta'_j \in H_1(\hat X_0; \bZ)$ be preimages of the $\beta^\trop_j$.
        Let $J := \{1, \dots, 2g(\hat X)\}$.
        We can complete the $(\beta'_j)_{j \in \{1, \dots, b_1(\hat G)\}}$ to a basis $(\beta'_j)_{j \in J}$ of $H_1(\hat X_0; \bZ)$ in such a way that $\beta'_j \in \ker \Phi$ for $i \in \{b_1(\hat G) + 1, \dots, 2g(\hat X_0)\}$.
		In other words, the new $\beta'_{j}$ have a representative with support in a single irreducible component $\hat X_{0,v}$ of $\hat X_0$.
        Chose cycles $\beta_j$ on a nearby surface $\hat X_t$ along the real ray such that $\beta_j$ converges to $\beta'_j$ for $t \to 0$.
        The $(\beta_j)_{j \in J}$ form a basis of $H_1(\hat X_t; \bZ)$ and $H_1(\hat X_t; \bC)$ for all $t$ in the real ray.

	With our chosen basis we may write $\gamma_t = \sum_{j \in J} c_t^{(j)} \beta_j$ for uniquely determined complex coefficients $c_t^{(j)}$ varying continuously for $t \in (0,1)$.
        Equation~\eqref{eq:gamma_cap_alpha_not_0} implies
        \begin{equation} \label{eq:gamma_in_basis_not_0}
                \lim_{t \to 0} \sum_{i=0}^{k-1} \sum_{j \in J} c_t^{(j)} \zeta^i \tau^i_* \beta_j \cap \alpha_t^{(\iota)} = k r_{\hat q_\iota} \qquad \text{for } \iota = 1, \dots, p+s.
        \end{equation}
        We claim that for all $\iota = 1, \dots, p+s$ we have
        \[
                \lim_{t \to 0} \sum_{i=0}^{k-1} \zeta^i \tau^i_* \beta_{j} \cap \alpha_t^{(\iota)} =
                \begin{dcases*}
                        R_{\beta_j^\trop}(\hat v)_{\hat h_\iota} & for $j = 1, \dots, b_1(\hat G)$ \\
                        0 & otherwise
                \end{dcases*}
        \]
        for $R_{\beta_j^\trop}(\hat v)_{\hat h_\iota}$ as in \eqref{eq:definition_R_vector}.
        For the first claim observe that for $j \in \{1, \dots, b_1(\hat G)\}$ the limit $\lim_{t \to 0} \beta_j \cap \alpha_t^{(\iota)}$ agrees with the coefficient of $\beta^\trop_j$ in front of the edge $\hat e \in \hat G$ corresponding to the node $\hat q_\iota \in \hat X_0$.
        Thus the claim follows by comparing equation~\eqref{eq:gamma_is_effective} with equation~\eqref{eq:definition_R_vector}.
        Here the sign appearing in \eqref{eq:definition_R_vector} is encoded in the intersection product of the implicitly oriented cycles $\beta_j$ and $\alpha_t^{(\iota)}$.
        For the second claim observe that for $j \in \{b_1(\hat G) + 1, \dots, 2g(\hat X)\}$ the cycle $\beta_j$ is chosen such that is does not intersect $\alpha_t^{(\iota)}$.
        Thus we may rewrite equation~\eqref{eq:gamma_in_basis_not_0} as
        \begin{equation} \label{eq:sum_of_residues}
                \sum_{j = 1}^{b_1(\hat G)} c^{(j)} R_{\beta_j^\trop}(\hat v)_{\hat h_\iota} = k r_{\hat q_\iota} \qquad \text{for } \iota = 1, \dots, p+s,
        \end{equation}
        where $c^{(j)} := \lim_{t \to 0} c_t^{(j)}$.

        Now assume that $h_{\iota_0}$ belongs to a horizontal edge. In this case, $r_{\hat q_{\iota_0}}$ is nonzero. Hence by equation~\eqref{eq:sum_of_residues}, there is a $j_0 \in J$ such that $R_{\beta_{j_0}^\trop}(\hat v)_{\hat h_{\iota_0}} \neq 0$.
        In particular $\beta_{j_0}^\trop$ is effective for $h_{\iota_0}$.

        Now assume that $v$ is inconvenient of type~I. In this case, the origin is not in the image of the residue map. Thus there is an $\iota_0$ such that $r_{\hat q_{\iota_0}}$ is nonzero.
        By the same argument as above, there is an $j_0$ such that $R_{\beta_{j_0}^\trop}(\hat v)_{\hat h_{\iota_0}} \neq 0$ by equation~\eqref{eq:sum_of_residues} and thus $\beta^\trop_{j_0}$ is admissible for $v$.

	Now assume that $v$ is inconvenient of type~II.
        In this case, the powers $(r_{\hat q_\iota}^{k_v})_\iota$ are contained in the image of the residue map. 
        Recall that the complement of the image of the residue map is an union of at most $2$-dimensional subvarieties $\bigcup_{\sigma'} W'_{\sigma'} \in \bC^{p+s}$.
        Thus the residues $(r_{\hat q_\iota})_\iota$ may not be contained in an union of at most $2$-dimensional subvarieties $\bigcup_\sigma W_\sigma \in \bC^{p+s}$, too: Each subvariety $W'$ in the image of the residue map gives rise to multiple subvarieties $W$ corresponding to choices of the $k$-th root.
        If there is an admissible cycle $\beta_j^\trop$ we are done.
        So assume that there is none.
        Moreover, assume for a contradiction that all vectors $(R_{\beta_j^\trop}(\hat v)_{\hat h_\iota})_\iota$ are contained in a single linear subspace $V$ of a subvariety $W_{\sigma_0}$.
        Then by equation~\eqref{eq:sum_of_residues}, the vector $(r_{\hat q_\iota})_\iota$ is contained in the same linear subspace $V$.
        But then the vector $(r_{\hat q_\iota})_\iota$ is not contained in the image of the residue map, which is a contradiction.
        Thus there is an independent pair of cycles $(\beta_{j_1}^\trop, \beta_{j_2}^\trop)$.
\end{proof}

	So far we have shown the claim to hold over $\bC$. We now generalize to any algebraically closed base field of characteristic 0.

\begin{myprop}
	If Theorem~\ref{thm:main_theorem} is true over the base field $\bC$, it is true over any algebraically closed base field of characteristic 0.
\end{myprop}
\begin{proof}
	If $K \subseteq L$ is a valued field extension of such fields, then there is a surjective map $\big( \bP \Omega^k \cM_{g}(1, \ldots, 1)_L \big)^\an \to \big( \bP \Omega^k \cM_{g}(1, \ldots, 1)_K \big)^\an$ and the following diagram commutes by \cite[Proposition~3.7]{Gubler_guide}
	\begin{equation*}
		\begin{tikzcd}[column sep = huge]
			\big(\bP \Omega^k \cM_{g}(1, \ldots, 1)_L \big)^\an \arrow[dr, "\trop_{\Xi^k}"] \arrow[d] & \\
			\big(\bP \Omega^k \cM_{g}(1, \ldots, 1)_K \big)^\an \arrow[r, "\trop_{\Xi^k}"'] & \tropMSD \ .
		\end{tikzcd}
	\end{equation*}
	Hence, the image of $\trop_{\Xi^k}$ does not depend on the base field (whether larger or smaller than $\bC$).
\end{proof}

\subsection{Dimensions}

In the abelian case \cite[Theorem~6.6]{MUW17_published} shows that the realizability locus is a pure dimensional generalized cone complex of dimension equal to $4g-4 = \dim \bP \Omega \cM_{g}(1,\dots,1)$.
From \cite[Theorem~1.1]{Uli15} we know that the realizability locus for $k \geq 2$ must be a generalized cone complex of dimension $\leq \dim \bP \Omega^k \cM_{g}(1,\dots,1) = (2+2k)(g-1)-1$ (see the discussion in Section~\ref{subsec:realizability_problem}).
Let us now prove Theorem~\ref{thm:dimension_realizability_locus} from the introduction and show that this bound is in fact attained and all maximal cones have the same dimension.
To do so we need two preparational statements.
The proof for the following lemma is the same as \cite[Lemma~6.8]{MUW17_published}.

\begin{mylem} \label{lem:dimension_spaned_cone}
	For every realizable tropical normalized cover $\pi : \hat \Gamma^+ \to \Gamma^+$ let $\Gamma^+_0$ be the level graph obtained by successively contracting edges in $\Gamma^+$ that have an $(n+1)$-valent genus zero node with $n \geq 1$ marked points at one of its ends.
    The dimension of the cone in the realizability locus with associated normalized cover $\pi$ is $1$ less than the number of levels plus the number of horizontal edges of $\Gamma^+_0$.
\end{mylem}

\begin{myprop} \label{prop:maximal_cone}
	Let $k \geq 2$ and let $\pi : \hat \Gamma^+ \to \Gamma^+$ be a realizable tropical normalized cover. Then $\pi$ is contained in a cone of dimension $(2+2k)(g-1)-1$.
\end{myprop}
\begin{proof}
	Let $c(\pi)$ denote the number of levels minus $1$ plus the number of horizontal edges of $\Gamma^+$.
	As $\pi$ is realizable, the underlying cover of enhanced level graphs cuts out a boundary stratum $D_\pi \subseteq \PMSD[k]{g}{n}{1, \dots, 1}$ of codimension $c(\pi)$, i.e.\ for all multi-scale $k$-differentials $(\hat X \to X, \bs, \omega, \hat G^+ \to G^+) \in D_\pi$ the underlying cover of enhanced level graphs $\hat G^+ \to G^+$ agrees with $\pi$, see \cite[Proposition~1.3]{CMZ19}.
	We will prove that the closure $\overline D_\pi$ intersects a boundary stratum of maximal codimension, i.e.\ that there is a multi-scale $k$-differential $(\pi' : \hat X' \to X') \in \overline D_\pi$ with $c(\pi') = (2+2k)(g-1)-1$.
	
	Assume that we have found such a $\pi'$.
	We would like to use Lemma~\ref{lem:dimension_spaned_cone} to see that the tropicalization of $\pi'$ gives rise to a degeneration of $\pi$ that spans a cone of the claimed dimension.
	But $\pi'$ may contain trees of marked points as in Figure~\ref{fig:rational_tail}.
	(Here we assume that the $q_i$ are the zeros and poles coming from the half-edges where the tree is attached to the remaining graph.)
	To apply Lemma~\ref{lem:dimension_spaned_cone}, we need to rearrange any such tree of marked points.
	In a first step, we contract the tree as in Figure~\ref{fig:rational_tail}.
	The vertex obtained in this way will no longer be zero dimensional.

	\begin{figure}[htb]
		\centering
		\begin{minipage}{0.9\textwidth}
			\centering
			\begin{tikzpicture}[scale=1.5, decoration={markings, mark=at position .5 with {\arrow{>}}}]
	\fill 
	(0, 0) circle (2pt)
	(0.5, -0.5) circle (2pt)
	(1, -1) circle (2pt);
	
	\path[draw] (0,0) -- (0.6, -0.6);
	\path[draw] (0.9,-0.9) -- (1,-1);
	\path[draw, dotted] (0.6, -0.6) -- (0.9,-0.9);
	
	\path[draw] (0,0) --++ (135:0.3);
	\draw (-0.3, 0.35) node {$q_1$};
	\draw (0, 0.35) node {$\cdots$};
	\path[draw] (0,0) --++ (45:0.3);
	\draw (0.3, 0.35) node {$q_n$};
	\path[draw] (0.5,-0.5) --++ (-135:0.3);
	\draw (0.2, -0.8) node {1};
	\path[draw] (1,-1) --++ (-135:0.3);
	\draw (0.7, -1.3) node {1};
	\path[draw] (1,-1) --++ (-45:0.3);
	\draw (1.3, -1.3) node {1};

        \begin{scope}[xshift=70]
	\fill 
	(0, 0) circle (2pt);
	
	\path[draw] (0,0) --++ (135:0.3);
	\draw (-0.3, 0.35) node {$q_1$};
	\path[draw] (0,0) --++ (45:0.3);
	\draw (0, 0.35) node {$\cdots$};
	\draw (0.3, 0.35) node {$q_n$};
	\path[draw] (0,0) --++ (-135:0.3);
	\draw (-0.3, -0.3) node {1};
	\draw (0, -0.3) node {$\cdots$};
	\path[draw] (0,0) --++ (-45:0.3);
	\draw (0.3, -0.3) node {1};
        \end{scope}

	\draw[->,decorate,decoration={snake,amplitude=.5mm}] (1, 0) -- (1.8, 0);
\end{tikzpicture}
		\end{minipage}
		\caption{Contracting a rational tail.}
		\label{fig:rational_tail}
	\end{figure}

	We will now explain how to degenerate this vertex into a chain of zero dimensional vertices without producing a tree of marked points.
	Our argument works by induction on the dimension of the vertex, i.e.\ we will explain how to split off a zero dimensional vertex which will reduce the dimension by one.
	We may assume $q_1 \leq \cdots \leq q_n$ and distinguish several cases. \medskip

	\textbf{Case $q_n > -k$:} we may degenerate as in Figure~\ref{subfig:rational_tail_zero}.
	Hereby we introduce a new level into the graph, such that the vertex containing $q_n$ is the only vertex in its level.
	If $k$ divides $q_n + 1$, then the vertex containing $q_n$ will be inconvenient by Proposition~\ref{prop:residue_map_origin_missing}~(iv). In this case any two preimages of the newly formed edge form an effective cycle for the pole and allow to choose a non-zero $k$-residue, thus redeeming the inconvenient vertex.
	Independently of this special case, the vertex on top level now has the orders $1, \dots, 1, q_1, \dots, q_{n-1}, q'_n$ with $q'_n > q_n$ at the marked points, and we can repeat the process until all vertices are of dimension zero.
	
	At last, we remark that any inconvenient vertex in the rest of the graph that was redeemed by an effective cycle or a pair of independent cycles before degenerating the vertex is still redeemed by essentially the same cycles afterwards: If the cycle was using the half-edge $q_1$, it needs to be replaced with the obvious cycle that contains the newly formed edge.
	The same will be true in all other cases we discuss below, and we will not repeat this argument for each case.
	So from now on we can assume $q_n \leq -k$.
	\medskip

	\textbf{Case $q_1 < -k$:} we distinguish three cases. In each case, we give a degeneration that essentially replaces $q_1$ by a pole of lower order, i.e.\ by a $q'_1 > q_1$.
	\begin{itemize}
		\item If $q_1 < -k-1$ and $k \nmid q_1$, we can always degenerate as in Figure~\ref{subfig:rational_tail1}.
			Hereby we introduce a new level into the graph, such that the vertex containing $q_1$ is the only vertex on its level.
			Note that this vertex is never illegal, and the $k$-residue at the pole at $q_1$ is always zero, cf.~the discussion preceeding Lemma~\ref{lem:induced_k_residues}. In particular, this vertex is never inconvenient.
			Every cycle that was used to redeem any inconvenient vertex in the graph still exists and continues to be ``at or above level'' for the respective vertex.
		\item If $q_1 < -k-1$ and $k \mid q_1$, we need to distinguish two sub-cases.
			\begin{itemize}
				\item If there is an effective cycle $\gamma$ ``at or above level'' for one of the preimages of $q_1$ that only uses other preimages of $q_1$ to come back to its lowest level, we claim that we can always degenerate as in the previous case, i.e.\ as in Figure~\ref{subfig:rational_tail1}.
					Note that the new vertex will be inconvenient by Proposition~\ref{prop:residue_map_origin_missing}~(iv), but the cycle $\gamma$ ensures that we can choose a non-zero $k$-residue for the pole at $q_1$.
					The vertex containing the remaining $q_2, \dots, q_n$ may be inconvenient, too. Note that the pole at the newly created edge is not divisible by $k$. By checking the list of all inconvenient vertices, we see that there is only one possible way for a vertex with a pole that is not divisible by $k$ to be inconvenient for $k\geq2$, namely Proposition~\ref{prop:residue_map_origin_missing}~(iii).
					This is only possible if $q_2, \dots, q_n$ are divisible by and strictly less than $k$.
					If for any of those an effective cycles exists, then we can redeem the inconvenient vertex.
					If no such cycles exist, we instead degenerate as in the next case for with $q_2$ as the new $q_1$.
				\item If there is no such cycle $\gamma$, then we can degenerate as in Figure~\ref{subfig:rational_tail1b}. Note that if $k=2$ and $q_1 = -4$, the image is not accurate: in this case the newly created edge will be horizontal. Moreover, in this case the new vertex is inconvenient by Figure~\ref{fig:residue_map_line_missing}, line~6, but can be redeemed by the cycles given by the preimages of the newly formed horizontal edge. In all other cases, the vertex containing $q_1$ will not be inconvenient. A priori this vertex appears to have positive dimension. However, the non-existence of the cycle $\gamma$ forces the $k$-residue at $q_1$ to be zero. This is an extra condition on the $k$-residues, and hence the stratum corresponding to the new vertex is in fact zero dimensional.
			\end{itemize}
		\item If $q_1 = -k-1$, then we may degenerate as depicted in Figure~\ref{subfig:rational_tail2}. The preimages of the new horizontal edge provide all the necessary cycles.
	\end{itemize}
	\medskip 

	\textbf{Case $q_1 = \cdots = q_n = -k$.}
	Observe that this case only needs to be considered for $n \geq 3$ for dimensional reasons.
	If $n \geq 4$, we can degenerate the vertex into two vertices connected by a new horizontal edge, where one of the new vertices contains $q_1, q_2$ and $k$ of the legs of order $1$, and the other vertex contains the remaining marked points.	
	Let us now discuss the case $n = 3$.
	
	Let $v$ be the vertex containing the marked points $q_1,q_2,q_3$ (and a bunch of marked points of order 1), and let $\pi^\top : \hat G^\top \to G^\top$ be the sub-cover of $\pi'$ consisting of all vertices (and all edges between those vertices) at or above the level of $v$.
	Moreover, let $\pi^\top_\circ : \hat G^\top_\circ \to G^\top_\circ$ be the sub-cover of $\pi^\top$ obtained by removing $v$, all preimages of $v$ and all incident edges.
	We claim that the number of connected components of $G^\top_\circ$ is strictly larger than the number of connected components of $G^\top$.
	Without loss of generality that $G^\top$ is connected.
	For the sake of the argument, let us ``undo'' the contraction of the rational tail, i.e. we split off all the marked points of order $1$ adjacent to $v$ in a rational tail.
	We denote the vertex incident to the marked points $q_i$ by $v'$.
	The vertex $v'$ has four marked points and must be zero dimensional.
	If $G^\top_\circ$ were connected, then there would be a cycle in $\hat G^\top$ from each preimage $q_i$ to each preimage of $q_j$ for all $i, j$.
	Observe that those cycles would necessarily be effective for the $q_i$, as all relevant edges are horizontal.
	This implies that we can choose the ($k$-)residues at each $q_i$ independently.
	But this would imply that the vertex $v'$ is in fact not zero dimensional, but one dimensional (recall that it has four marked points), a contradiction.
	Hence $G^\top_\circ$ must be disconnected.%
	\footnote{
		We remark that the converse statement is false in general: If $G^\top_\circ$ is disconnected, the vertex might still be one dimensional.
	}
	By applying the same argument again we may assume that in $G^\top$ there is no simple cycle at or above the level of $v$ that uses a preimage of $q_1$ and of either $q_2$ or $q_3$.
	Let $G^\top_1$ be the collection of connected components of $G^\top_\circ$ which are adjacent to $q_1$ in $G^\top$.
	We may now degenerate our vertex $v$ as in Figure~\ref{subfig:rational_tail1}. While doing so, we move the subgraph $G^\top_1$ ``on level up'', such that $q_1$ remains at its lowest level (i.e.\ all edges $q_i$ remain to be horizontal). 
	
	This completes the case distinction. \medskip

	\begin{figure}[htb]
		\centering
		\begin{minipage}{0.9\textwidth}
			\begin{center}
				\hfill
\begin{subfigure}{.4\textwidth}
	\centering
	\begin{tikzpicture}[scale=1.5, decoration={markings, mark=at position .5 with {\arrow{>}}}]
		\fill
		(0, 0) circle (2pt)
		(-1, -0.5) circle (2pt);

		\path[draw] (0,0) -- (-1, -0.5);

		\path[draw] (0,0) --++ (135:0.3);
		\draw (-0.3, 0.35) node {$q_1$};
		\path[draw] (0,0) --++ (45:0.3);
		\draw (0, 0.35) node {$\cdots$};
		\draw (0.43, 0.35) node {$q_{n-1}$};
		\path[draw] (0,0) --++ (-135:0.3);
		\draw (-0.3, -0.3) node {1};
		\draw (0, -0.3) node {$\cdots$};
		\path[draw] (0,0) --++ (-45:0.3);
		\draw (0.3, -0.3) node {1};

		\path[draw] (-1,-0.5) --++ (135:0.3);
		\draw (-1.3, -0.15) node {$q_n$};
		\path[draw] (-1,-0.5) --++ (-135:0.3);
		\draw (-1.3, -0.8) node {$1$};
	\end{tikzpicture}
	\caption{The case $q_n > -k$.}
	\label{subfig:rational_tail_zero}
\end{subfigure}
\hfill
\begin{subfigure}{.4\textwidth}
	\centering
	\begin{tikzpicture}[scale=1.5, decoration={markings, mark=at position .5 with {\arrow{>}}}]
		\fill
		(0, 0) circle (2pt)
		(-1, 0.5) circle (2pt);

		\path[draw] (0,0) -- (-1, 0.5);

		\path[draw] (0,0) --++ (135:0.3);
		\draw (-0.3, 0.35) node {$q_2$};
		\path[draw] (0,0) --++ (45:0.3);
		\draw (0, 0.35) node {$\cdots$};
		\draw (0.3, 0.35) node {$q_n$};
		\path[draw] (0,0) --++ (-135:0.3);
		\draw (-0.3, -0.3) node {1};
		\draw (0, -0.3) node {$\cdots$};
		\path[draw] (0,0) --++ (-45:0.3);
		\draw (0.3, -0.3) node {1};

		\path[draw] (-1,0.5) --++ (135:0.3);
		\draw (-1.3, 0.85) node {$q_1$};
		\path[draw] (-1,0.5) --++ (-135:0.3);
		\draw (-1.3, 0.2) node {$1$};
	\end{tikzpicture}
	\caption{The case $q_1 < -k -1$ (with non-zero $k$-residue if $k \mid q_1$).}
	\label{subfig:rational_tail1}
\end{subfigure}
\hfill
\\
\hfill
\begin{subfigure}{.4\textwidth}
	\centering
	\begin{tikzpicture}[scale=1.5, decoration={markings, mark=at position .5 with {\arrow{>}}}]
		\fill
		(0, 0) circle (2pt)
		(-1, 0.5) circle (2pt);

		\path[draw] (0,0) -- (-1, 0.5);

		\path[draw] (0,0) --++ (135:0.3);
		\draw (-0.3, 0.35) node {$q_2$};
		\path[draw] (0,0) --++ (45:0.3);
		\draw (0, 0.35) node {$\cdots$};
		\draw (0.3, 0.35) node {$q_n$};
		\path[draw] (0,0) --++ (-135:0.3);
		\draw (-0.3, -0.3) node {1};
		\draw (0, -0.3) node {$\cdots$};
		\path[draw] (0,0) --++ (-45:0.3);
		\draw (0.3, -0.3) node {1};

		\path[draw] (-1,0.5) --++ (135:0.3);
		\draw (-1.3, 0.85) node {$q_1$};
		\path[draw] (-1,0.5) --++ (-135:0.3);
		\draw (-1.3, 0.2) node {$1$};
		\path[draw] (-1,0.5) --++ (-45:0.3);
		\draw (-0.7, 0.2) node {1};
	\end{tikzpicture}
	\caption{The case $q_1 < -k -1$ with $k \mid q_1$ and zero $k$-residue.}
	\label{subfig:rational_tail1b}
\end{subfigure}
\hfill
\begin{subfigure}{.4\textwidth}
	\centering
	\begin{tikzpicture}[scale=1.5, decoration={markings, mark=at position .5 with {\arrow{>}}}]
		\fill
		(0, 0) circle (2pt)
		(-1, 0) circle (2pt);

		\path[draw] (0,0) -- (-1, 0);

		\path[draw] (0,0) --++ (135:0.3);
		\draw (-0.3, 0.35) node {$q_2$};
		\path[draw] (0,0) --++ (45:0.3);
		\draw (0, 0.35) node {$\cdots$};
		\draw (0.3, 0.35) node {$q_n$};
		\path[draw] (0,0) --++ (-135:0.3);
		\draw (-0.3, -0.3) node {1};
		\draw (0, -0.3) node {$\cdots$};
		\path[draw] (0,0) --++ (-45:0.3);
		\draw (0.3, -0.3) node {1};

		\path[draw] (-1,0) --++ (135:0.3);
		\draw (-1.3, 0.35) node {$q_1$};
		\path[draw] (-1,0) --++ (-135:0.3);
		\draw (-1.3, -0.3) node {$1$};
	\end{tikzpicture}
	\caption{The case $q_1 = -k -1$.}
	\label{subfig:rational_tail2}
\end{subfigure}
\hfill
			\end{center}
		\end{minipage}
		\caption{Rearranging a rational tail}
		\label{fig:rational_tail_expansion}
	\end{figure}

	We will prove the existence of the degeneration $\pi'$ by induction on $c(\pi)$. Assume that $c(\pi)$ is not maximal, i.e.\ that $c(\pi) < (2+2k)(g-1)-1$. Let us prove that $\overline D_\pi$ intersects a boundary stratum of higher codimension. To this end, let for a level $L$ of $\Gamma^+$
	\[
		\sigma_{L} : D_\pi \longrightarrow \left(\prod_{v \in L} \Omega^k \cM_{g(v)}(\mu(v)) \right) \Big/ \bC^\times
	\]
	be the map that cuts out level $L$.
	(Note that projectivization of strata of multi-scale differentials is done with respect to the diagonal $\bC^\times$-action, and not with respect to the $\bC^\times$ action on each irreducible component.)
	Since $c(\pi)$ was assumed to not be maximal,
	there exists a level $L$ such that $\dim \sigma_{L}(D_\pi) \geq 1$.
	We need to argue that such a level $L$ can always be degenerated.

	If $\sigma_L(D_\pi)$ consists of multiple connected components that may be rescaled independently, then we obtain an additional level by taking the limit of this rescaling, i.e.\ by distributing the vertices of level $L$ to two levels according to the rescaling.
	If $\sigma_L(D_\pi)$ does not consist of multiple connected components (or those cannot be rescaled independently), then there must be an irreducible component of positive (projective) dimension in $\sigma_L(D_\pi)$.
	For ease of notation we assume that $\sigma_L(D_\pi)$ consists of only one such component, i.e.
	\begin{equation} \label{eq:levelsubspace}
		\sigma_L(D_\pi) \subseteq \bP \Omega^k \cM_{g(v')}(\mu(v')).
	\end{equation}
        The map $\sigma_L$ can be extended to the closure $\overline D_\pi$ using an appropriate moduli space of multi-scale $k$-differentials as codomain for the extended map.
        If $\sigma_L(\overline D_\pi)$ contains a point from the boundary of $\bP \Omega^k \cM_{g(v')}(\mu(v'))$, then this point gives rise to a degeneration of the level $L$.
        Such a point always exists: The image $\sigma_L(\overline D_\pi)$ is a complete variety and the stratum $\bP \Omega^k \cM_{g(v')}(\mu(v'))$ on the right hand side of \eqref{eq:levelsubspace} does not contain any complete variety by \cite[Th\'{e}or\`{e}me~1 and Corollaire~2]{Gen20}. This concludes the proof.
\end{proof}

\begin{myrem}
	The problem of rearranging tails of rational curves appeared for $k=1$ in the proof of \cite[Proposition~6.9]{MUW17_published}, but the proof in loc.\ cit. is incomplete at this point.
	For the sake of completeness, we explain how to rearrange tails of rational curves for $k=1$ using a similar case distinction as in the above proof.
	Assume that we arrive in the situation on the right hand side of Figure~\ref{fig:rational_tail} after contracting a rational tail.
	As in the above proof, we assume that $q_1 \leq \cdots \leq q_n$.
	If $q_n > -1$, we can degenerate is in Figure~\ref{subfig:rational_tail_zero}.
	So assume that $q_n \leq -1$. We distinguish three cases.
	\begin{itemize}
		\item If $q_1 < -2$ we can always degenerate is in Figure~\ref{subfig:rational_tail1}.
		\item If $q_1 = -2$ we have to consider two subcases.
			\begin{itemize}
				\item If there is a cycle ``at or above level'' through $q_1$, then we may degenerate as in Figure~\ref{subfig:rational_tail2}. Here we remind the reader that for $k=1$ all cycles are effective.
				\item If there is no cycle ``at or above level'' through $q_1$, then we may degenerate as in Figure~\ref{subfig:rational_tail1b}. Observe that none of the vertices is inconvenient.
			\end{itemize}
		\item If $q_1 = -1$ (i.e. $q_1 = \cdots = q_n = -1$) and $n \geq 4$, we can degenerate as described in the above proof for the case $q_1 = \cdots = q_n = -k$ and $n \geq 4$.
			The cases $n=2,3$ are impossible for dimensional reasons.
	\end{itemize}
\end{myrem}

\begin{proof}[Proof of Theorem~\ref{thm:dimension_realizability_locus}]
	The dimension of the realizability locus is bounded from above by the dimension of the domain of the tropicalization map, i.e.\ by $\dim \PMSD[k]{g}{n}{1, \dots, 1} = (2+2k)(g-1)-1$.
	That this bound is actually obtained and all maximal cones are of the expected dimension follows from the previous proposition.
\end{proof}

\begin{myrem}
	We emphasize that for $k=1$ the above formula does not give the correct dimension for the maximal cones. This is due to the formula for the dimension of the principal stratum being different in the abelian case.
\end{myrem}

\subsection{Obstructions to realizability}

The following are two simple criteria which can be used to recognize non-realizable tropical normalized covers. An application is illustrated in Section \ref{sec:example} below.

\begin{mycor} \label{cor:contracting_horizontal_edges}
	Let $\pi : \hat \Gamma^+ \to \Gamma^+$ be a tropical normalized cover with enhancements associated by Lemma~\ref{lem:def_trop_normalized_cover_is_redundand}. Let $e$ be a horizontal edge in $\Gamma^+$ and denote by $\pi / \{e\}$ the tropical normalized cover obtain from $\pi$ by contracting $e$ in the base and every $\hat e \in \pi^{-1}(e)$. If $\pi$ is realizable, then so is $\pi / \{e\}$. 
\end{mycor}

\begin{proof}
	Let $\hat X \to X$ be a realization of $\pi$.
	Smoothing of a horizontal edge of $\hat X \to X$ is always possible, see \cite[Section~3.1]{CMZ19}, and produces a realization of $\pi / \{e\}$.
\end{proof}

With the same proof, we get

\begin{mycor} \label{cor:smoothing_level}
	In the situation of Corollary \ref{cor:contracting_horizontal_edges} let $E'$ be the set of all edges in $\Gamma^+$ connecting two neighboring levels. If $\pi$ is realizable, then so is $\pi / E'$.
\end{mycor}

\begin{myexa} \label{exa:contracting_level}
	\begin{figure}[htb]
		\centering
		\begin{tabular}{c c}
			\begin{minipage}{0.5\textwidth}
				\centering
				\begin{tikzpicture}[scale=1.5, decoration={markings, mark=at position .5 with {\arrow{>}}}]
	\fill 
	(-1, 0) circle (2pt)
	(0.5, 0) circle (2pt)
	(1.5, 0) circle (2pt);
	\draw
	(-0.5, 1) node[draw, thick, shape = circle, inner sep = 1pt, anchor = south] {1}
	(-1.5, 1) node[draw, thick, shape = circle, inner sep = 1pt, anchor = south] {1}
	(0.5, 1) node[draw, thick, shape = circle, inner sep = 1pt, anchor = south] {1}
	(1.5, 1) node[draw, thick, shape = circle, inner sep = 1pt, anchor = south] {1};
	
	\path[draw] (-1,0) to (-1.5, 1);
	\path[draw] (-1,0) to (-0.5, 1);
	\path[draw] (-1,0) to (0.5, 1);
	\path[draw] (-1,0) to (1.5, 1);
	\path[draw] (0.5,0) to (0.5, 1);
	\path[draw] (0.5,0) to (1.5, 1);
	\path[draw] (1.5,0) to (0.5, 1);
	\path[draw] (1.5,0) to (1.5, 1);
	
	\path[draw] (-1,0) --++ (-45:0.3);
	\draw (-0.7, -.3) node {4};
	\path[draw] (-1,0) --++ (-135:0.3);
	\draw (-1.3, -.3) node {2};
	\path[draw] (0.5,0) --++ (0,-0.25);
	\draw (0.5, -.4) node {2};
	\path[draw] (1.5,0) --++ (0,-0.25);
	\draw (1.5, -.4) node {2};
	
	\draw (-2.5, .5) node {$\hat G_1^+$};

	\draw (-1, -1) node[draw, thick, shape = circle, inner sep = 1pt, anchor = south] {1};
	\draw (1, -1) node[draw, thick, shape = circle, inner sep = 1pt, anchor = south] {1};
	
	\path[draw] (-1, -2) to (-1, -1);
	\path[draw] (-1, -2) to (1, -1);
	\path[draw] (1, -2) to[bend left = 30] (1, -1) to[bend left = 30] (1, -2);
	
	\fill (-1, -2) circle (2pt);
	\fill (1, -2) circle (2pt);
	
	\path[draw] (-1,-2) --++ (-45:0.3);
	\draw (-0.7, -2.3) node {3};
	\path[draw] (-1,-2) --++ (-135:0.3);
	\draw (-1.3, -2.3) node {1};
	\path[draw] (1,-2) --++ (0,-0.25);
	\draw (1, -2.4) node {4};
	
	\draw (-2.5, -1.5) node {$G_1^+$};
	
\end{tikzpicture}
			\end{minipage}
			&
			\begin{minipage}{0.5\textwidth}
				\centering
				\begin{tikzpicture}[scale=1.5, decoration={markings, mark=at position .5 with {\arrow{>}}}]
	\fill 
	(-1, 0) circle (2pt);
	\draw
	(-0.5, 1) node[draw, thick, shape = circle, inner sep = 1pt, anchor = south] {1}
	(-1.5, 1) node[draw, thick, shape = circle, inner sep = 1pt, anchor = south] {1}
	(1, 1) node[draw, thick, shape = circle, inner sep = 1pt, anchor = south] {3};
	
	\path[draw] (-1,0) to (-1.5, 1);
	\path[draw] (-1,0) to (-0.5, 1);
	\path[draw] (-1,0) to (1, 1);
	\path[draw] (-1,0) to (1, 1);
	
	\path[draw] (-1,0) --++ (-45:0.3);
	\draw (-0.7, -.3) node {4};
	\path[draw] (-1,0) --++ (-135:0.3);
	\draw (-1.3, -.3) node {2};
	\path[draw] (1,1) --++ (-45:0.3);
	\draw (0.7, 0.7) node {2};
	\path[draw] (1,1) --++ (-135:0.3);
	\draw (1.3, 0.7) node {2};
	
	\draw (-2.5, .5) node {$\hat G_2^+$};

	\draw (-1, -1) node[draw, thick, shape = circle, inner sep = 1pt, anchor = south] {1};
	\draw (1, -1) node[draw, thick, shape = circle, inner sep = 1pt, anchor = south] {2};
	
	\path[draw] (-1, -2) to (-1, -1);
	\path[draw] (-1, -2) to (1, -1);
	
	\fill (-1, -2) circle (2pt);
	
	\path[draw] (-1,-2) --++ (-45:0.3);
	\draw (-0.7, -2.3) node {3};
	\path[draw] (-1,-2) --++ (-135:0.3);
	\draw (-1.3, -2.3) node {1};
	\path[draw] (1,-1) --++ (0,-0.25);
	\draw (1, -1.4) node {4};
	
	\draw (-2.5, -1.5) node {$G_2^+$};
	
\end{tikzpicture}
			\end{minipage}
		\end{tabular}
		\caption{A realizable graph and a non-realizable undegeneration}
		\label{fig:smoothing_level}
	\end{figure}

	We emphasize that in the situation of Corollary~\ref{cor:smoothing_level} only complete level passages may be smoothed.
	We give an example for this in terms of enhanced level graphs that can easily be imagined as the top most levels of enhanced level graphs associated to a tropical normalized cover by Lemma~\ref{lem:def_trop_normalized_cover_is_redundand}.
	Consider the cover of graphs $\hat G_1^+ \to G_1^+$ on the left of Figure~\ref{fig:smoothing_level}. We assume that the $o$-value at the top end of all edges is $0$. Both vertices on bottom level of $G_1^+$ are inconvenient. In $\hat G_1^+$, there is precisely one effective cycle (up to the $\tau$-action) and this cycle is admissible for both inconvenient vertices.
	Therefore, this cover is realizable by Theorem~\ref{thm:main_theorem}.
	After smoothing only some of the edges between top and bottom level, we obtain the cover $\hat G_2^+ \to G_2^+$ on the right.
	This is no longer realizable: There is no effective cycle (and hence no admissible cycle or independent pair of cycles) for the inconvenient vertex on bottom level. And besides that, the vertex of genus $2$ is illegal.

	This example also highlights another aspect. As we have seen in Example~\ref{ex:pair_of_cycles} there are inconvenient vertices of type~II for which there is no admissible cycle, but a pair of independent cycles. The inconvenient vertex of genus zero on the left of $G_1^+$ is inconvenient of type~II, but in this case there exists only an admissible cycle and no independent pair of cycles. Thus in fact both situations may occur.
\end{myexa}

\section{Examples}
\label{sec:example}

\subsection{$kK_\Gamma$ is always realizable}
\label{subsec:kK_Gamma_realizable}

Consider a stable tropical curve $\Gamma$ with divisor $kK_\Gamma$ for some $k \geq 2$. We show that the pair $(\Gamma, kK_\Gamma)$ is realizable\footnote{The authors are very grateful to D. Maclagan who raised this question during a discussion at the 2021 edition of the conference ``Effective Methods in Algebraic Geometry''.}.
Note that stability implies that every vertex is in the support of the divisor. Hence the construction from Lemma \ref{lem:construction_of_k-enhancement} produces for every vertex at least one incident leg with enhancement value 1. Thus when constructing a tropical normalized cover $\hat \Gamma \to \Gamma$, every vertex of $\Gamma$ has necessarily only a single preimage. Furthermore, all edges in $\Gamma^+$ are horizontal. Hence the only possible choice for $\hat\Gamma$ is to replace each edge of $\Gamma$ with $k$ parallel edges. But now the conditions of Theorem \ref{thm:main_theorem} are easy to check: the necessary cycles are provided by the parallel edges.

\subsection{Dumbbell graph}
\label{subsec:example_dumbbell_graph}

Let $k = 2$ and consider the \emph{dumbbell graph} $\Gamma$ consisting of two vertices of genus~0, connected with a bridge edge and having a self-loop at each vertex, see Figure \ref{fig:dumbbel_graph}. We claim that Figure \ref{fig:realizability_dumbbell_graph} shows all maximal cones in the realizability locus over the dumbbell graph. Note that each of them is of dimension 5 as was expected by Theorem~\ref{thm:dimension_realizability_locus}.

\begin{figure}[htb]
	\centering
	\begin{minipage}{0.9\textwidth}
		\centering
		\begin{tikzpicture}[scale=1.5]
			\draw (-0.5, 0) circle [radius = 0.5];
			\path[draw] (0, 0) -- ++ (2, 0);
			\draw (2.5, 0) circle [radius = 0.5];			
			
			\fill
			(0,0) circle [radius = 2pt]
			(2,0) circle [radius = 2pt];
		\end{tikzpicture}
	\end{minipage}
	\caption{Dumbbell graph.}
	\label{fig:dumbbel_graph}
\end{figure}

To verify our claim let us start with the divisor $2K_\Gamma$. We focus on one of the trivalent vertices $v$: notice that the divisor produces two legs at $v$. How can we move these to arrive at a combinatorial type with more degrees of freedom? The first option is to move both legs onto the incident self-loop. Necessarily they will be symmetric on the loop. Performing this move on both sides of the graph produces the configuration from Figure \ref{subfig:Dumbbell_1}. This is realizable by the same argument as in the previous Section \ref{subsec:kK_Gamma_realizable}: we may simply cover every vertex by a unique preimage, producing cycles above every horizontal edge. In this case there are no inconvenient vertices.

The second option would be to move a single leg onto the bridge and leave the other one at $v$. Assume the resulting situation was realizable. Contract the self-loop at $v$ and obtain a vertex of genus 1 and type $(-1, 1)$. This is an illegal vertex and by Corollary \ref{cor:contracting_horizontal_edges} we obtain a contradiction. Hence, we cannot leave a single leg behind.

So the third an final option is to move both legs onto the bridge. This will leave a part of the graph that is depicted in Figure \ref{fig:dumbbel_graph_split_component} behind. Note that any cover of this part of the graph must be disconnected. This can been seen directly with Corollary \ref{cor:contracting_horizontal_edges} again. Alternatively, observe that the vertex has to have two preimages (otherwise it would be illegal). Connecting these two copies with two parallel horizontal edges fails to provide an effective cycle above the (horizontal) self-loop on the base (see Example~\ref{exa:non-effective_cycle}). Hence, each of the two instances of the vertex must have a self-loop attached.

\begin{figure}[htb]
	\centering
	\begin{minipage}{0.9\textwidth}
		\centering
		\begin{tikzpicture}[scale=1.5]
			\draw (1, 0) ellipse [x radius = 1, y radius = 0.2];
			\path[draw] (0, 0) -- ++ (-0.25, -0.5);		
			
			\draw 
			(0.2, 0.4) node {$-2$}
			(0.2, -0.4) node {$-2$}
			(-0.3, -0.7) node {0};
		\end{tikzpicture}
	\end{minipage}
	\caption{Part of an enhanced dumbbell graph. This is only realizable when provided with a disconnected cover.}
	\label{fig:dumbbel_graph_split_component}
\end{figure}

With this observation in mind, we may choose to keep the pair of legs that we just moved onto the bridge together or separate them. In the former case, we arrive at situations from Figure \ref{subfig:Dumbbell_2} or \ref{subfig:Dumbbell_3}. In each case there is a unique tropical normalized cover $\hat \Gamma^+$ and it can be checked to satisfy the conditions from Theorem \ref{thm:main_theorem}. Hence, both of these configurations are realizable. The other option does however produce an inconvenient vertex of type $(1,-1; -4)$. This vertex does not have any simple closed cycle above it, which violates the conditions of Theorem \ref{thm:main_theorem}.

Finally, we may push three or four of the legs onto the same self-loop. The first option does not produce a realizable configuration: Similar to the cases discussed above, the leg left behind always produces an inconvenient vertex that is not redeemed by an appropriate cycle. So let us consider the case that all four legs have been pushed to the same self-loop. This gives a realizable configuration if and only if the vertices are pairwise at the same level, as depicted in Figure \ref{subfig:Dumbbell_4}. We check this by providing a realizable cover $\hat \Gamma^+$ in Figure~\ref{fig:realizable_dumbbell_graph_with_cover}. In fact, taking into consideration what we discussed about Figure~\ref{fig:dumbbel_graph_split_component} this is the only possible normalized cover without illegal vertices. Notice that there are again two inconvenient vertices of type $(-1, 1; -4)$ on the base. But the cover does admit a simple closed cycle above them which can be checked to be effective. We emphasize that this cycle passes through both vertices of type $(-1,1;-4)$. Thus if those vertices were not on the same level, then there would not exist a simple cycle ``at or above level'' for the higher of the two. By Theorem~\ref{thm:main_theorem} this would contradict realizability.

\begin{figure}[htb]
	\centering
	\begin{minipage}{0.9\textwidth}
		\begin{center}
			\input{tikzpicture_realizability_dumbbell_graph}
		\end{center}
	\end{minipage}
	\caption{Realizability locus over the dumbbell graph. Vertices of $\Gamma^+$ connected by dashed lines lie on the same level.}
	\label{fig:realizability_dumbbell_graph}
\end{figure}

\begin{figure}[htb]
	\centering
	\begin{minipage}{0.9\textwidth}
		\begin{center}
			\begin{tikzpicture}[scale=2]
\tikzstyle{every node}=[font=\normalsize]

\def \xrad{0.3}
\def \yrad{.05}

\coordinate (center1) at (0.6, -.5);
\coordinate (center2) at (1.7,-.5);
\coordinate (trivalent1) at ($(center1) + (.6, .0)$);
\coordinate (trivalent2) at ($(center2) - (\xrad, 0)$);

\coordinate (D1) at ($(center1) + (80:0.6 cm and \yrad cm)$);
\coordinate (D2) at ($(center1) + (120:0.6 cm and \yrad cm)$);
\coordinate (D3) at ($(center1) + (260:0.6 cm and \yrad cm)$);
\coordinate (D4) at ($(center1) + (300:0.6 cm and \yrad cm)$);

\def \heightLoopLeft{0.25}
\def \heightLoopRight{1.2}
\coordinate (centerBase1) at ($(center1) + (0,\heightLoopLeft)$);
\coordinate (centerBase2) at ($(center2) + (0,\heightLoopRight)$);

\coordinate (D2base) at ($(D2) + (0,\heightLoopLeft)$);
\coordinate (D3base) at ($(D3) + (0,\heightLoopLeft)$);
\coordinate (D1base) at ($(D1) + (0,\heightLoopLeft + .15)$);
\coordinate (D4base) at ($(D4) + (0,\heightLoopLeft + .15)$);

\def \heightBridgeEdge{0.30}
\coordinate (trivalent1base) at ($(trivalent1) + (0,\heightBridgeEdge + 0.6)$);
\coordinate (trivalent2base) at ($(trivalent2) + (0,\heightLoopRight)$);

\draw (1.3,-.7) node[] {\phantom{}}; 
\draw (2.3,1.3) node[] {\phantom{}}; 

\draw (D2base) arc (120:260:0.6 cm and \yrad cm);			
\path[draw] (D2base) -- (D1base) -- (trivalent1base) -- (D4base) -- (D3base);
\path[draw] (trivalent1base) -- (trivalent2base);
\draw (centerBase2) ellipse (\xrad cm and \yrad cm);

\fill
(D1base) circle (1pt)
(D2base) circle (1pt)
(D3base) circle (1pt)
(D4base) circle (1pt)
(trivalent1base) circle (1pt)
(trivalent2base) circle (1pt);

\path[draw] (D1base) [] -- ++(-45:.15);		
\path[draw] (D2base) [] -- ++(-45:.15);		
\path[draw] (D3base) [] -- ++(-45:.15);		
\path[draw] (D4base) [] -- ++(-45:.15);		

\def \heightLoopLeft{1.25}
\def \heightLoopRight{1.7}
\def \doubleVertexOffset{0.08}
\coordinate (centerCover1) at ($(center1) + (0,\heightLoopLeft)$);
\coordinate (centerCover2A) at ($(center2) + (0,\heightLoopRight)$);
\coordinate (centerCover2B) at ($(center2) + (0,\heightLoopRight + \doubleVertexOffset)$);

\coordinate (D2cover) at ($(D2) + (0,\heightLoopLeft)$);
\coordinate (D3cover) at ($(D3) + (0,\heightLoopLeft)$);
\coordinate (D1cover) at ($(D1) + (0,\heightLoopLeft + .15)$);
\coordinate (D4cover) at ($(D4) + (0,\heightLoopLeft + .15)$);

\def \heightBridgeEdge{1}
\coordinate (trivalent1coverA) at ($(trivalent1) + (0,\heightBridgeEdge + 0.6)$);
\coordinate (trivalent1coverB) at ($(trivalent1) + (0,\heightBridgeEdge + 0.6 + \doubleVertexOffset)$);
\coordinate (trivalent2coverA) at ($(trivalent2) + (0,\heightLoopRight)$);
\coordinate (trivalent2coverB) at ($(trivalent2) + (0,\heightLoopRight + \doubleVertexOffset)$);

\draw[double, double distance=1pt] (D2cover) arc (120:260:0.6 cm and \yrad cm);			
\path[draw,dotted,thick] (D1base) -- (D1cover) -- (D4cover) -- (D4base) -- (D1base);
\path[draw,dotted,thick] (D2base) -- (D2cover) -- (D3cover) -- (D3base) -- (D2base);
\path[draw] (D2cover) -- (D1cover);
\path[draw] (D4cover) -- (D3cover);
\path[draw] (D1cover) -- (trivalent1coverA) -- (D4cover);
\path[draw] (D1cover) -- (trivalent1coverB) -- (D4cover);
\draw[dotted,thick] (trivalent1coverB) -- (trivalent1base);
\path[draw] (trivalent1coverA) -- (trivalent2coverA);
\path[draw] (trivalent1coverB) -- (trivalent2coverB);
\draw (centerCover2A) ellipse (\xrad cm and \yrad cm);
\draw (centerCover2B) ellipse (\xrad cm and \yrad cm);
\draw[dotted,thick] (trivalent2coverB) -- (trivalent2base);

\fill
(D1cover) circle (1pt)
(D2cover) circle (1pt)
(D3cover) circle (1pt)
(D4cover) circle (1pt)
(trivalent1coverA) circle (1pt)
(trivalent1coverB) circle (1pt)
(trivalent2coverA) circle (1pt)
(trivalent2coverB) circle (1pt);

\path[draw] (D1cover) [] -- ++(-45:.15);		
\path[draw] (D2cover) [] -- ++(-45:.15);		
\path[draw] (D3cover) [] -- ++(-45:.15);		
\path[draw] (D4cover) [] -- ++(-45:.15);		

\draw (-.3,1) node {$\hat \Gamma^+$};
\draw (-.3,0) node {$\Gamma^+$};
\end{tikzpicture}
		\end{center}
	\end{minipage}
	\caption{A realizable cover for the configuration in Figure~\ref{subfig:Dumbbell_4}.}
	\label{fig:realizable_dumbbell_graph_with_cover}
\end{figure}

	\appendix

\newcommand{\bg}{\mathbf{g}}
\newcommand{\bn}{\mathbf{n}}
\newcommand{\bd}{\mathbf{d}}
\newcommand{\bmu}{\boldsymbol{\mu}}
\newcommand\PMSDR[4][]{\bP \Xi^{#1}_\bd \overline \cM^{\frakR}_{{#2},{#3}}({#4})}

\section{Nonemptiness of boundary strata}

\label{app:generalized_strata}

While the boundary strata of the moduli space $\PMSD[k]{g}{n}{\mu}$ of multi-scale $k$-differentials of type $\mu$ are indexed by normalized covers of enhanced level graphs $\pi : \hat G^+ \to G^+$, not every such cover in fact corresponds to a nonempty boundary stratum $D_\pi$.
Theorem~\ref{thm:main_theorem} implicitly solves the problem to determine if a boundary stratum $D_\pi$ is in fact nonempty for the moduli space $\PMSD[k]{g}{n}{1, \dots, 1}$, that is: Given a cover of enhanced level graphs $\pi : \hat G^+ \to G^+$ where all legs of $G^+$ have $o$-value $1$, is there a corresponding normalized cover of twisted $k$-differentials $\hat X \to X$ in $\PMSD[k]{g}{n}{1, \dots, 1}$?
In fact our methods can be applied to all strata of $k$-differentials. We will make this explicit in the slightly more general setting of so-called \emph{generalized strata}.

To motivate the definition of generalized strata, fix a type $\mu$ and consider a boundary stratum $D_\pi$ of $\PMSD[k]{g}{n}{\mu}$ given by a cover of enhanced level graphs $\pi : \hat G^+ \to G^+$.
Now consider the space $B$ given by the family corresponding to the projection of $D_\pi$ to some level $L$ of $G^+$.
In general $B$ fails to be a honest stratum for two reasons: The level $G^+_{=L}$ may have several connected components, and the $k$-residues at the poles connecting $G^+_{=L}$ to higher levels may be restricted by the GRC.
In other words, $B$ is a subspace of a product of strata.
The definition of a generalized stratum models this kind of spaces.
Our definition is a generalization to the setting of $k$-differentials of the definition given in \cite{CMZ20_euler} for the abelian case.

Let $\Omega^k_d \cM_{g,n}(\mu)$ denote the connected components of $\Omega^k \cM_{g,n}(\mu)$ which parametrize $d$-th powers of primitive $k/d$-differentials.
Let
\[
	\Omega^k_\bd \cM_{\bg,\bn}(\bmu) = \prod_{i=1}^\kappa \Omega^k_{d_i} \cM_{g_i,n_i}(\mu_i)
\]
be a disconnected stratum and let 
\[
	\Omega \cM_{\hat \bg, \hat \bn}(\hat \bmu) = \prod_{j=1}^{\hat \kappa} \Omega \cM_{\hat g_j, \hat n_j}(\hat \mu_j)
\]
be the product of the strata that contain the canonical covers.
Here the bold letters on the left denote the tuples of the corresponding letters on the right, i.e.\ $\bd = (d_1, \dots, d_\kappa)$.
Let $\hat \mu_j = (\hat m_{j,1}, \dots, \hat m_{j,\hat n_j})$ and let $H_p := \{ (j,l) \mid \hat m_{j,l} < -1 \}$ be the set of marked non-simple poles in the cover.
Let $\lambda$ be a partition of $H_p$ with parts denoted by $\lambda^{(a)}$ and let $\lambda_\frakR$ be a subset of the parts of $\lambda$ such that $\lambda_\frakR$ is $\tau$-invariant as a set.
Let
\[
	\frakR := \left\{ r = (r_{j,l})_{(j,l) \in H_p} \in \bC^{|H_p|} \;\middle|\; \sum_{(j,l)\in \lambda^{(a)}} r_{j,l} = 0 \text{ for all } \lambda^{(a)} \in \lambda_\frakR \right\} \ .
\]
Denote by $\Omega \cM^\frakR_{\hat \bg, \hat \bn}(\hat \bmu)$ the subspace with residues in $\frakR$ and denote by $\Omega^k_\bd \cM^\frakR_{\bg,\bn}(\bmu)$ the corresponding subspace of $k$-differentials.
Note that this is well-defined, as we have chosen $\lambda_\frakR$ to be $\tau$-invariant.

\begin{mydef}
	We call a stratum of the form $\Omega^k_\bd \cM^\frakR_{\bg,\bn}(\bmu)$ a \emph{generalized stratum} and denote by $\PMSDR[k]{\bg}{\bn}{\bmu}$ the corresponding projectivized generalized stratum of multi-scale $k$-differentials.
\end{mydef}

Let $\hat G^+ \to G^+$ be a cover of enhanced level graphs.
Here and in the following we allow for $G^+$ to be disconnected.
Picking up an idea from \cite{CMZ20_diffstrata}, we construct the cover of \emph{auxiliary enhanced level graphs} $\hat G^+_\infty \to G^+_\infty$ in the following way.
For each $\lambda^{(a)} \in \lambda_\frakR$, we add a vertex $\hat v_{(a)}$ to $\hat G^+$ and think of these new vertices as being at level $\infty$, i.e.\ all new vertices are on a new level above all other levels.
As the legs of $\hat G^+$ correspond to the marked points of the stratum with orders $\hat m_{j,l}$, we may think of the legs of $\hat G^+$ as beeing indexed by the tuples $(j,l)$.
In particular the leg $(j,l)$ has $o$-value $\hat m_{j,l}$.
For each $(j,l) \in \lambda^{(a)}$, we add an edge to $\hat G^+$ connecting the leg $(j,l)$ to the new vertex $\hat v_{(a)}$ and we take the $o$-value at the upper half-edge of the new edge to be $-\hat m_{j,l} - 2k$.
If the sum of the $o$-values of the legs incident to $\hat v_{(a)}$ is odd we add an additional leg with $o$-value $1$ to $\hat v_{(a)}$.
Then the genus of $\hat v_{(a)}$ is determined by the $o$-values of the incident legs.
Finally, the action of $\tau$ on the new edges, vertices and legs is determined by the action of $\tau$ on the legs $(j,l)$.
We call this new graph $\hat G^+_\infty$ and we add edges, vertices and legs to $G^+$ to complete $G^+$ to the quotient $G^+_\infty := \hat G^+_\infty / \tau$.
We emphasis that the new vertices added to $G^+_\infty$ are never inconvenient, as they do not contain any poles.

We call the marked poles of $\Omega \cM_{\hat \bg, \hat \bn}(\hat \bmu)$ that are not contained in $\lambda_\frakR$ the \emph{free poles}.
When we assign residues to the graph $\hat G^+$ as in the proof of our main theorem, we may alter the residues at free poles at will (while maintaining the residue theorem at each component), as the GRC does not restrict the residues at components containing a free pole in any way, see Definition~\ref{def:GRC}.
We reflect this in the following definition.

\begin{mydef}
	A \emph{free pole path} is a simple path in $\hat G^+$ starting and ending in a (different) free pole.
	A \emph{generalized cycle} is a free pole path or a simple cycle.
\end{mydef}

The definitions of an effective (resp.\ admissible) cycle and of an independent pair of cycles can by adapted for generalized cycles in the obvious way.
By applying the methods of our proof to the enveloping stratum of the cover of auxiliary enhanced level graphs $\hat G^+_\infty \to G^+_\infty$, it is not hard to check that the proof of Theorem~\ref{intro:thm:main_theorem} in fact proves

\begin{mythm} \label{thm:realizability_generalized_strata}
	A normalized cover of enhanced level graphs $\pi \colon \hat G^+ \to G^+$ corresponds to a nonempty boundary stratum of a generalized stratum $\PMSDR[k]{\bg}{\bn}{\bmu}$ if and only if the following conditions hold.
	\begin{enumerate}[(i)]
		\item There is no illegal vertex in $\pi$.
		\item For every horizontal edge $\hat e$ in $\hat G^+$ there is an effective generalized cycle in $\hat G^+_\infty$ through $\hat e$.
		\item For every inconvenient vertex $v$ in $G^+$ there is an admissible generalized cycle in $\hat G^+_\infty$ through one of the preimages $\hat v$ or there is an independent pair of generalized cycles.
	\end{enumerate}
\end{mythm}

\begin{myrem}
	For $k=1$, Theorem~\ref{thm:realizability_generalized_strata} recovers \cite[Proposition~3.2]{CMZ20_diffstrata} in the same way as Theorem~\ref{intro:thm:main_theorem} recovered \cite[Theorem~6.3]{MUW17_published}, see Remark~\ref{rem:MUW_recovered}.
\end{myrem}

	\printbibliography

\end{document}